\documentclass[11pt]{article}

\usepackage{mathrsfs, textcomp}
\usepackage{easybmat,etex,tikz,caption}

\DeclareCaptionFormat{reverse}{#3#2#1}
\DeclareCaptionLabelFormat{fullparens}%
{(\bothIfFirst{#1}{˜}#2)}
\DeclareCaptionLabelSeparator{fill}{\hfill}
\usepackage{amsmath,amssymb,enumitem,amsthm,subfigure,mathabx}
\usepackage{pstricks,xy,pst-node,color, graphics, graphicx}
\xyoption{all}
\usepackage{stmaryrd}
\usepackage{color}

\usepackage{tabularx}
\makeatletter
\def\hlinewd#1{%
\noalign{\ifnum0=`}\fi\hrule \@height #1 %
\futurelet\reserved@a\@xhline}
\makeatother

\usetikzlibrary{matrix,arrows,shapes,snakes}
\usepackage{accents}
\numberwithin{equation}{section}
\usepackage{float}
\floatstyle{plaintop}
\newfloat{diagram}{thp}{lop}
\floatname{diagram}{Diagram}
\numberwithin{diagram}{section}
\makeatletter\let\c@equation\c@diagram\makeatother
\newtheorem{thm}{Theorem}[section] 
\newtheorem{prp}[thm]{Proposition}
\newtheorem{lmm}[thm]{Lemma}   
\newtheorem{crl}[thm]{Corollary}

\newtheorem{dfn}[thm]{Definition}

\theoremstyle{definition}
\newtheorem{rmk}[thm]{Remark}
\newtheorem{dfn-lmm}[thm]{Definition-Lemma}
\newtheorem{eg}[thm]{Example}

\newcommand{\isomto}{\overset{\sim}{\rightarrow}}
\def\eset{\emptyset} 

\newcounter{saveenumi}

\def\lan{\langle}
\def\ran{\rangle}
\def\lr#1{\lan#1\ran}
\def\blr#1{\big\lan#1\big\ran}
\def\bblr#1{\bigg\lan#1\bigg\ran}
\def\ov#1{\overline{#1}}
\def\ti#1{\tilde{#1}}
\def\wt#1{\widetilde{#1}}
\def\e_ref#1{(\ref{#1})}
\def\smsize#1{\begin{small}#1\end{small}}

\def\sf#1{\textsf{#1}}

\def\lrbr#1{\llbracket{#1}\rrbracket}
\def\det{\textnormal {det}}
\def\part{\partial}
\def\Lau#1{\llceil{#1}\rrceil}

\def\llbr{\llbracket}
\def\rrbr{\rrbracket}
\def\LR#1{\left\llbr{#1}\right\rrbr}

\def\wh{\widehat}
\def\lra{\longrightarrow}
\def\Lra{\Longrightarrow}
\def\Llra{\Longleftrightarrow}
\def\lra{\longrightarrow}
\def\Lra{\Longrightarrow}
\def\Llra{\Longleftrightarrow}

\def\al{\alpha}
\def\be{\beta}
\def\ga{\gamma}
\def\de{\delta}

\def\la{\lambda}
\def\bla{\boldsymbol{\la}}
\def\om{\omega}
\def\si{\sigma}

\def\ze{\zeta}
\def\wceta{\widecheck{\eta}}

\def\Ga{\Gamma}
\def\La{\Lambda}
\def\Om{\Omega}

\def\De{\Delta}
\def\Si{\Sigma}

\def\i{\infty}
\def\h{\hbar}
\def\b0{\mathbf{0}}

\def\cA{\mathcal A}
\def\nA{\textnormal A}
\def\fa{\mathfrak a}

\def\nB{\textnormal B}

\def\C{\mathbb C}
\def\cC{\mathcal C}
\def\ctC{\wt{\mathcal{C}}}
\def\nc{\mathrm{c}}

\def\fC{\mathfrak C}
\def\nC{\textnormal{C}}
\def\tinC{\wt{\nC}}
\def\ticC{\wt{\cC}}
\def\fd{\mathfrak d}
\def\bd{\mathbf d}

\def\fD{\mathfrak D}

\def\bD{\mathbf D}

\def\dec{\textnormal{dec}}
\def\cE{\mathcal E}
\def\E{\mathbf e}
\def\nE{\textnormal{E}}
\def\Ep{\accentset{\blacktriangledown}E}
\def\Epp{\accentset{\blacktriangledown\blacktriangledown}E}
\def\wbE{\boldsymbol{\wt{\mathbf{e}}}}
\def\F{\mathbb F}
\def\cF{\mathcal F}

\def\bH{\mathbf H}
\def\nH{\textnormal{H}}
\def\I{\mathfrak i}
\def\mi{\mathrm i}

\def\cJ{\mathcal J}
\def\cK{\mathcal K}
\def\cL{\mathcal L}

\def\M{\mathfrak M}

\def\cO{\mathcal O}
\def\P{\mathbb P}
\def\br{\mathbf r}
\def\bs{\mathbf s}

\def\bp{\mathbf p}
\def\Poly{\textnormal{Poly}}
\def\fp{\mathfrak p}

\def\R{\mathbb R}
\def\cR{\mathcal R}

\def\O{\mathcal O}
\def\Q{\mathbb Q}
\def\T{\mathbb T}
\def\cT{\mathcal T}
\def\fX{\mathfrak X}
\def\U{\mathfrak U}
\def\nU{\textnormal{U}}

\def\cY{\mathcal Y}
\def\hY{\widehat{Y}}
\def\hcY{\widehat{\cY}}
\def\Z{\mathbb Z}
\def\cZ{\mathcal Z}

\def\bM{\mathbf{M}}

\def\V{ \mathscr{V}^{\tau}_M}
\def\msV{\mathscr{V}}
\def\cV{\mathcal V}
\def\cVp{\accentset{\blacktriangledown}{\cV}}
\def\cVpp{\accentset{\blacktriangledown\blacktriangledown}{\cV}}
\def\Yp{\accentset{\blacktriangledown}{Y}}
\def\Ypp{\accentset{\blacktriangledown\blacktriangledown}{Y}}
\def\cYp{\accentset{\blacktriangledown}{\cY}}
\def\cYpp{\accentset{\blacktriangledown\blacktriangledown}{\cY}}

\def\Ip{\accentset{\blacktriangledown}{I}}
\def\Ipp{\accentset{\blacktriangledown\blacktriangledown}{I}}
\def\tinCp{\accentset{\blacktriangledown}{\wt{\nC}}}
\def\tinCpp{\accentset{\blacktriangledown\blacktriangledown}{\wt{\nC}}}
\def\ticCp{\accentset{\blacktriangledown}{\wt{\cC}}}
\def\ticCpp{\accentset{\blacktriangledown\blacktriangledown}{\wt{\cC}}}
\def\fCp{\accentset{\blacktriangledown}{\fC}}
\def\fCpp{\accentset{\blacktriangledown\blacktriangledown}{\fC}}
\def\nd{\textnormal{d}}
\def\ne{\textnormal{e}}
\def\std{\textnormal{std}}
\def\tr{\textnormal{tr}}
\def\bga{\boldsymbol\gamma}
\def\bs{\mathbf s}
\def\bx{\mathbf x}
\def\bu{\mathbf u}

\def\nv{\textnormal v}
\def\nw{\textnormal{w}}
\def\triv{\textnormal{triv}}
\def\pr{\textnormal{pr}}
\def\X{X^{\tau}_M}
\def\Id{\textnormal{Id}}
\def\Zp{\accentset{\blacktriangledown}{Z}}
\def\Zpp{\accentset{\blacktriangledown\blacktriangledown}{Z}}
\def\cZp{\accentset\blacktriangledown{\cZ}}
\def\cZpp{\accentset{\blacktriangledown\blacktriangledown}{\cZ}}

\def\Aut{\textnormal{Aut}}
\def\Edg{\textnormal{Edg}}

\def\End{\textnormal{End}}
\def\ev{\textnormal{ev}}

\def\mod{\textnormal{mod~}}
\def\Res{\textnormal{Res}}
\def\supp{\textnormal{supp}}

\def\Span{\textnormal{Span}}
\def\val{\textnormal{val}}
\def\Ver{\textnormal{Ver}}

\def\fgt{\textnormal{fgt}}
\def\PD{\textnormal{PD}}

\begin{document}

\title{Two-Point Gromov-Witten Formulas\\ for Symplectic Toric Manifolds}
\author{Alexandra Popa}
\date{\today}
\maketitle

\begin{abstract}
We show that the standard generating functions for genus~$0$ two-point twisted Gromov-Witten invariants
arising from concavex vector bundles over symplectic toric manifolds are explicit transforms of the corresponding one-point
generating functions. The latter are, in turn, transforms of Givental's $J$-function. We obtain closed formulas for them and, in particular, for two-point Gromov-Witten invariants
of non-negative toric complete intersections. Such two-point formulas should play a key role in the computation of genus~$1$
Gromov-Witten invariants (closed, open, and unoriented) of toric complete intersections as they indeed do in the case of the projective complete intersections.
\end{abstract}

\tableofcontents

\section{Introduction}\label{intro_sec}
Torus actions on moduli spaces of stable maps into a smooth projective variety
facilitate the computation of equivariant Gromov-Witten invariants \cite{Gi_equiv} 
via the Localization Theorem~\cite{ABo},~\cite{GraPa}.
Equivariant formulas lead to other interesting consequences beyond the computation of non-equivariant Gromov-Witten invariants. In the case of the projective spaces,
two-point equivariant Gromov-Witten formulas in \cite{bcov0_ci} lead to 
the confirmation of mirror symmetry predictions concerning open and unoriented genus~1 Gromov-Witten invariants in the same paper and to the computation of closed genus~1 Gromov-Witten invariants in \cite{bcov1_ci}.
In this paper we obtain equivariant formulas expressing the standard two-point closed genus~$0$ generating function
for certain twisted Gromov-Witten invariants of symplectic toric manifolds in terms
of the corresponding one-point generating functions. We also obtain explicit 
formulas for the latter. 
In particular, we show that the standard generating function for these two-point invariants 
is a fairly simple transform of the well-known Givental's $J$-function.
The formulas obtained in this paper compute, in particular, the twisted/un-twisted Gromov-Witten numbers \e_ref{tGW_e}/\e_ref{GW_e} below.

For a smooth projective variety $X$ and a class $A\!\in\!\!H_2(X;\Z)$, $\ov\M_{0,m}(X,A)$ denotes 
the moduli space of stable maps 
from genus $0$ curves with $m$ marked points into $X$ representing $A$. Let
$$\ev_i:\ov\M_{0,m}(X,A)\lra X$$
be the evaluation map at the $i$-th marked point; see~\cite[Chapter~24]{MirSym}.
All cohomology groups in this paper will be with rational coefficients unless
otherwise specified.
For each $i\!=\!1,2,\ldots,m$, let $\psi_i\in H^2(\ov\M_{0,m}(X,A))$
be the first Chern class of the universal cotangent line bundle for
the $i$-th marked point. Let 
$$\pi\!:\U\lra\ov\M_{0,m}(X,A)$$
be the universal curve and $\ev:\U\lra X$
the natural evaluation map; see \cite[Section~24.3]{MirSym}.

A holomorphic vector bundle $E\!\lra\!X$ is called \sf{concavex} if
$$E=E^+\oplus E^-,\qquad\hbox{with}\quad 
H^1\left(\P^1,f^*E^+\right)=0,~~~
 H^0\left(\P^1,f^*E^-\right)=0\quad\forall\,f:\P^1\lra X.$$
Such a vector bundle induces a vector orbi-bundle $\cV_E$ over $\ov\M_{0,m}(X,A)$:
\begin{equation}\label{cVdfn_e}
\cV_E\equiv\cV_{E^+}\oplus\cV_{E^-},\quad\textnormal{where}\quad\cV_{E^+}\equiv\pi_*\ev^*E^+,\quad\cV_{E^-}\equiv R^1\pi_*\ev^*E^-.
\end{equation}
Given a class $A\!\in\!H_2(X;\Z)$ and classes $\eta_1,\eta_2\!\in\!H^*(X)$,
the corresponding genus~0 twisted two-point Gromov-Witten (GW) invariants of $X$ are:
\begin{equation}\label{tGW_e}
\blr{\psi^{p_1}\eta_1,\psi^{p_2}\eta_2}^X_{A,E}
\equiv\int_{[\ov\M_{0,2}(X,A)]^{vir}}
\big(\psi_1^{p_1}\ev_1^*\eta_1\big)\big(\psi_2^{p_2}\ev_2^*\eta_2\big)e(\cV_E)\in\Q.
\end{equation}
In particular, if $E\!=\!E^+$, the twisted Gromov-Witten invariants \e_ref{tGW_e} are the genus~0
two-point Gromov-Witten invariants of a complete intersection $Y\!\equiv\!s^{-1}(0)\!\hookrightarrow\!X$
defined by a generic holomorphic section $s\!:\!X\!\lra\!E^+$:
\begin{equation}\label{GW_e}
\blr{\psi^{p_1}\eta_1,\psi^{p_2}\eta_2}^X_{A,E^+}=\blr{\psi^{p_1}\eta_1,\psi^{p_2}\eta_2}^Y_{A}\equiv
\blr{\psi^{p_1}\eta_1,\psi^{p_2}\eta_2}^Y_{A,0}\qquad\forall\,\eta_1,\eta_2\!\in\!H^*(Y);
\end{equation}
the first equality follows from \cite[Theorem~0.1.1, Remark~0.1.1]{El}.

The numbers \e_ref{tGW_e} have been computed in the $X\!=\!\P^{n-1}$ case 
under various assumptions on $E$ through various approaches.
The case when $E$ is a positive line bundle is solved in~\cite{BK} and~\cite{bcov0}
and extended to the case when $E$ is a sum of positive line bundles
in \cite{bcov0_ci}. The former led to the computation of the genus~1 Gromov-Witten invariants of
Calabi-Yau hypersurfaces in \cite{bcov1}, while the latter to the computation of the genus~1
Gromov-Witten invariants of Calabi-Yau complete intersections in \cite{bcov1_ci}. 
The case when $E$ is a concavex vector bundle has been solved in \cite{Ch}
in the setting of \cite{LLY1}.
More recently, genus~0 formulas with any number of $\psi$ classes have been obtained in \cite{multipt}.
In this paper we extend the approaches of \cite{bcov0} and \cite{bcov0_ci}
to the case when $X$ is an arbitrary compact symplectic toric manifold and $E$ is a sum of non-negative and negative line bundles.

I thank Aleksey Zinger for proposing 
the questions answered in this paper,
for his many suggestions which consistently improved it, for pointing out errors in previous versions, for explaining \cite{bcov0}
and parts of \cite{Gi_mirr} to me, and for his guidance and 
encouragement while I was working on this paper. 
I am also grateful to Melissa Liu for answering my many questions on toric manifolds,
for explaining parts of \cite{Gi_mirr} to me, and for bringing \cite{LLY3} to my attention.

\subsection{Some results}
\label{neqres_sec}

If $n$ is a non-negative integer, we write 
$$[n]\equiv\left\{1,2,\ldots,n\right\}.$$
Let $s\!\ge\!1$, $N_1,\ldots,N_s\!\ge\!2$ and for each $i\!\in\![s]$ let
$$\nH_i\equiv\pr^*_i\nH\in H^2\left(\prod\limits_{j=1}^s\P^{N_j-1}\!\right)\,,$$
where $\pr_i\!:\!\prod\limits_{j=1}^s\P^{N_j-1}\!\!\lra\!\P^{N_i-1}\!$ 
is the projection onto the $i$-th component and $\nH\!\in\!H^2(\P^{N_i-1})$
is the hyperplane class on $\P^{N_i-1}$.

\begin{thm}\label{GWproj_thm}
Let $\bd\!=\!(d_1,\!\ldots\!,d_s)\!\in\!(\Z^{>0})^s$.
The degree $\bd$ genus~$0$ two-point 
GW invariants \e_ref{GW_e} of $\prod\limits_{i=1}^s\P^{N_i-1}$
are given by the following identity in
$\frac{\Q[\nA_1,\ldots,\nA_s,\nB_1,\ldots,\nB_s]}
{\left(\nA_i^{N_i},\nB_i^{N_i}\quad\forall\,i\in[s]\right)}[[\h_1^{-1},\h_2^{-1}]]$:
\begin{equation*}\begin{split}
\sum\limits_
{\begin{subarray}{c}a_1,\ldots,a_s\ge 0\\
b_1,\ldots,b_s\ge 0\end{subarray}}
\nA_1^{a_1}\ldots\nA_s^{a_s}\nB_1^{b_1}\ldots\nB_s^{b_s}
\bblr{\frac{\nH_1^{N_1-1-a_1}\ldots\nH_s^{N_s-1-a_s}}{\h_1\!-\!\psi},
\frac{\nH_1^{N_1-1-b_1}\ldots\nH_s^{N_s-1-b_s}}{\h_2\!-\!\psi}}
^{\prod\limits_{i=1}^s\P^{N_i-1}}_{\bd}\!\\=
\frac{1}{\h_1\!+\!\h_2}
\sum_{\begin{subarray}{c} a_i,b_i,e_i,f_i\ge 0\\
a_i+b_i=N_i-1\\ e_i+f_i=d_i \end{subarray}}
\frac{\left(\nA_1\!+e_1\!\h_1\right)^{a_1}\ldots\left(\nA_s\!+\!e_s\h_1\right)^{a_s}
\left(\nB_1\!+f_1\!\h_2\right)^{b_1}\ldots\left(\nB_s\!+\!f_s\h_2\right)^{b_s}}
{\prod\limits_{i=1}^s\left(\prod\limits_{r=1}^{e_i}\left(\nA_i\!+\!r\h_1\right)^{N_i}\,
\prod\limits_{r=1}^{f_i}\left(\nB_i\!+\!r\h_2\right)^{N_i}\right)}.
\end{split}\end{equation*}
\end{thm}
This follows from Corollary~\ref{2pt_crl} in Section~\ref{stat_sec}.

The results below concern the GW invariants of a compact symplectic toric manifold 
$\X$ defined by \e_ref{toricman_e} from a minimal 
toric pair $(M,\tau)$ as in Definition~\ref{toric_dfn}.
We assume that the vector bundle $E$ splits 
\begin{equation}\label{splitE_e}
E\equiv E^+\oplus E^-\lra\X, \qquad\textnormal{where}\quad 
E^+\equiv\bigoplus_{i=1}^aL_i^+,
\quad E^-\equiv\bigoplus_{i=1}^bL_i^-,
\end{equation}
$L_i^+$ are non-trivial, non-negative line bundles and $L_i^-$ 
are negative line bundles.\footnote{\label{negative_f}Recall that a line bundle $L\!\lra\!\X$ is called positive (respectively negative)
if $c_1(L)\!\in\!H^2(\X;\R)$ (respectively $-c_1(L)$)
can be represented by a K\"{a}hler form on $\X$.
A line bundle $L\!\lra\!\X$ is called non-negative if
$c_1(L)\!\in\!H^2(\X;\R)$ can be represented by a closed real $(1,1)$-form $\om$
satisfying $\om(v,Jv)\!\ge\!0$ for all $v$.
The assumptions that the line bundles $L_i^+$ are 
non-trivial and that $L_i^-$ are negative (that is, $c_1(L_i^-)\!<\!0$ as opposed to just
$c_1(L_i^-)\!\le\!0$) are only used in the theorems that rely on the one-point mirror theorem 
\e_ref{Z'Y'_e} of \cite{LLY3}, that is Theorem~\ref{Y_thm}, Corollary~\ref{nuE0_crl},
Corollary~\ref{2pt_crl}, Theorem~\ref{cY_thm}, Corollary~\ref{enuE0_crl}, and Corollary~\ref{e2pt_crl}.}
Theorem~\ref{Z2pt_thm} and Remark~\ref{untwZ_rmk} below describe two-point twisted GW
invariants in terms of one-point ones.
As is usually done, the twisted GW invariants will be assembled 
into a generating function in the formal variables
$$Q=(Q_1,\ldots,Q_k)$$ with powers indexed by
\begin{equation}\label{La_e}
\La\equiv\big\{\bd\!\in\!H_2\left(X^\tau_M;\Z\right):\blr{\om,\bd}\!\ge\!0\quad\forall\,\om\!\in\!
\ov\cK^{\tau}_M\big\},
\end{equation}
where $\ov\cK^{\tau}_M$ is the closed K\"{a}hler cone 
of~$\X$.\footnote{\label{dual_f}By \cite[Theorem~4.5]{Br}, 
a non-empty closed convex subset of $\R^d$ is the intersection of its supporting half-spaces. 
The supporting half-spaces of a closed convex cone $C$ in $\R^d$ are all sets of the form $\{v\!\in\!\R^d:\blr{v,w}\!\ge\!0\}$
for some $w\!\in\!\R^d$ such that $\blr{v,w}\!\ge\!0$ for all $v\!\in\!C$.
This implies that
$$\om\!\in\!\ov\cK^{\tau}_M\qquad\Llra\qquad\blr{\om,\bd}\!\ge\!0~~\forall\,\bd\!\in\!\La.$$}

A ring $R$ and the monoid $\La$ induce an $R$-algebra denoted $R[[\La]]$:
to each $\bd$ we associate a formal variable denoted $Q^{\bd}$ and
set
\begin{equation*}\begin{split}
R[[\La]]\equiv\left\{\sum_{\bd\in\La}a_{\bd}Q^{\bd}:
a_{\bd}\!\in\!R~~\forall\,\bd\!\in\!\La\right\}.
\end{split}\end{equation*}
Addition in $R[[\La]]$ is defined naturally; multiplication is defined by 
$$Q^{\bd}\!\cdot\!Q^{\bd'}\equiv Q^{\bd+\bd'}\qquad\forall\,\bd,\bd'\!\in\!\La$$
and extending by $R$-linearity.

For each $m\!\ge\!1$ and each $\bd\!\in\!\La\!-\!\{0\}$, let 
$\si_i\!:\!\ov\M_{0,m}(\X,\bd)\!\lra\!\U$ be the section of the universal curve given by the $i$-th marked point, 
\begin{equation}\label{cV'dfn_e}\begin{split}
\cVp_E&\equiv R^0\pi_*\left(\ev^*E^+(-\si_1)\right)\oplus R^1\pi_*\left(\ev^*E^-(-\si_1)\right)\lra\ov\M_{0,m}(\X,\bd),\quad\textnormal{and}\\
\cVpp_E&\equiv R^0\pi_*\left(\ev^*E^+(-\si_2)\right)\oplus R^1\pi_*\left(\ev^*E^-(-\si_2)\right)\lra\ov\M_{0,m}(\X,\bd)\quad\textnormal{whenever}\quad m\!\ge\!2.
\end{split}\end{equation}
If $m\!\ge\!3$ and $\bd\!=\!0$, $\cVp_E$ and $\cVpp_E$ are well-defined as well and they are $0$.
We next define the genus~0 two-point generating function
$\Zp$:
\begin{equation}\label{Z2ptdfn_e}
\Zp\left(\h_1,\h_2,Q\right)\equiv
\frac{\h_1\h_2}{\h_1+\h_2}
\!\sum_{\bd\in\La}\!Q^{\bd}\left(\ev_1\!\times\!\ev_2\right)_*
\left[\frac{e(\cVp_E)}{\left(\h_1\!-\!\psi_1\right)\left(\h_2\!-\!\psi_2\right)}\right],
\end{equation}
where $\ev_1,\ev_2\!:\!\ov\M_{0,3}(\X,\bd)\!\lra\!\X$ are 
the evaluation maps at the first two marked points.
This is used - in the case of the projective spaces -
for the computation of the genus~1 GW invariants of Calabi-Yau complete intersections.

With $\ev_1,\ev_2:\!\ov\M_{0,2}(\X,\bd)\!\lra\!\X$ denoting the evaluation maps at 
the two marked points and for all $\eta\!\in\!H^2(\X)$, let
\begin{equation}\label{Zeta1ptdfn_e}\begin{split}
\Zp_{\eta}\left(\h,Q\right)&\equiv\eta\!+\!
\sum_{\bd\in\La-0}Q^{\bd}\ev_{1*}\left[\frac{e(\cVp_E)\ev_2^*\eta}
{\h\!-\!\psi_1}\right]\!\in\!H^*\!(\X)[\h^{-1}][[\La]],\\
\Zpp_{\eta}\left(\h,Q\right)&\equiv\eta\!+\!\sum_{\bd\in\La-0}
Q^{\bd}\ev_{1*}\left[\frac{e(\cVpp_E)\ev_2^*\eta}
{\h\!-\!\psi_1}\right]\in H^*(\X)[\h^{-1}][[\La]].
\end{split}\end{equation}

\begin{thm}\label{Z2pt_thm}
Let $\pr_i:\X\!\times\!\X\!\lra\!\X$  denote the projection onto the $i$-th component
and let $\eta_j,\wceta_j\!\in\!H^*(\X)$ be such~that
$$\sum_{j=1}^s\pr_1^*\eta_j\pr_2^*\wceta_j\!\in\!H^{2(N-k)}(\X\!\times\!\X)$$ 
is the Poincar\'{e} dual to the diagonal class, where $N\!-\!k$ is the complex dimension of~$\X$.
Then,
$$\Zp\left(\h_1,\h_2,Q\right)=\frac{1}{\h_1\!+\!\h_2}
\sum_{j=1}^s\pr_1^*\Zp_{\eta_j}\left(\h_1,Q\right)\pr_2^*\Zpp_{\wceta_j}\left(\h_2,Q\right).$$
\end{thm}
This follows from Theorem~\ref{eZ2pt_thm} below, which is an equivariant version of Theorem~\ref{Z2pt_thm}.
\begin{rmk}\label{untwZ_rmk}
The genus~0 two-point twisted GW invariants \e_ref{tGW_e} are assembled into 
\begin{equation}\label{untwZ_e}
Z^*\left(\h_1,\h_2,Q\right)\equiv
\!\sum_{\bd\in\La-0}\!Q^{\bd}\left(\ev_1\!\times\!\ev_2\right)_*
\left[\frac{e(\cV_E)}{\left(\h_1\!-\!\psi_1\right)\left(\h_2\!-\!\psi_2\right)}\right]\!\in\!H^*\!\left(\X\!\times\!\X\right)[\h_1^{-1},\h_2^{-1}][[\La]],
\end{equation}
where $\ev_1,\ev_2\!:\!\ov\M_{0,2}(\X,\bd)\!\lra\!\X$.
By the string relation \cite[Section~26.3]{MirSym},
$$Z^*\left(\h_1,\h_2,Q\right)=\frac{\h_1\h_2}{\h_1+\h_2}\sum\limits_{\bd\in\La-0}(\ev_1\!\times\!\ev_2)_*\left[\frac{e(\cV_E)}{(\h_1\!-\!\psi_1)(\h_2\!-\!\psi_2)}\right]\!\in\!H^*\!\left(\X\!\times\!\X\right)[\h_1^{-1},\h_2^{-1}][[\La]],$$
where $\ev_1,\ev_2\!:\!\ov\M_{0,3}(\X,\bd)\!\lra\!\X$. 
By \e_ref{cV'dfn_e} and \e_ref{cVdfn_e},
$$e(\cVp_E)\ev_1^*e(E^+)=e(\cV_E)\ev_1^*e(E^-).$$
The last  two equations imply that
$$\Zp^*(\h_1,\h_2,Q)\pr_1^*e(E^+)=Z^*(\h_1,\h_2,Q)\pr_1^*e(E^-),$$
where $\Zp^*$ is obtained from $\Zp$ by disregarding the $Q^{\b0}$ term and $\pr_1\!:\!\X\!\times\!\X\!\lra\!\X$ is the projection onto the first component.
This together with Theorem~\ref{Z2pt_thm} expresses $Z^*$
in terms of $\Zp_{\eta}$, $\Zpp_{\eta}$ in the $E\!=\!E^+$ case.
In all other cases, $Z^*$ can be expressed in terms of one-point GW generating functions
which can be computed under one additional assumption; see Remark~\ref{untwZ_rmk2}.
\end{rmk}

\begin{rmk}
If $E\!=\!\cO_{\P^2}(-1)\oplus\cO_{\P^2}(-2)$ and $\nH\!\in\!H^2(\P^2)$ is the hyperplane class,
then
$$\int_{\ov\M_{0,2}(\P^2,d)}e(\cV_E)\ev_1^*\nH^2\ev_2^*\nH=
\int_{\ov\M_{0,2}(\P^2,d)}e(\cV_E)\ev_1^*\nH\ev_2^*\nH^2=(-1)^d\frac{(2d)!}{2d(d!)^2}\qquad\forall\,d\!\ge\!1.
$$
If $E\!=\!\cO_{\P^2}(-1)\oplus\cO_{\P^2}(-1)\oplus\cO_{\P^2}(-1)$ and $\nH\!\in\!H^2(\P^2)$
is the hyperplane class,
then
$$\int_{\ov\M_{0,2}(\P^2,d)}e(\cV_E)\ev_1^*\nH^2\ev_2^*\nH^2=\frac{(-1)^{d+1}}{d}\qquad\forall\,d\!\ge\!1.$$
If $E\!=\!\cO_{\P^1}(-1)\oplus\cO_{\P^1}(-1)$ and $\nH\!\in\!H^2(\P^1)$ is the hyperplane class, then
$$\int_{\ov\M_{0,2}(\P^1,d)}e(\cV_E)\ev_1^*\nH\,\ev_2^*\nH=\frac{1}{d}\qquad\forall\,d\!\ge\!1.$$
These follow from \e_ref{localCY_e} in Section~\ref{stat_sec} which relies on Theorem~\ref{eZ2pt_thm},
the equivariant version of Theorem~\ref{Z2pt_thm} above.
The first of these equations implies the first statement in \cite[Proposition~2]{KlPa} by the divisor relation of
\cite[Section~26.3]{MirSym}, the second recovers the first statement in \cite[Lemma~3.1]{PaZ}, and
the third implies the Aspinwall-Morrison formula.
\end{rmk}
\subsection{Outline of the paper}\label{paperoutline_sec}

Section~\ref{overview_sec} presents the facts about symplectic toric manifolds needed for the
Gromov-Witten theory parts of the paper. This section is inspired by the view in \cite{Gi_mirr}
of a symplectic toric manifold as given by a matrix and the choice of a certain regular value
together with the holomorphic charts of \cite{Ba}. It contains proofs of all statements or references to the ones that are omitted. 
The reader interested only in the Gromov-Witten theory part may want to skip 
all proofs in Section~\ref{overview_sec}.

Section~\ref{stat_sec} gives formulas for the one-point GW generating functions $\Zp_{\eta}$, $\Zpp_{\eta}$ of \e_ref{Zeta1ptdfn_e} under an additional assumption
in terms of explicit formal power series constructed in Section~\ref{construction_sec}.
It begins with a short setup. 

The explicit GW formulas of Section~\ref{stat_sec} and Theorem~\ref{Z2pt_thm} above
follow from the equivariant statements of Section~\ref{estat_sec}.
In particular, equivariant versions of $\Zp_{\eta}$ and $\Zpp_{\eta}$ are expressed
in terms of explicit power series constructed in Section~\ref{econstruction_sec}.
Section~\ref{equiv_sec} also begins with a short setup.

An outline of the proofs of the equivariant theorems of Section~\ref{estat_sec} is given in Section~\ref{pfsoutline_sec}.
The remaining subsections of Section~\ref{pfs_sec} provide the details.

\section{Overview of symplectic toric manifolds}
\label{overview_sec}

This section reviews the basics of symplectic toric manifolds and sets up notation
that will be used throughout the rest of the paper. 
It combines the perspectives of \cite[Chapter~VII]{Au}, \cite[Section~11.3]{McDSa}, 
\cite[Section~2]{Ba}, \cite[Section~3.3.4]{CK}, \cite{Gi_mirr}, \cite{Gi_fpt}, and \cite[Sections~5,6]{Sp}.

Sections~\ref{nequiv-intro_sec}-\ref{nequivcoh_sec} give  
the definition and describe the basic properties of a compact symplectic toric manifold.
Section~\ref{equiv-intro_sec} is a preparation for localization computations in a toric setting;
it describes the fixed points and curves and the equivariant cohomology. 

\subsection{Definition, charts, and K\"{ahler} classes}
\label{nequiv-intro_sec}

Throughout this paper, $k$ and $N$ denote fixed positive integers such that $k\!\le\!N$ and
$$[N]\equiv\big\{1,2,\ldots,N\big\}.$$
If $v\!\in\!\R^k$ (or $v\!\in\!\C^N$) and $j\!\in\![k]$ (or $j\!\in\![N]$), let
$v_j\!\in\!\R$ (or $v_j\!\in\!\C$) denote the $j$-th component of~$v$
and define
$$\supp(v)\equiv\big\{j\!:\,v_j\!\neq\!0\big\}.$$
If $J\!\subseteq\![N]$, let
$$\R^J\equiv\big\{v\!\in\!\R^N\!\!:\supp(v)\!\subseteq\!J\big\}\!\cong\!\R^{|J|},\qquad\C^J \equiv\big\{z\!\in\!\C^N\!\!:\supp(z)\!\subseteq\!J\big\}\!\cong\!\C^{|J|}.$$
If $A\!=\!(a_{ij})_{i\in[k],j\in[N]}$ is a $k\!\times\!N$ matrix
and $J\!\subseteq\![N]$, denote by $A_J$ the $k\!\times\!|J|$ submatrix of~$A$ consisting
of the columns indexed by the elements of~$J$.
Let 
\begin{gather*}
\om_{\std}\equiv\frac{\mi}{2}\sum\limits_{j=1}^N\!\nd z_j\!\wedge\!\nd \overline{z}_j
\end{gather*}
be the standard symplectic form on $\C^N$. Let
\begin{gather*}
\mu_{\std}\!:\C^N\lra\R^N,~~~ 
\mu_{\std}(z_1,\ldots,z_N)\equiv\big(|z_1|^2,\ldots,|z_N|^2\big)
\end{gather*}
be the moment map for the restriction of the standard action of $\T^N\!\!\equiv\!(\C^*)^N$ 
on $(\C^N,-2\,\om_{\std})$,
$$(t_1,\ldots,t_N)\!\cdot\!(z_1,\ldots,z_N)=(t_1z_1,\ldots,t_Nz_N),$$
 to $(S^1)^N\!\!\subset\!\T^N$.
 
An integer $k\!\times\!N$ matrix $M\!=\!(m_{ij})_{i\in[k],j\in[N]}$ induces an action of
$\T^k\!\equiv\!(\C^*)^k$ on $(\C^N,-2\om_{\std})$,
\begin{equation}\label{Maction_e}
(t_1,\ldots,t_k)\!\cdot\!(z_1,\ldots,z_N)=
(t_1^{m_{11}}t_2^{m_{21}}\ldots t_k^{m_{k1}}z_1,\ldots,t_1^{m_{1N}}t_2^{m_{2N}}\ldots t_k^{m_{kN}}z_N);
\end{equation}
the moment map of its restriction to $(S^1)^k\!\subset\!\T^k$ is
$$\mu_M\equiv M\circ\mu_{\std}\!:\C^N\lra\R^k.$$
If in addition $\tau\!\in\!\R^k$, let 
\begin{equation}\label{toricman_e}\begin{split}
P^{\tau}_M&\equiv M^{-1}(\tau)\cap(\R^{\ge0})^N,\\
\wt{X}_M^{\tau}\equiv\C^N-\!\!\!\!\bigcup_{\begin{subarray}{c}J\subseteq[N]\\ \C^J\cap\mu_M^{-1}(\tau)=\eset\end{subarray}}\!\!\!\!\!\!\!\!\C^J\,
&=\big\{z\!\in\!\C^N\!:\,\C^{\supp(z)}\!\cap\!\mu_M^{-1}(\tau)\!\neq\!\eset\big\},\qquad
X_M^{\tau}\equiv\wt{X}_M^{\tau}\big/\T^k\,;
\end{split}\end{equation}
see diagram~\e_ref{Mtau_diag}.
By Proposition~\ref{toric_prp} below, $X_M^{\tau}$ is a compact projective manifold 
if the pair $(M,\tau)$ is toric
in the sense of Definition~\ref{toric_dfn}.
In this case,
$\mu_{\std}^{-1}(P^{\tau}_M)/(S^1)^k$ has a unique smooth structure making the projection 
$$\mu_{\std}^{-1}(P^{\tau}_M)\lra\mu_{\std}^{-1}(P^{\tau}_M)/(S^1)^k$$
a submersion. With this smooth structure,
$\mu_{\std}^{-1}(P^{\tau}_M)/(S^1)^k$
is diffeomorphic to
$X_M^{\tau}$
via a diffeomorphism induced by the inclusion $\mu_{\std}^{-1}(P^{\tau}_M)\!\hookrightarrow\!\wt{X}_M^{\tau}$.
We summarize this setup in a diagram:
\begin{equation}\begin{split}\label{Mtau_diag}
\xymatrix{&&&  P^{\tau}_M\!\equiv\! M^{-1}(\tau)\!\cap\!\left(\R^{\ge 0}\right)^N\ar@{^{(}->}[d]\\
\mu_M^{-1}(\tau)\!\equiv\!\mu_{\std}^{-1}(P^{\tau}_M)
\ar@{->>}_{\textnormal{projection}}[d] \ar@{^{(}->}[r] &
\wt{X}_M^{\tau} \ar@{->>}^{\textnormal{projection}}[d] \ar@{^{(}->}[r]
& \C^N\ar[r]^{\mu_{\std}}\ar[drr] _{\mu_M}
& \left(\R^{\ge 0}\right)^N\ar@{^{(}->}[r] & \ar[d]^{M}\R^N\\ 
  \frac{\mu^{-1}_{\std}(P^{\tau}_M)}{(S^1)^k}
          \ar[r]^{\textnormal{\small{diffeo}}} &
       X^{\tau}_M&&& \R^k&\hspace{-.65in}\ni\!\tau   }
\end{split}\end{equation}


Given a pair $(M,\tau)$ consisting of an integer $k\!\times\!N$ matrix $M$ and
a vector $\tau\!\in\!\R^k$, we define
\begin{equation}\begin{split}\label{Vset_e}
\V&\equiv\Big\{J\!\subseteq\![N]:|J|\!=\!k,P^{\tau}_M\!\cap\!\R^J\!\neq\!\emptyset\Big\}\\
&\equiv
\Big\{J\!\subseteq\![N]:|J|\!=\!k,\exists\,  v\!\in\!M^{-1}(\tau)\!\cap\!(\R^{\ge 0})^N
~~\hbox{s.t.}~~\supp(v)\!\subseteq\!J\Big\}.
\end{split}\end{equation}

\begin{dfn}\label{toric_dfn} 
A pair $(M,\tau)$ consisting of an integer $k\!\times\!N$ matrix $M$ 
and a vector $\tau\!\in\!\R^k$ is \sf{toric}~if
\begin{enumerate}[label=(\roman*),leftmargin=*] 
\item\label{O} $\tau$ is a regular value of $\mu_M$ and $P^{\tau}_M\!\neq\!\eset$;
\item\label{S} $\det\,M_J\!\in\!\{\pm1\}$ for all $J\!\in\!\V$;
\item\label{C} $P^0_M\!=\!\{0\}$ ($\Llra P^{\tau}_M$ is bounded).
\setcounter{saveenumi}{\arabic{enumi}}
\end{enumerate}
A toric pair $(M,\tau)$ is \sf{minimal} if
\begin{enumerate}[label=(\roman*),leftmargin=*]
\setcounter{enumi}{\arabic{saveenumi}}
\item\label{M} $P_M^{\tau}\!\cap\!\R^{[N]-\{j\}}\!\neq\!\eset$ for all $j\!\in\![N]$.
\end{enumerate}
\end{dfn}

If a pair $(M,\tau)$ satisfies \ref{S} in Definition~\ref{toric_dfn} above, then
$$z\!\in\!\C^N,\,\supp(z)\!\supseteq\!J\quad\textnormal{for some}\quad J\!\in\!\V\quad\Lra\quad
\exists\,t\!\in\!\T^k\quad\textnormal{such that}\quad (t\!\cdot\!z)_j\!=\!1\qquad\forall\,j\!\in\!J.$$
If $(M,\tau)$ is a toric pair, then a point $z\!\in\!\C^N$ lies in $\wt{X}^{\tau}_M$ 
if and only if $\supp(z)\!\supseteq\!J$ for some $J\!\in\!\V$ and
the $\T^N$-fixed points of $\X$ are indexed by $\V$;
see Lemma~\ref{moment_lmm}\ref{charts_lmm2} and Corollary~\ref{fixed_crl}\ref{fpt_crl}. 

\begin{prp}\label{toric_prp}
If $(M,\tau)$ is a toric pair, then $X^{\tau}_M$ is
a connected compact projective manifold of complex dimension $N\!-\!k$ endowed with
a $\T^N$\!-action induced from the standard action of $\T^N$ on~$\C^N$.
\end{prp}

\begin{proof}[Proof of Proposition~\ref{toric_prp}]
By Lemmas~\ref{top_lmm}\ref{connected_lmm}, \ref{compact_lmm}, and \ref{manifold_lmm} below,
$\X$ is a connected, compact complex manifold.
It admits a positive line bundle by 
Lemmas~\ref{positive_lmm}, \ref{reg_locus_lmm}\ref{cone_lmm}, and \ref{Kahler_lmm} below.
By the Kodaira Embedding Theorem \cite[p181]{GH}, $X^{\tau}_M$ is then projective.
\end{proof}

\begin{rmk}\label{toric_rmk}
If $X$ is a compact symplectic toric manifold in the sense of \cite[Definition~I.1.15]{Ca},
then the image of its moment map is a Delzant polytope $P$ (a polytope with certain properties \cite[Definition~I.2.1]{Ca});
see \cite[Theorem~1]{At} or \cite[Theorem~5.2]{GS1}.
This polytope~$P$ determines a fan~$\Si_P$, which in turn determines a compact complex manifold
$X_{\Si_P}$; see \cite[Section~VII.1.ac]{Au}. 
This complex manifold $X_{\Si_P}$ is endowed with a symplectic form, a torus action, and a moment map
with image $P$
making it into a symplectic toric manifold; 
see \cite[Theorem~VII.2.1]{Au}. 
Moreover, this symplectic form is K\"{a}hler with respect to the complex structure,
as stated in \cite[Section~3]{Gi_mirr} and can be deduced from \cite[Chapter~VII]{Au}. 
Since $X$ and $X_{\Si_P}$ have the same moment polytope (i.e.~image of the moment map),
they are isomorphic as symplectic toric manifolds by Delzant's uniqueness theorem \cite[Theorem~2.1]{De}.
On the other hand, $X_{\Si_P}\!=\!X_M^{\tau}$ for some minimal toric pair $(M,\tau)$
by the proof of \cite[Theorem~VII.2.1]{Au}. 
Thus, a compact symplectic toric manifold $(X^{2n},\om,(S^1)^n,\mu)$ in the sense of \cite[Definition~I.1.15]{Ca}
admits a complex structure~$\cJ$ so that $(X,\om,\cJ)$ is K\"{a}hler and $(X,\cJ)$ is isomorphic to $X_M^{\tau}$ for 
some minimal toric pair~$(M,\tau)$.
\end{rmk}

Lemma~\ref{top_lmm} relies on parts~\ref{charts_lmm2}  and~\ref{IJchan_lmm}
of Lemma~\ref{moment_lmm} below 
which in turn rely on the other parts of Lemma~\ref{moment_lmm}.
Lemma~\ref{Kahler_lmm} is based on Lemma~\ref{levels_lmm} and Lemma~\ref{reg_locus_lmm}\ref{indep_lmm}.
Lemma~\ref{reg_locus_lmm}\ref{cone_lmm} follows from Lemma~\ref{reg_locus_lmm}\ref{reg_lmm}, while the proof of Lemma~\ref{reg_locus_lmm}\ref{indep_lmm} uses Lemma~\ref{reg_locus_lmm}\ref{Vindep_lmm}. 

For $t\!=\!(t_1,t_2,\ldots,t_k)\!\in\!\T^k$ and $\bp\!=\!(p_1,p_2,\ldots,p_k)\!\in\!\Z^k$, let
$$t^{\bp}\equiv t_1^{p_1}t_2^{p_2}\ldots t_k^{p_k}.$$ 

\begin{lmm}\label{moment_lmm}
Let $(M,\tau)$ be a toric pair.
\begin{enumerate}[label=(\emph{\alph*}),leftmargin=*]
\item\label{polytope_lmm}
The subset $P^{\tau}_M\!\subset\!(\R^{\ge0})^N$ is a 
polytope (i.e.~the convex hull of a finite set of points).

\item\label{Mrestr_lmm} Let $\eta\!\in\!\R^k$ be any regular value of $\mu_M$. If
$w\!\in\!P^{\eta}_M$, then $$M\!:\{v\!\in\!\R^N\!:\,\supp(v)\subseteq\supp(w)\}\lra\R^k$$ is onto. In particular, if $w\!\in\!P^{\eta}_M$, then $|\supp(w)|\!\ge\!k$.

\item\label{V_lmm} If $J\!\in\!\V$, then $J\!=\!\supp(y)$ for some $y\!\in\!\mu_M^{-1}(\tau)$.

\item\label{real_lmm}If $J\!\subseteq\![N]$ and $\supp(v)\!\subseteq\!J$ for some $v\!\in\!P^{\tau}_M$,
then $\supp(w)\!=\!J$ for some $w\!\in\!P^{\tau}_M$.

\item\label{Pdim_lmm} The polytope $P^{\tau}_M$ has dimension $N\!-\!k$.

\item\label{vertex_lmm} If $v$ is a vertex of $P^{\tau}_M$, then $\supp(v)\!\in\!\V$.

\item\label{vertcoresp_lmm}
If $\textnormal{Vertices}^{\tau}_M$ is the set of vertices of the polytope $P^{\tau}_M$, the map
$$\supp:\textnormal{Vertices}^{\tau}_M\lra\V,\quad v\lra\supp(v),$$
is a bijection.

\item\label{charts_lmm} If $y\!\in\!\mu_M^{-1}(\tau)$, then $\supp(y)\!\supseteq\!J$
for some $J\!\in\!\V$.

\item\label{charts_lmm2} Let $z\!\in\!\C^N$. Then,
$z\!\in\!\wt{X}^{\tau}_M$ if and only if $\supp(z)\!\supseteq\!J$ for some $J\!\in\!\V$.

\item\label{IJchan_lmm} 
Let $I,J\!\in\!\V$ and $t^{(n)}\!\in\!\T^k$. If $|t^{(n)}|\!\lra\!\i$
and there exists $\de\!>\!0$ such that $|t^{(n)}_i|\!\ge\!\de$ for all $i\!\in\![k]$,
then $|(t^{(n)})^{M_I^{-1}M_j}|$ is unbounded for some $j\!\in\!J$.
\end{enumerate}
\end{lmm}

\begin{proof}
\ref{polytope_lmm} By \cite[Theorem~1.1]{Zi}, a subset of $\R^N$ is a polytope 
if and only if it is a bounded intersection of half-spaces.
Thus, the claim follows from~\ref{C} in Definition~\ref{toric_dfn}.\\
\ref{Mrestr_lmm} This is immediate from the surjectivity of $\nd_w\mu_M$.\\
\ref{V_lmm} This follows from the second statement in \ref{Mrestr_lmm}.\\
\ref{real_lmm}  Assume that $\supp(v)\!\subseteq\!I\!\subsetneqq\!J$ 
and that there exists $v'\!\in\!P^{\tau}_M$ with $\supp(v')\!=\!I$.
Let $I_1\!\supset\!I$ with $I_1\!\subseteq\!J$ and $|I_1|\!=\!|I|\!+\!1$.
We show that there exists $w\!\in\!P^{\tau}_M$ with $\supp(w)\!=\!I_1$.
By the first statement in~\ref{Mrestr_lmm},
there exists $w'\!\in\!M^{-1}(\tau)\!\subset\!\R^N$ with $\supp(w')\!=\!I_1$.
Let $w\!=\!(1\!-\!\la)v'\!+\!\la w'$ with $\la\!\in\!\R$ satisfying
$$\la w_j'\!>\!0\quad\hbox{if}\quad j\in I_1\!-\!I\quad\hbox{and}\quad
\la\left(1\!-\!\frac{w_j'}{v_j'}\right)\!<\!1\quad\forall\,j\!\in\!I.$$\\
\ref{Pdim_lmm} By \ref{real_lmm} together with the second condition in~\ref{O} in Definition~\ref{toric_dfn}, $\supp(w)\!=\![N]$ for some $w\!\in\!P^{\tau}_M$
and thus $\dim P^{\tau}_M\!=\!N\!-\!k$, since~$M$ has rank~$k$ by \ref{Mrestr_lmm}.\\ 
\ref{vertex_lmm} By \ref{Pdim_lmm}, $|\supp(v)|\!\le\!k$; 
the opposite inequality follows from the second statement in~\ref{Mrestr_lmm}.\\
\ref{vertcoresp_lmm}
By \ref{vertex_lmm}, $\supp(v)\!\in\!\V$ for every vertex $v$ of $P^{\tau}_M$ .
The map $\supp$ is injective by \ref{S} in Definition~\ref{toric_dfn} and surjective
by \ref{V_lmm} and \ref{S} in Definition~\ref{toric_dfn}.\\
\ref{charts_lmm} By \cite[Proposition~2.2]{Zi}, every polytope is the convex hull of its vertices; 
since $\mu_{\std}(y)\!\in\!P^{\tau}_M$ and $P^{\tau}_M$ is a polytope by \ref{polytope_lmm},
$$\mu_{\std}(y)=\sum_{s=1}^r\la_sv_s$$
for some vertices $v_1,v_2,\ldots,v_r\!\in\!P^{\tau}_M$ and $\la_1,\la_2,\ldots,\la_r\!\in\!\R^{>0}$.
Then, $\supp(y)\!\supseteq\!\supp(v_1)$
and $\supp(v_1)\!\in\!\V$ by \ref{vertex_lmm}.\\
\ref{charts_lmm2} If $z\!\in\!\wt{X}^{\tau}_M$,
there exists $y\!\in\!\C^{\supp(z)}\!\cap\!\mu_M^{-1}(\tau)$.
By~\ref{charts_lmm}, there exists $J\!\in\!\V$ with
$J\!\subseteq\!\supp(y)$. Since $\supp(y)\!\subseteq\!\supp(z)$,
it follows that $J\!\subseteq\!\supp(z)$. The converse follows from~\ref{V_lmm}.\\
\ref{IJchan_lmm} 
By~\ref{V_lmm}, there exist $v,w\!\in\!(\R^{>0})^k$ such that $M_Iv\!=\!\tau\!=\!M_Jw$.
By \ref{S} in Definition~\ref{toric_dfn}, it follows that there exists $\fa\!\in\!(\Z^{>0})^k$ such that $M_I^{-1}M_J\fa\!\in\!(\Z^{>0})^k$.

Assume by contradiction that $|(t^{(n)})^{M_I^{-1}M_j}|$ is a bounded sequence for all $j\!\in\!J$.
By passing to subsequences, we may assume that $|(t^{(n)})^{M_I^{-1}M_j}|$ is convergent for all $j\!\in\!J$.
It follows that
\begin{equation}\label{magic_e}
\prod\limits_{j\in J}\big|(t^{(n)})^{M_I^{-1}M_j}\big|^{\fa_j}=\big|(t^{(n)})^{M_I^{-1}M_J\fa}\big|
\end{equation}
is also convergent. On the other hand, by passing to some subsequences, we may assume that
for each $i\!\in\![k]$, $|t^{(n)}_i|$ has a limit (possibly $\i$). Since at least one of these limits is $\i$
and none is $0$, the right-hand side of \e_ref{magic_e} diverges leading to a contradiction.
\end{proof}

For $z\!\in\!\C^N$ and
$J\!=\!\{j_1\!<\!j_2\!<\ldots\!<j_n\}\!\subseteq\![N]$, let
$$z_J\equiv(z_{j_1},z_{j_2},\ldots,z_{j_n}).$$
For $z\!\in\!\wt{X}^{\tau}_M$, let $[z]\!\in\!\X$ denote the corresponding class.
 
\begin{lmm}\label{top_lmm}
Let $(M,\tau)$ be a toric pair.
\begin{enumerate}[label=(\emph{\alph*}),leftmargin=*]
\item\label{connected_lmm} The space $\wt{X}^{\tau}_M$ is path-connected.
\item\label{free_lmm} The torus $\T^k$ acts freely on $\wt{X}^{\tau}_M$.
\item\label{sat_lmm} The subset $\T^k\!\cdot\!\mu_M^{-1}(\tau)$ of $\C^N$ is open.
\item\label{closed_lmm} The subset $\T^k\!\cdot\!\mu_M^{-1}(\tau)$ of 
$\wt{X}^{\tau}_M$ is closed.

\item\label{levelgen_lmm} There is a unique map
$$\rho_M^{\tau}\!: \ti{X}_M^{\tau}\lra(\R^{>0})^k\subset\T^k
\qquad\textnormal{s.t.}\quad
\rho_M^{\tau}(z)\!\cdot\!z\in\mu_M^{-1}(\tau)~~\forall\,z\!\in\!\wt{X}^{\tau}_M\,.$$
Furthermore, this map is smooth.

\item\label{muMquot_lmm} The quotient $\mu_M^{-1}(\tau)/(S^1)^k$ is a compact 
and Hausdorff.

\item\label{compact_lmm}
The inclusion $\mu_M^{-1}(\tau)\!\hookrightarrow\!\wt{X}^{\tau}_M$
induces a homeomorphism
\begin{equation}\label{re-co_e}
\mu_M^{-1}(\tau)/(S^1)^k\lra X^{\tau}_{M}.
\end{equation}
In particular, $\X$ is compact and Hausdorff.

\item\label{manifold_lmm}
The space $\X$ is a complex manifold of complex dimension $N\!-\!k$.
\end{enumerate}
\end{lmm}

\begin{proof}
\ref{connected_lmm}
This holds since $\wt{X}^{\tau}_M$ is the complement of coordinate subspaces in $\C^N$.\\
\ref{free_lmm}
Let $t\!\in\!\T^k$ and $z\!\in\!\wt{X}^{\tau}_M$ be such that $t\!\cdot\!z\!=\!z$.
By Lemma~\ref{moment_lmm}\ref{charts_lmm2}, there exists $J\!\in\!\V$ such~that
$$J\!\equiv\{j_1\!<\!\ldots\!<\!j_k\}\subseteq\!\supp(z).$$
By~\ref{S} in Definition~\ref{toric_dfn}, the group homomorphism 
$$\T^k\lra\T^{k},\quad t\lra(t^{M_{j_1}},\ldots,t^{M_{j_k}}),$$
is injective and so $t\!=\!(1,1,\ldots,1)$.\\
\ref{sat_lmm}
For each $z\!\in\!\C^N$, let 
$$M_z\equiv M\left(\begin{array}{ccc}|z_1|&&0\\ &\ddots&\\ 0&&|z_N|\end{array}\right).$$
If $z\!\in\!\mu_M^{-1}(\tau)$, $\supp(z)\!\supseteq\!J$ for some $J\!\in\!\V$
by Lemma~\ref{moment_lmm}\ref{charts_lmm}.
Since $M_J$ is invertible by~\ref{S} in Definition~\ref{toric_dfn}, so are $(M_z)_J$ and 
$M_z(M_z)^{\tr}$.
Since the differential of the map
$$\left(\R^{>0}\right)^k\lra\R^k, \qquad t\lra\mu_M(t\!\cdot\!z),$$
at $t\!=\!(1,\ldots,1)\!\in\!(\R^{>0})^k\subset\T^k$ is $2M_z(M_z)^{\tr}$, the differential of
the map
$$\T^k\times\mu_M^{-1}(\tau)\lra\R^k, \qquad (t,z)\lra \mu_M(t\!\cdot\!z),$$
is surjective at $(1,z)$ for all $z\!\in\!\mu_M^{-1}(\tau)$.
Since the restriction of this differential to the second component vanishes, 
the differential of the map 
\begin{equation}\label{multiplication_e}
\T^k\times \mu_M^{-1}(\tau)\lra\C^N, \qquad (t,z)\lra t\!\cdot\!z,
\end{equation}
is surjective at $(1,z)$ for all $z\!\in\!\mu_M^{-1}(\tau)$ and so,
by the Inverse Function Theorem, the image of \e_ref{multiplication_e}
contains an open neighborhood of $\mu_M^{-1}(\tau)$ in $\C^N$.\\
\ref{closed_lmm} Let $z^{(n)}\!\in\!\wt{X}^{\tau}_M$  and $t^{(n)}\!\in\!\T^k$
be sequences such~that 
$$\lim_{n\lra\i}z^{(n)}=z\!\in\!\wt{X}^{\tau}_M \qquad\hbox{and}\qquad
y^{(n)}\equiv t^{(n)}\!\cdot\!z^{(n)}\in \mu_M^{-1}(\tau).$$
By~\ref{C} in Definition~\ref{toric_dfn}, we can assume that 
$y^{(n)}\!\lra\!y\!\in\!\mu_M^{-1}(\tau)$.
By Lemma~\ref{moment_lmm}\ref{charts_lmm2}, there exist $J(y),J(z)\!\in\!\V$ such~that
$$J(y)\equiv\{j_1\!<\!\ldots\!<\!j_k\big\}\subseteq\supp(y)
\qquad\hbox{and}\qquad J(z)\subseteq\supp(z);$$ 
we can assume that $J(y),J(z)\!\subseteq\!\supp(y^{(n)})\!=\!\supp(z^{(n)})$ for all~$n$.
By \ref{S} in Definition~\ref{toric_dfn}, $M_{J(y)}$ is invertible and so
$$t^{(n)}_i=\big(\ti{t}^{(n)}\big)^{\left(M_{J(y)}^{-1}\right)_i}\,,
\qquad\hbox{where}\quad
\big(\ti{t}^{(n)}\big)_i=\frac{(y^{(n)})_{j_i}}{(z^{(n)})_{j_i}}
\qquad\forall\,i=1,\ldots,k.$$
Since $(y^{(n)})_j\!\lra\!y_j\!\neq\!0$ for all $j\!\in\!J(y)$ and $(z^{(n)})_j\!\lra\!z_j$, 
$|(\ti{t}^{(n)})_i|\!\ge\!\de$ for some $\de\!\in\!\R^{>0}$ and for all~$n$ and $i$.
If $|(\ti{t}^{(n)})|$ is not bounded above,
after passing to a subsequence we can assume that $|\ti{t}^{(n)}|\!\lra\!\i$.
By Lemma~\ref{moment_lmm}\ref{IJchan_lmm}, there exists $j\!\in\!J(z)$ such that, after passing to a subsequence, 
$$\big|(t^{(n)})^{M_j}\big|=\big|(\ti{t}^{(n)})^{M_{J(y)}^{-1}M_j}\big|\lra\i.$$ 
Since $t^{(n)}\!\cdot\!z^{(n)}\!\lra\!y$, it follows that 
$(z^{(n)})_j\!\lra\!0$ and so $j\!\not\in\!\supp(z)$,
contrary to the assumption.
Thus, $\{\ti{t}^{(n)}\}$ is a compact subset of~$\T^k$.
After passing to a subsequence, we can thus assume that $t^{(n)}\!\lra\!t\!\in\!\T^k$.
It follows that 
$$t\cdot z=\lim_{n\lra\i}t^{(n)}\cdot\lim_{n\lra\i}z^{(n)}
=\lim_{n\lra\i}t^{(n)}\!\cdot\!z^{(n)}=\lim_{n\lra\i}y^{(n)}=y.$$
Thus, $z\!\in\!\T^k\!\cdot\!\mu_M^{-1}(\tau)$.\\
\ref{levelgen_lmm} By the proof of~\ref{sat_lmm}, $\tau$ is a regular value of the smooth map
$$\Phi\!: \big(\R^{>0}\big)^k\times\wt{X}_M^{\tau}\lra\R^k, \qquad
(t,z)\lra\mu_M(t\!\cdot\!z),$$
and the projection map $\pi_2\!:\!\Phi^{-1}(\tau)\!\lra\!\wt{X}_M^{\tau}$ is a submersion.
By~\ref{connected_lmm}, \ref{sat_lmm}, and~\ref{closed_lmm}, this map is surjective.
We show that it is also injective; by~\ref{connected_lmm}, \ref{sat_lmm}, and~\ref{closed_lmm},
this is equivalent to showing that
$$(r_1,\ldots,r_k)\!\in\!\R^k,~z,(\ne^{r_1},\ldots,\ne^{r_k})\!\cdot\!z\in\mu_M^{-1}(\tau)
\qquad\Lra\qquad r_i=0\qquad\forall\,i=1,\ldots,k,$$
where the action of $(\ne^{r_1},\ldots,\ne^{r_k})\!\in\!\T^k$ on $z$ is defined by \e_ref{Maction_e} as above.
We present the argument in the proof of \cite[7.2~Lemma]{Ki}.
Let 
$$f:\R\lra\R,\quad f(u)\equiv\Big\langle\mu_M\left[\left(\ne^{ur_1},\ldots,\ne^{ur_k}\right)\!\cdot\!z\right],(r_1,\ldots,r_k)\Big\rangle\qquad\forall\,u\!\in\!\R.$$
Since $f(0)\!=\!f(1)$, there exists $u_0\!\in\!(0,1)$ such that $f'(u_0)\!=\!0$.
Since
$$f'(u_0)\!=\!2\sum\limits_{j=1}^N\ne^{2u_0\lr{(r_1,\ldots,r_k),M_j}}\Big\langle(r_1,\ldots,r_k),M_j\Big\rangle^2|z_j|^2,$$
$f'(u_0)\!=\!0$ implies that $\lr{(r_1,\ldots,r_k),M_j}z_j\!=\!0$ for all $j\!\in\![N]$.
By Lemma~\ref{moment_lmm}\ref{charts_lmm2}, there exists $J\!\in\!\V$ such that $J\!\subseteq\!\supp(z)$
and so $\lr{(r_1,\ldots,r_k),M_j}\!=\!0$ for all $j\!\in\!J$. By \ref{S} in Definition~\ref{toric_dfn}, this implies that
$r_i\!=\!0$ for all $i\!\in\![k]$. The map $\rho^{\tau}_M$ is $\pi_2^{-1}$ composed with the projection $(\R^{>0})^k\!\times\!\X\!\lra\!(\R^{>0})^k$.\\
\ref{muMquot_lmm}  Since $\mu_M^{-1}(\tau)$ is compact by~\ref{C} in Definition~\ref{toric_dfn},
so is the quotient space $\mu_M^{-1}(\tau)/(S^1)^k$.
If $p$ is the quotient projection map and $A\!\subset\!\mu_M^{-1}(\tau)$ is a closed subset,
$$p^{-1}\big(p(A)\big)=(S^1)^k\cdot A
\equiv\big\{t\!\cdot\!z\!: z\!\in\!A,~t\!\in\!(S^1)^k\big\}$$
is the image of the compact subset $(S^1)^k\!\times\!A$ in $\mu_M^{-1}(\tau)$
under the continuous multiplication map
$$(S^1)^k\times \mu_M^{-1}(\tau)\lra  \mu_M^{-1}(\tau)$$
and thus compact.
Since $\mu_M^{-1}(\tau)$ is Hausdorff, it follows that
$p^{-1}(p(A))$ is a closed subset of $\mu_M^{-1}(\tau)$.
We conclude the quotient map~$p$ is a closed map.
Since $\mu_M^{-1}(\tau)$ is a normal topological space,  
by \cite[Lemma~73.3]{Mu} so is $\mu_M^{-1}(\tau)/(S^1)^k$.\\
\ref{compact_lmm}
The map~\e_ref{re-co_e} is well-defined, since the inclusion 
$\mu_M^{-1}(\tau)\!\hookrightarrow\!\wt{X}^{\tau}_M$ is equivariant 
under the inclusion $(S^1)^k\!\hookrightarrow\!\T^k$, and is continuous 
by the defining property of the quotient topology.
The~map
$$\wt{X}^{\tau}_M\lra \mu_M^{-1}(\tau), \qquad z\lra \rho_M^{\tau}(z)\cdot z,$$
is equivariant with respect to the natural projection $\T^k\!\lra\!(S^1)^k$
by the uniqueness property in~\ref{levelgen_lmm}
and thus induces a continuous map in the opposite direction to~\e_ref{re-co_e}.
Since $\rho^{\tau}_M|_{(\mu_M^{-1}(\tau))}\!=\!(1,\ldots,1)$, the two maps are easily seen to be mutual inverses.\\
\ref{manifold_lmm}
We cover $X_M^{\tau}$ by holomorphic charts as in \cite[Propositions~2.17, 2.18]{Ba}.
For each $J\!\in\!\V$, let
\begin{gather}
[N]\!-\!J\equiv \{i_1\!<\!i_2\!<\!\ldots\!<\!i_{N-k}\},\quad
\wt{U}_J\equiv\left\{z\!\in\!\C^N\!\!:\supp(z)\!\supseteq\!J\right\},\
\quad U_J\equiv\wt{U}_J/\T^k,
\notag\\
\label{charts_e}
h_J:U_J\lra\C^{N-k},\qquad
h_J[z]\equiv
\left(\frac{z_{i_1}}{z_J^{M_J^{-1}M_{i_1}}},\frac{z_{i_2}}{z_J^{M_J^{-1}M_{i_2}}},\ldots,
\frac{z_{i_{N-k}}}{z_J^{M_{J}^{-1}M_{i_{N-k}}}}\right).
\end{gather}
By Lemma~\ref{moment_lmm}\ref{charts_lmm2}, the collections
$\{\wt{U}_J:J\!\in\!\V\}$ and $\{U_J:J\!\in\!\V\}$ cover $\wt{X}^{\tau}_M$ and $\X$, respectively.
The map $h_J$ is well-defined. First, $M_J^{-1}$ exists and is an integer matrix by \ref{S} in Definition~\ref{toric_dfn}.
Second, if $t\!\in\!\T^k$, $z\!\in\!\wt{U}_J$, and
$J\!\equiv\!\{j_1\!<\!j_2\!<\ldots\!<\!j_k\}$, then
\begin{align*}\begin{split}
(t\!\cdot\!z)_J^{-M_J^{-1}M_{i_s}}(t\!\cdot\!z)_{i_s}\!&=\!
\left(\left(t^{M_{j_1}}z_{j_1}\right)^{-\left(M_J^{-1}M_{i_s}\right)_{1}}
\ldots\left(t^{M_{j_k}}z_{j_k}\right)^{-\left(M_J^{-1}M_{i_s}\right)_{k}}\right)t^{M_{i_s}}z_{i_s}\\
&=t^{-M_J\left(M_J^{-1}M_{i_s}\right)+M_{i_s}}z_J^{-M_J^{-1}M_{i_s}}z_{i_s}
\!=\!
z_J^{-M_J^{-1}M_{i_s}}z_{i_s},\qquad\forall\,s\!\in\![N\!-\!k].\qquad\qquad\,\,\,\,\,
\end{split}\end{align*}
The map $h_J^{-1}$ is the composition of the continuous maps
$$\C^{N-k}\xrightarrow{\wt{h_J^{-1}}}\wt{U}_{J}\xrightarrow{\textnormal{projection}}U_J,\qquad
\left(\wt{h_J^{-1}}(z)\right)_i=
\begin{cases}
z_s,&\hbox{ if }i\!=\!i_s,\\
1,&\hbox{ if }i\!\in\!J,
\end{cases}\qquad\forall\,i\!\in\![N].$$
The composition $\C^{N-k}\lra U_J\xrightarrow{h_J}\C^{N-k}$ is obviously the identity.
The other relevant composition is given by
$U_J\!\ni\![z]\lra [y]\!\in\!U_J$, where
$$y_i=\begin{cases}
z_J^{-M_J^{-1}M_{i_s}}\cdot z_{i_s} ,&\hbox{ if }i\!=\!i_s,\\
1,&\hbox{ if }i\!\in\!J,
\end{cases}
\qquad\forall\,j\!\in\![N].$$
Let $t_r\!\equiv\!z_J^{-(M_J^{-1})_r}$ for all $r\!\in\![k]$; it follows that $t\!\cdot\!z\!=\!y$.

If in addition $J'\!\in\!\V$, the domain and image of the overlap map $h_J\!\circ\!h_{J'}^{-1}$
are complements of some the coordinate subspaces in~$\C^{N-k}$,
and every component of this map is a ratio of monomials in the complex coordinates.
In particular, this map is holomorphic.
\end{proof}

\begin{rmk}\label{submersion_rmk}
Let $(M,\tau)$ be a toric pair. 
The projection $\pi\!:\!\wt{X}^{\tau}_M\!\lra\!\X$ is a holomorphic submersion;
this can be seen using the charts \e_ref{charts_e}.
\end{rmk}

Let $K^{\tau}_M$ be the connected component of $\tau$ inside the regular value locus of $\mu_M$.

\begin{lmm}\label{reg_locus_lmm}
Let $(M,\tau)$ be a toric pair.
\begin{enumerate}[label=(\emph{\alph*}),leftmargin=*]
\item\label{reg_lmm}
Let $\eta\!\in\!\R^k$.
Then, $\eta$ is a regular value of $\mu_M$ if and only if 
$\eta\!\not\in\!M_J(\R^{\ge 0})^{|J|}$ for every
$J\!\subset\![N]$ with $|J|\!\le\!k\!-\!1$.
\item\label{cone_lmm}
The subset $K^{\tau}_M$ of $\R^k$ is an open cone (i.e. an open subset of $\R^k$ such that $\la\eta\!\in\!K^{\tau}_M$
whenever $\la\!>\!0$ and $\eta\!\in\!K^{\tau}_M$). 
\item\label{Vindep_lmm}
For every $\eta\!\in\!K^{\tau}_M$,  
$\mathscr{V}^{\eta}_M\!=\!\V$.
\item\label{indep_lmm}
For every $\eta\!\in\!K^{\tau}_M$,
$(M,\eta)$ is a toric pair and $X^{\eta}_M\!=\!X^{\tau}_M$.
\end{enumerate}
\end{lmm}

\begin{proof}
\ref{reg_lmm}
If $\eta$ is a regular value of $\mu_M$, 
$\eta\!\not\in\!M_J(\R^{\ge 0})^{|J|}$ for every
$J\!\subset\![N]$ with $|J|\!\le\!k\!-\!1$ by
the second statement in Lemma~\ref{moment_lmm}\ref{Mrestr_lmm}.
Suppose $\eta\!\not\in\!M_J(\R^{\ge 0})^{|J|}$ for every
$J\!\subset\![N]$ with $|J|\!\le\!k\!-\!1$.
We prove that for every $v\!\in\!P^{\eta}_M$
there exists $J\!\subseteq\!\supp(v)$ such that $|J|\!=\!k$ and $\det\,M_J\!\neq\!0$.
Suppose not, i.e. $\det\,M_J\!=\!0$ for all $J\!\subseteq\!\supp(v)$ with $|J|\!=\!k$.
We show that there exists $v'\!\in\!P^{\eta}_M$ with $|\supp(v')|\!<\!k$; this contradicts the assumption on~$\eta$.
If $|\supp(v)|\!\ge\!k$, there exists $w\!\in\!M^{-1}(0)\!\subset\!\R^N$ such that
$\supp(w)\!\subseteq\!\supp(v)$
and $w_{j_0}\!>\!0$ for some $j_0\!\in\!\supp(v)$.
Let 
$$\la\equiv\min\left\{\frac{v_j}{w_j}:j\!\in\!\supp(v)\textnormal{ such that }w_j\!>\!0\right\}.$$
It follows that $v\!-\!\la w\!\in\!P^{\eta}_M$ and $\supp(v\!-\!\la w)\!\subsetneqq\!\supp(v)$.
Continuing in this way, we obtain $v'\!\in\!P^{\eta}_M$ with $|\supp(v')|\!<\!k$.\\
\ref{cone_lmm}
This follows immediately from \ref{reg_lmm}.\\
\ref{Vindep_lmm}
We show that the set $\{\eta\!\in\!K^{\tau}_M:\mathscr{V}^{\eta}_M\!=\!\V\}$
is open and closed in $K^{\tau}_M$ and thus equals $K^{\tau}_M$.
It suffices to show that for any $\mathscr{P}\!\subseteq\!\{J\!\subseteq\![N]:|J|\!=\!k\}$ the set
$$\left\{\eta\!\in\!K^{\tau}_M:\mathscr{V}^{\eta}_M\!=\!\mathscr{P}\right\}\!=\!
\bigcap\limits_{J\in\mathscr{P}}\left\{\eta\!\in\!K^{\tau}_M:P^{\eta}_M\!\cap\!\R^J\!\neq\!\eset\right\}
\cap
\bigcap\limits_{\begin{subarray}{c}J\subseteq[N], |J|=k\\J\notin\mathscr{P}\end{subarray}}\left\{\eta\!\in\!K^{\tau}_M:P^{\eta}_M\!\cap\!\R^J\!=\!\eset\right\}$$
is open.
We show that  the set 
$$ \big\{\eta\!\in\!K^{\tau}_M\!:\,P^{\eta}_M\!\cap\!\R^J\!\neq\!\eset\big\}$$
with $J\!\subseteq\![N]$ and $|J|\!=\!k$ is open. 
Let $\eta'$ be any of its elements and let
$w\!\in\!P^{\eta'}_M\!\cap\!\R^J$.
By the surjectivity of $\nd_w\mu_M$, $\supp(w)\!=\!J$ and $\det\,M_J\!\neq\!0$; this shows that
$M_J(\R^{>0})^k$ is open and
$$\eta'\!\in\!M_J\left(\R^{>0}\right)^k\!\cap\!K^{\tau}_M\!\subseteq\!
\{\eta\!\in\!K^{\tau}_M:P^{\eta}_M\!\cap\!\R^J\!\neq\!\eset\}.$$
The set 
$$\big\{\eta\!\in\!K^{\tau}_M\!:\,P^{\eta}_M\!\cap\!\R^J\!=\!\eset\big\}
=K^{\tau}_M\!-\!M_J(\R^{\ge 0})^{|J|}$$ 
with $J\!\subseteq\![N]$  and $|J|\!=\!k$ is open as well.\\ 
\ref{indep_lmm}
Since $P^{\tau}_M\!\neq\!\eset$, $\mu_M^{-1}(\tau)\!\neq\!\eset$ and so $\V\!\neq\!\eset$ by Lemma~\ref{moment_lmm}\ref{charts_lmm}.
Since $\V\!\neq\!\eset$, $\mathscr{V}^{\eta}_M\!\neq\!\eset$ by~\ref{Vindep_lmm}
and so $P^{\eta}_M\!\neq\!\eset$.
Since $(M,\tau)$ satisfies \ref{S} in Definition~\ref{toric_dfn},
by \ref{Vindep_lmm} so does $(M,\eta)$.
Thus, $(M,\eta)$ is toric.
The equality $X^{\eta}_M\!=\!X^{\tau}_M$ follows from \ref{Vindep_lmm} together with 
Lemma~\ref{moment_lmm}\ref{charts_lmm2}.
\end{proof}

\begin{lmm}\label{levels_lmm}
Let $(M,\tau)$ be a toric pair.
\begin{enumerate}[label=(\emph{\alph*}),leftmargin=*]
\item\label{smooth_lmm}
The quotient $\mu_M^{-1}(\tau)/(S^1)^k$ admits a unique smooth structure such that the projection
\begin{equation}\label{re-projection_e}
\pi_{\tau}:\mu_M^{-1}(\tau)\lra\mu_M^{-1}(\tau)/(S^1)^k
\end{equation}
is a submersion.
\item\label{symp_lmm}
There exists a unique symplectic form $\om_{\tau}$
on $\mu_M^{-1}(\tau)/(S^1)^k$ such that 
$$\pi_{\tau}^*\om_{\tau}=\om_{\std}\Big|_{\mu_M^{-1}(\tau)},$$
where $\pi_{\tau}$ is the projection \e_ref{re-projection_e}.
\item\label{diffeo_lmm}
The map \e_ref{re-co_e} is a diffeomorphism.
\end{enumerate}
\end{lmm}
\begin{proof}
\ref{smooth_lmm} By \cite[Proposition~5.2]{tD}, if $G$ is a compact Lie group
acting freely and smoothly on a manifold $M$, then the quotient
$M/G$ carries a unique differentiable structure such that the projection
$$M\lra M/G$$ is a submersion.
Thus, the claim follows from \ref{O} in Definition~\ref{toric_dfn}
and Lemma~\ref{top_lmm}\ref{free_lmm}.\\
\ref{symp_lmm}
This follows from the Marsden-Weinstein symplectic reduction theorem \cite[Theorem~1]{MW}.\\
\ref{diffeo_lmm} By \ref{smooth_lmm} and Lemma~\ref{top_lmm}\ref{compact_lmm}, 
it is enough to show that the restriction
$$\pi\Big|_{\mu^{-1}_M(\tau)}:\mu_M^{-1}(\tau)\lra X^{\tau}_M$$
of the projection $\pi\!:\!\wt{X}^{\tau}_M\!\lra\!X^{\tau}_M$
is a submersion. 
This follows from the fact that the map
$$\T^k\!\times\!\mu_M^{-1}(\tau)\lra X^{\tau}_M,\quad (t,z)\lra [z],$$
is a submersion 
whose differential at $(t,z)$ vanishes on $T_t\T^k\!\times\!0$.
This map is a submersion because it is the composition of two submersions,
$$ \T^k\!\times\!\mu_M^{-1}(\tau)\lra\wt{X}^{\tau}_M,~~ (t,z)\lra t\!\cdot\!z 
\quad\hbox{and}\quad \pi:\wt{X}^{\tau}_M\lra X^{\tau}_M.$$ 
The former map is a submersion by the proof of Lemma~\ref{top_lmm}\ref{sat_lmm}, 
while  $\pi$ is a submersion by Remark~\ref{submersion_rmk}.
\end{proof}

If $(M,\tau)$ is a toric pair, we abuse notation and denote by $\om_{\tau}$ not only the form on
$\mu_M^{-1}(\tau)/(S^1)^k$ defined by Lemma~\ref{levels_lmm}\ref{symp_lmm},
but also the form it induces on $X^{\tau}_M$ via the diffeomorphism \e_ref{re-co_e}
of Lemma~\ref{levels_lmm}\ref{diffeo_lmm}.
In this case, by Lemma~\ref{reg_locus_lmm}\ref{indep_lmm} and Lemma~\ref{levels_lmm}\ref{symp_lmm}, for every $\eta\!\in\!K^{\tau}_M$,
$\om_{\eta}$ is the unique symplectic form on $X^{\tau}_M$
satisfying
\begin{equation}\label{Kahler_form_e}
\pi^*\om_{\eta}\big|_{\mu_M^{-1}(\eta)}=\om_{\std}\big|_{\mu_M^{-1}(\eta)},\quad\textnormal{where}\quad
\pi:\wt{X}^{\tau}_M\lra\X
\end{equation}
is the projection; see also diagram~\e_ref{Mtau_diag}.

\begin{lmm}\label{Kahler_lmm}
Let $(M,\tau)$ be a toric pair. For every $\eta\!\in\!K^{\tau}_M$,
$\om_{\eta}$ is K\"{a}hler with respect to the complex structure on $X^{\tau}_M$.
\end{lmm}

\begin{proof}
The form $\om_{\eta}$ is positive with respect to the complex structure on $X^{\tau}_M$ by \e_ref{Kahler_form_e}
together with
the equality $\T^k\!\cdot\mu^{-1}_M(\eta)\!=\!\wt{X}^{\tau}_M$
(justified by Lemmas~\ref{top_lmm}\ref{compact_lmm} and~\ref{reg_locus_lmm}\ref{indep_lmm}),  Remark~\ref{submersion_rmk}, and the positivity of $\om_{\std}$.
\end{proof}

\begin{rmk}\label{submfld_rmk}
If $(M,\tau)$ is a toric pair
and $J\!\subseteq\![N]$,
the pair $(M_J,\tau)$ is toric if and only if 
$P^{\tau}_{M_J}\!\neq\!\emptyset$.
In this case, $X^{\tau}_{M_J}$ is a connected compact projective manifold of complex dimension
$|J|\!-\!k$ by Proposition~\ref{toric_prp}.
It is biholomorphic to
$$X^{\tau}_M(J)\!\equiv\!\left\{[z]\!\in\!X^{\tau}_M:\supp(z)\!\subseteq\!J\right\}$$
via the map
\begin{equation}\label{subman_e}
X^{\tau}_{M_J}\ni[z]\lra[\iota_{J}(z)]\!\in\!X^{\tau}_M(J),
\qquad \textnormal{where} \quad
\big(\iota_{J}(z)\big)_j\equiv
\begin{cases}z_r,&\hbox{ if }j\!=\!j_r,\\
0,&\hbox{if}~j\!\notin\!J,
\end{cases} 
\end{equation}
if $J\!=\!\{j_1\!<\!j_2\!<\!\ldots\!<\!j_r\}$.
In particular, if $(M,\tau)$ is a minimal toric pair and $M_{\widehat{j}}$
is the matrix obtained from $M$ by deleting the $j$-th column, then
$X^{\tau}_{M_{\widehat{j}}}$ is a connected compact projective manifold of complex dimension $N\!-\!1$.
The map \e_ref{subman_e} identifies $X^{\tau}_{M_{\widehat{j}}}$ with
the hypersurface
\begin{equation}\label{torichyp_e}
X^{\tau}_M([N]\!-\!\{j\})\equiv D_j\equiv \left\{[z]\in X^{\tau}_M:z_j\!=\!0\right\}.
\end{equation}
If $J\!\in\!\V$ with $\V$ defined by \e_ref{Vset_e}, then $X^{\tau}_M(J)$ is the point 
\begin{equation}\label{fpt_e}
[J]\equiv[z_1,\ldots,z_N],\qquad\textnormal{where}\qquad
z_j\equiv\begin{cases}1,&\hbox{if}\quad j\!\in\!J;\\
0,&\hbox{otherwise}.\end{cases}\end{equation}
This follows from Lemma~\ref{moment_lmm}\ref{charts_lmm2} and \ref{S} in Definition~\ref{toric_dfn}.

If $J\!\subseteq\![N]$ is such that $P^{\tau}_M\!\cap\!\R^J\!\neq\!\eset$ and $|J|\!=\!k\!+\!1$, then
$X^{\tau}_M(J)$ is a one-dimensional complex manifold and there exist exactly $2$
multi-indices $I\!\in\!\V$ with $I\!\subset\!J$. The latter follows since 
multi-indices $I\!\in\!\V$ with $I\!\subset\!J$ correspond bijectively via $\iota_J$
to elements of $\mathscr{V}^{\tau}_{M_J}$,
which in turn correspond  to the vertices of
$P^{\tau}_{M_J}$ by Lemma~\ref{moment_lmm}\ref{vertcoresp_lmm}; $P^{\tau}_{M_J}$ has dimension~$1$ by Lemma~\ref{moment_lmm}\ref{Pdim_lmm}.
\end{rmk}

\begin{rmk}\label{matinv_rmk}
If $(M,\tau)$ is a toric pair with $M$ a $k\!\times\!N$ matrix, then $(VM,V\tau)$ is a toric pair
whenever $V\!\in\!\textnormal{GL}_k(\Z)$. In this case, $\msV^{\tau}_M\!=\!\msV^{V\tau}_{VM}$ and $X^{\tau}_M$
is biholomorphic to $X^{V\tau}_{VM}$.
The pair $(VM,V\tau)$ satisfies the first condition of \ref{O} in Definition~\ref{toric_dfn},
since $V$ is an isomorphism.
Since $P^{\tau}_M\!=\!P^{V\tau}_{VM}$, $\msV^{\tau}_M\!=\!\msV^{V\tau}_{VM}$
and so $(VM,V\tau)$ satisfies the second condition of \ref{O}, \ref{S}, and \ref{C} in Definition~\ref{toric_dfn}  as well.
\end{rmk}

\begin{rmk}\label{1dim_rmk}
If $(M,\tau)$ is a toric pair with $M$ a $k\!\times\!(k\!+\!1)$ matrix,
then $X^{\tau}_M$ is biholomorphic to~$\P^1$.
In order to see this, note first that $|\V|\!=\!2$ by Lemma~\ref{moment_lmm}\ref{vertcoresp_lmm} and Lemma~\ref{moment_lmm}\ref{Pdim_lmm}.
By Remark~\ref{matinv_rmk}, we can assume that $M_J\!=\!\Id_k$ for some $J\!\in\!\V$.
The claim now follows from \e_ref{charts_e}: 
$X^{\tau}_M$ is a compact manifold covered by two charts 
$$h_J:U_J\isomto\C,\qquad h_I:U_I\isomto\C$$
satisfying $h_J(U_J\!\cap\!U_I)\!=\!h_I(U_J\!\cap\!U_I)\!=\!\C^*$ since $I\!\cup\!J\!=\![k\!+\!1]$ and $h_I\!\circ\!h_J^{-1}(z)\!=\!z^{\pm 1}$
by \ref{S} in Definition~\ref{toric_dfn}.
\end{rmk}

\subsection{Cohomology, K\"{a}hler cone, and Picard group}
\label{nequivcoh_sec}

Throughout the remaining part of this paper, $(M,\tau)$ is a toric pair.
In order to complete the proof in Section~\ref{nequiv-intro_sec} 
that $X^{\tau}_M$ is projective,
we describe some holomorphic line bundles over it.
For each $\bp\!\in\!\Z^k$, let
\begin{equation}\label{Lp_e}
L_{\bp}\equiv\wt{X}^{\tau}_M\times_{\bp}\C
\equiv \wt{X}^{\tau}_M\!\times\!\C\big/\sim,
\qquad\textnormal{where}\quad
(z,c)\sim(t^{-1}\!\cdot\!z ,t^{\bp} c),~\forall\,t\in\T^k.
\end{equation}
Since $\pi\!:\!\wt{X}^{\tau}_M\!\lra\!\X$ with $\pi(z)\!\equiv\![z]$ is a $\T^k$-principal bundle by Lemma~\ref{top_lmm}\ref{free_lmm} and Remark~\ref{submersion_rmk},
$$L_{\bp}\lra\X,\qquad[z,c]\lra[z],$$
is a holomorphic line bundle. 
Furthermore,
$$L_0=\O_{X^{\tau}_M},\qquad L_{\bp}^*= L_{-\bp},\qquad L_{\bp}\otimes L_{\br}=L_{\bp+\br}.$$
The line bundle $L_{-M_j}$ admits a holomorphic section
\begin{equation}\label{sjdfn_e}
s_j:X^{\tau}_M\lra L_{-M_j},\qquad [z]\lra [z,z_j].
\end{equation}
Since $s_j$ is transverse to the zero set by \e_ref{charts_e}
and $s_j^{-1}(0)\!=\!D_j$ by \e_ref{torichyp_e},
$c_1(L_{-M_j})\!=\!\PD(D_j)$. 

For all $j\!\in\![N]$ and $i\!\in\![k]$, let
\begin{equation}\label{bdle_e}\begin{split}
\nU_j\equiv c_1(L_{-M_j}),\qquad
\ga_i\equiv L_{e_i},\qquad \nH_i\equiv c_1(\ga_i^*),
\end{split}\end{equation}
where $\{e_i:i\!\in\![k]\}\!\subset\!\Z^k$ is the standard basis.
Thus,
\begin{equation}\label{U_e}
L_{-M_j}=
\ga_1^{*\otimes m_{1j}}\otimes\ga_2^{*\otimes m_{2j}}
\otimes\ldots\otimes\ga_k^{*\otimes m_{kj}}\quad
 \Lra\quad
 \nU_j=\sum_{i=1}^k m_{ij}\nH_i\qquad\forall\,j\!\in\![N].
\end{equation}
Lemma~\ref{positive_lmm} below is used in the proof of Proposition~\ref{toric_prp} in Section~\ref{nequiv-intro_sec}
and to describe the K\"{a}hler cone of $\X$ in Proposition~\ref{dimK_prp} below.

\begin{lmm}\label{positive_lmm}
For every $\eta\!\in\!\Z^k\!\cap\!K^{\tau}_M$,
$$c_1(L_{-\eta})\!=\frac{1}{\pi}[\om_{\eta}],$$
where $\om_{\eta}$ is the K\"{a}hler form defined by \e_ref{Kahler_form_e}.
\end{lmm}

\begin{proof}
We follow closely the proof of \cite[Proposition~VII.3.1]{Au}.
Let 
$$L_{-\eta}^{\R}\lra \mu_M^{-1}(\eta)\big/(S^1)^k$$ 
be the pull-back of $L_{-\eta}$ via the diffeomorphism
\e_ref{re-co_e} of Lemma~\ref{levels_lmm}\ref{diffeo_lmm}
and
$$L_{-\eta}^{S^1}\equiv
\mu_M^{-1}(\eta)\times_{-\eta}S^1\xrightarrow{p}\frac{\mu_M^{-1}(\eta)}{(S^1)^k}$$
be its sphere bundle.
Let
$$\xymatrix@R=3pc@C=3pc{\mu_M^{-1}(\eta)\!\times\!S^1\ar[d]_{q}\ar[r]^{\wt{p}}
&\mu_M^{-1}(\eta)\ar[d]^{\pi_{\eta}}\\
L_{-\eta}^{S^1}\ar[r]^{p}&\frac{\mu_M^{-1}(\eta)}{(S^1)^k}}$$
be the natural projections.

Let $e_i^{\#}$ be the fundamental vector field on $\mu_M^{-1}(\eta)\!\times\!S^1$
corresponding to $e_i\!\in\!\R^k$ for the $\T^k$-action given by~\e_ref{Lp_e}
with $\bp\!=\!-\eta$.
Thus, 
\begin{equation*}\begin{split}
e_i^{\#}
&\equiv\left.\frac{\nd}{\nd t}\right|_{t=0}
\left(\exp(\mi t e_i)\cdot(x_1\!+\!\mi y_1,\!\ldots\!,
x_N\!+\!\mi y_N,x\!+\!\mi y)\right)\\
&=
\sum_{j=1}^Nm_{ij}\left(y_j\frac{\partial}{\partial x_j}\!-\!x_j\frac{\partial}{\partial y_j}\right)
+\eta_i\left(y\frac{\partial}{\partial x}\!-\!x\frac{\partial}{\partial y}\right),
\end{split}\end{equation*}
where $x_j,y_j,x,y$ are the standard coordinates on 
$\C^N\!\equiv\!(\R^2)^N$ and $\C\!\equiv\!\R^2$, respectively.
Let 
$$\al\equiv\sum_{j=1}^N(-x_j\nd y_j+y_j\nd x_j)\in\Om^1\left(\mu_M^{-1}(\eta)\right)
\qquad\hbox{and}\qquad
\si\equiv x\nd y\!-\!y\nd x\in\Om^1(S^1).$$
Since $\iota_{e_i^{\#}}(\al\oplus\si)\!=\!0$ 
on $\mu_M^{-1}(\eta)\!\times\!S^1$ for all $i\!\in\![k]$, $\al\oplus\si$
descends to a 1-form $(\al\oplus\si)_{S^1}$ on $L_{-\eta}^{S^1}$.
This form is a connection 1-form for the principal $S^1$-bundle $L^{S^1}_{-\eta}$
because it satisfies
$$\cL_{X^{\#}}(\al\oplus\si)_{S^1}\!=\!0,\qquad\iota_{X^{\#}}(\al\oplus\si)_{S^1}\!=\!1,$$
where 
$$X^{\#}=-y\frac{\partial}{\partial x}+x\frac{\partial}{\partial y}$$
is the fundamental vector field for the $S^1$-action on 
$L^{S^1}_{-\eta}$ as a principal $S^1$-bundle; see \cite[Exercises~V.4,V.5]{Au}.
Let $\be$ denote the curvature form associated to $(\al\oplus\si)_{S^1}$.
By \cite[Section~V.4.c]{Au}, it is uniquely determined by 
$$p^*\be=\nd(\al\oplus\si)_{S^1}.$$
Since $q^*\nd((\al\oplus\si)_{S^1})=-2 \wt{p}^*\om_{\std}$, $\be=-2\om_{\eta}$
by the uniqueness of reduced symplectic form $\om_{\eta}$
of Lemma~\ref{levels_lmm}\ref{symp_lmm}.
Thus, by
\cite[Proposition~VI.1.18]{Au} and \cite[Section~VI.5.b]{Au},
$$c_1(L^{\R}_{-\eta})=\frac{-1}{2\pi}[\be]=\frac{1}{\pi}[\om_{\eta}]
\!\in\!H^2_{\textnormal{deR}}(\mu_M^{-1}(\eta)/(S^1)^k),$$
as claimed.
\end{proof}

We define
\begin{equation}\label{cE_e}
\cE^{\tau}_M\equiv\left\{J\!\subseteq\![N]:\bigcap\limits_{j\in J}
D_j\!=\!\emptyset\right\}\!=\!
\left\{J\!\subseteq\![N]:M_{J^c}^{-1}(\tau)\!\cap\!(\R^{>0})^{|J^c|}\!=\!\emptyset\right\};
\end{equation}
the second equality follows from 
Lemma~\ref{moment_lmm}\ref{charts_lmm2}\ref{V_lmm}\ref{real_lmm} and \e_ref{torichyp_e}.

\begin{prp}\label{nequivcoh_prp}
If $(M,\tau)$is a toric pair, 
$$H^*\left(X_M^{\tau}\right) \cong
\frac{\Q\left[\nH_1,\nH_2,\ldots,\nH_k,\nU_1,\nU_2,\ldots,\nU_N\right]}
{\left(\nU_j-\sum\limits_{i=1}^km_{ij}\nH_i,~1\!\le\!j\!\le\![N]\right)
 +\left(\prod\limits_{j\in J}\nU_{j}\!: J\!\in\!\cE^{\tau}_M\right)}.$$
If, in addition $(M,\tau)$ is minimal,
$H^2\left(X_M^{\tau};\Z\right)$ is free with basis $\{\nH_1,\nH_2,\ldots,\nH_k\}$.
\end{prp}

\begin{proof}
This follows from \cite[Section~11.3]{McDSa}
together with Lemma~\ref{levels_lmm}\ref{diffeo_lmm}.
\end{proof}

\begin{rmk}\label{gencoh_rmk}
By Proposition~\ref{nequivcoh_prp},  $H^*(\X)$ is generated as a $\Q$-algebra 
by $\{\nH_1,\ldots,\nH_k\}$.
Along with~\ref{S} in Definition~\ref{toric_dfn}, this implies that 
$H^*(\X)$ is generated as a $\Q$-algebra by $\{\nU_1,\ldots,\nU_N\}$. 
\end{rmk}

\begin{prp}\label{dimK_prp}
If $(M,\tau)$ is a minimal toric pair, there is a basis 
$\{c_1(L_{-\eta_i})\!:i\!\in\![k]\}$ for  $H^2(\X)$
formed by the first Chern classes of ample line bundles,
with $L_{-\eta_i}$ as in \e_ref{Lp_e}. 
In particular, the K\"{a}hler cone $\cK^{\tau}_M$ of $X^{\tau}_M$ has dimension~$k$.
\end{prp}

\begin{proof}
By Lemmas~\ref{positive_lmm}, \ref{Kahler_lmm}, and
~\ref{reg_locus_lmm}\ref{cone_lmm}\ref{indep_lmm}, 
there exists a subset $\{\eta_1,\ldots,\eta_k\}\!\subseteq\!\Z^k$,
linearly independent over~$\Q$,
such that the line bundles $L_{-\eta_j}$ are positive.
The first Chern classes of these line bundles  form a $\Q$-basis of $H^2(X^{\tau}_M)$
by the last statement in Proposition~\ref{nequivcoh_prp}.
\end{proof}

\begin{prp}\label{Pic_prp}
If $(M,\tau)$ is a minimal toric pair,
the Picard group of $X^{\tau}_M$ is free of rank $k$ and has a $\Z$-basis given by
$\ga_1,\ldots,\ga_k$ defined by \e_ref{bdle_e}.
\end{prp}
\begin{proof}
The first Chern class homomorphism is an isomorphism because $h^{0,1}(X^{\tau}_M)\!=\!h^{0,2}(X^{\tau}_M)\!=\!0$
which in turn follows from Proposition~\ref{nequivcoh_prp}.
\end{proof}

\begin{rmk}\label{ctan_rmk}
If $(M,\tau)$ is a toric pair, there is a short exact sequence
\begin{equation}\label{TXses_e}
0\lra \cO_{\X}^{\oplus k} \stackrel{F}{\lra} \bigoplus_{j=1}^N L_{-M_j}
\stackrel{G}{\lra} T\X \lra0.
\end{equation}
Specifically, we can take
\begin{alignat*}{2}
F([z],e_i)&\equiv\big[z,m_{i1}z_1,m_{i2}z_2,\ldots,m_{iN}z_N\big]
&\qquad&\forall\,i\!\in\![k],\,[z]\!\in\!\X,\\
G(z,y_1,\ldots,y_N)&\equiv \sum_{j=1}^Ny_j\nd_z\pi
 \left(\frac{\partial}{\partial z_j}\Big|_z\right),
&\qquad&\forall\,z\in\wt{X}^{\tau}_M,\,y_1,\ldots,y_N\!\in\!\C,
\end{alignat*}
where $\{e_i:i\!\in\![k]\}$ is the 
standard basis for $\C^k$ and $\pi\!:\!\wt{X}^{\tau}_M\!\lra\!\X$ is the projection.
Thus,
\begin{equation}\label{ctan_e}
c_1\left(T\X\right)=\sum_{j=1}^N \nU_j.
\end{equation}
\end{rmk}

\subsection{Torus action and equivariant notation}
\label{equiv-intro_sec}

The \sf{equivariant cohomology} of a topological space $X$ endowed 
with a continuous $\T^N$\!-action is
$$H^*_{\T^N}(X)\equiv H^*\left(E\T^N\!\times_{\T^N}\!X\right),$$
where $E\T^N\!\equiv\!(\C^{\i}-0)^N$ is the classifying space for $\T^N$\!.
In particular, the equivariant cohomology of a point is 
\begin{equation*}
H^*_{\T^N}(point)\cong H^*_{\T^N}\equiv H^*(\left(\P^{\i})^N\right)
=\Q[\al_1,\ldots,\al_N]\equiv\Q[\al],
\end{equation*}
where $\al_j\!\equiv\!c_1(\pi_j^*\cO_{\P^{\i}}(1))$,
$\pi_j\!:\!(\P^{\i})^N\!\lra\!\P^{\i}$
is the projection onto the $j$-th component and
$\cO_{\P^{\i}}(1)$ is dual to the tautological line bundle  over~$\P^{\i}$.
The \sf{equivariant Euler class} of an oriented vector bundle $V\!\lra\!X$
endowed with a lift of the $\T^N$-action on $X$ is
$$\E(V)\equiv e(E\T^N\!\times_{\T^N}\!V)\in H^*_{\T^N}(X).$$
A $\T^N$-equivariant map $f\!:\!X\!\lra\!Y$ between compact oriented manifolds
induces a push-forward map 
$$f_*\!:\!H^{s}_{\T^N}(X) \lra H^{s+\dim Y-\dim X}_{\T^N}(Y)$$ 
characterized~by
\begin{equation}\label{push-fwd_e}
\int_{X}(f^*\eta)\eta'\!=\!\int_Y\eta
(f_*\eta')\qquad\forall\,\eta\!\in\!H^*_{\T^N}(Y),\eta'\!\in\!H^*_{\T^N}(X).
\end{equation}
If $Y$ is a point, $f_*$ is the integration along the fiber homomorphism
$\int_{X}\!:\!H^s_{\T^N}(X)\!\lra\!H^{s-\dim X}_{\T^N}$.
The push-forward map $f_*$ extends to a homomorphism between the modules of fractions
with denominators in $\Q[\al]$; in particular,
the integration along the fiber homomorphism extends to
$$\int_X:H^*_{\T^N}(X)\otimes_{\Q[\al]}\Q(\al)\lra\Q(\al),\qquad\textnormal{where}
\qquad\Q(\al)\equiv\Q(\al_1,\ldots,\al_N)$$
is the field of fractions of $\Q[\al]$.
If $X$ is a compact oriented manifold on which $\T^N$ acts smoothly, then, 
by the classical Localization Theorem~\cite{ABo}
\begin{equation}\label{ABo_e}
\Q[\al]\ni\int_{X}\eta=\sum_{F\subset X^{\T^N}}\int_{F}\frac{\eta}{\E(N_{F/X})}\in\Q(\al),
\qquad\forall\,\eta\!\in\!H^*_{\T^N}(X),
\end{equation}
where the sum runs over the components of the $\T^N$ pointwise fixed locus $X^{\T^N}$ of $X$.

\begin{lmm}\label{orbit_lmm}
If $(M,\tau)$ is a toric pair, 
$(\T^N\!\cdot\!z/\T^k)$ is diffeomorphic to $\T^{|\supp(z)|-k}$
for every $z\!\in\!\wt{X}^{\tau}_M$.
\end{lmm}

\begin{proof}
By Lemma~\ref{moment_lmm}\ref{charts_lmm2} and \ref{S} in Definition~\ref{toric_dfn},
there exists $J\!\subseteq\!\supp(z)$ with $|J|\!=\!k$ and $\det\,M_J\!\in\!\{\pm 1\}$.
The map 
$$\frac{\T^N\!\cdot\!z}{\T^k}\ni[y_1,\ldots,y_N]\lra
\left(y_J^{-M_J^{-1}M_s}y_s\right)_{s\,\in\,\supp(z)-J}\in\T^{|\supp(z)|-k}$$
is a diffeomorphism with inverse
$$\T^{|\supp(z)|-k}\ni\la\lra[t_1z_1,\ldots,t_Nz_N]\in\frac{\T^N\!\cdot\!z}{\T^k},\qquad
\textnormal{where}\qquad
t_j\equiv\begin{cases}
1,&\hbox{ if }j\!\notin\!\supp(z),\\
\frac{\la_s}{z_{j_s}},&\hbox{ if }j=j_s,\\
\frac{1}{z_j},&\hbox{ if }j\!\in\!J,
\end{cases}$$ and $\supp(z)\!-\!J\!\equiv\!\{j_1\!<\!\ldots\!<\!j_{|\supp(z)|-k}\}$; see the proof of Lemma~\ref{top_lmm}\ref{manifold_lmm}
in Section~\ref{nequiv-intro_sec}.
\end{proof}

\begin{crl}\label{fixed_crl}
\begin{enumerate}[label=(\emph{\alph*}),leftmargin=*]
\item\label{fpt_crl} The $\T^N$\!-fixed points in $X^{\tau}_M$ are the points $[J]$ of \e_ref{fpt_e}.
\item\label{fc_crl} The closed $\T^N$\!-fixed curves in $X^{\tau}_M$ are the submanifolds
$X^{\tau}_M(J)$ of Remark~\ref{submfld_rmk} with $|J|\!=\!k\!+\!1$;
all such tuples~$J$ are of the form $J\!=\!I_1\!\cup\!I_2$ with
with $I_1,I_2\!\in\!\V$ and $|I_1\!\cap\!I_2|\!=\!k\!-\!1$.
These curves are biholomorphic to~$\P^1$.
\end{enumerate}\end{crl}

\begin{proof}
The first two statements follow from Lemma~\ref{orbit_lmm}. The third follows from
the last part of Remark~\ref{submfld_rmk}.
The last follows from Remarks~\ref{submfld_rmk} and \ref{1dim_rmk}. 
\end{proof}

We next consider lifts of the standard action of $\T^N$ on $\X$ to the line bundles
$L_{\bp}$ of \e_ref{Lp_e} which will be used in describing the equivariant cohomology of $\X$.
One such lift is the canonical one
\begin{equation}\label{trivact_e}
(t_1,\ldots,t_N)\!\cdot\![z_1,\ldots,z_N,c]\equiv [t_1z_1,\ldots,t_Nz_N,c]
\end{equation}
for all $(t_1,\ldots,t_N)\!\in\!\T^N$, $(z_1,\ldots,z_N)\!\in\!\wt{X}_M^{\tau}$, and $c\!\in\!\C$. 
We denote by $$E\T^N\!\times_{\triv}\!L_{\bp}\!\lra\!E\T^N\!\times_{\T^N}\!X^{\tau}_M$$ 
the induced line bundle.
Another lift is given~by
\begin{equation}\label{jact_e}
(t_1,\ldots,t_N)\!\cdot\![z_1,\ldots,z_N,c]\equiv[t_1z_1,\ldots,t_Nz_N,t_jc]
\end{equation}
for all $(t_1,\ldots,t_N)\!\in\!\T^N$,
$(z_1,\ldots,z_N)\!\in\!\wt{X}^{\tau}_M$, and $c\!\in\!\C$. We denote  by
$$E\T^N\!\times_j\!L_{\bp}\!\lra\!E\T^N\!\times_{\T^N}\!X^{\tau}_M$$
the induced line bundle.
These line bundles are related by isomorphisms
\begin{alignat}{1}
\label{j_e} (E\T^N\!\times_{\triv}\!L_{\bp})\!\otimes\!(E\T^N\!\times_j\!L_0)&\cong E\T^N\!\times_j\!L_{\bp},\\
\label{jL0_e} E\T^N\!\times_j\!L_0&\cong \pr_1^*\pi_j^*\cO_{\P^{\i}}(-1),
\end{alignat}
where $\pr_1:\!E\T^N\!\times_{\T^N}\!X^{\tau}_M\!\lra(\P^{\i})^N$ denotes the natural projection.
The first of these follows by considering the isomorphism
$$L_{\bp}\!\otimes\!L_0\!\lra\!L_{\bp},\quad[z,c_1]\!\otimes[z,c_2]\lra [z,c_1c_2]
\qquad\forall\,z\!\in\wt{X}^{\tau}_M,c_1,c_2\!\in\!\C$$
which is $\T^N$-equivariant with respect to
the $\T^N$ action on $L_{\bp}\!\otimes\!L_0$ obtained by tensoring \e_ref{trivact_e} with \e_ref{jact_e}  
and the action \e_ref{jact_e} on $L_{\bp}$.
The second is given by
$$E\T^N\!\times_j\!L_0\ni(e,z,c)
\lra(e,z,ce_j)\in \pr_1^*\pi_j^*\cO_{\P^{\i}}(-1)\quad\forall\,
e\!=\!(e_1,\ldots,e_N)\!\in\!E\T^N,z\!\in\!\wt{X}^{\tau}_M,c\!\in\!\C.$$

For all $j\!\in\![N],i\!\in\![k]$ and with $\ga_i$ defined by \e_ref{bdle_e}, let
\begin{equation}\label{ebdle_e}
{\bf[D_j]}\!\equiv\!E\T^N\!\times_j\!L_{-M_j},\quad
\bga_i\!\equiv\!E\T^N\!\times_{\triv}\!\ga_i; \qquad
u_j\!\equiv\!c_1({\bf[D_j]}),\quad
x_i\!\equiv\!c_1({\bga_i^*})\!\in\!H^*_{\T^N}(X^{\tau}_M).
\end{equation}
For each $J\!\in\!\V$, the inclusion $[J]\!\hookrightarrow\!X^{\tau}_M$ 
induces a restriction map
\begin{equation}\label{dfnerestr_e}
\cdot\,(J),~\cdot\big|_{J}:H^*_{\T^N}(X^{\tau}_M)\lra H^*_{\T^N}([J])\cong H^*_{\T^N}.
\end{equation}
By \e_ref{U_e}, \e_ref{j_e} and \e_ref{jL0_e},
\begin{equation}\label{u_e}
u_j=\sum_{i=1}^km_{ij}x_i\!-\!\al_j\qquad\forall\,j\!\in\![N].
\end{equation}
For each $j\!\in\![N]$, the section \e_ref{sjdfn_e} of $L_{-M_j}\!\lra\!\X$ is 
$\T^N$-equivariant with respect to the action~\e_ref{jact_e} and thus
induces a section~$\bs_j$ of $[\bD_j]$ over $E\T^N\!\times\!_{\T^N}\X$.
If $J\!\in\!\V$ and $j\!\in\!J$, $\bs_j$ does not vanish on 
$E\T^N\!\times\!_{\T^N}[J]$ and thus
\begin{equation}\label{fprestr_e}
 J\!\in\!\V\qquad\Lra\qquad u_j(J)=0~~\forall\,j\!\in\!J.
\end{equation}
On the other hand, if $J\!\in\!\cE^{\tau}_M$, with $\cE^{\tau}_M$  defined by \e_ref{cE_e},
then $\bigoplus_{j\in J}\bs_j$ is a nowhere zero section of $\bigoplus_{j\in J}[\bD_j]$ and thus
\begin{equation}
\label{0prod_e} J\!\in\!\cE^{\tau}_M\qquad\Lra\qquad
\prod_{j\in J}u_j=0\in H^*_{\T^N}(\X).
\end{equation}

\begin{prp}\label{equivcoh_prp}
Let $(M,\tau)$ be a toric pair.
\begin{enumerate}[label=(\emph{\alph*}),leftmargin=*]
\item\label{erestr_prp} If $J\!=\!(j_1\!<\!\ldots\!<\!j_k)\!\in\!\V$, 
\begin{equation*}
\Big(x_1(J)\,\,x_2(J)\,\,\ldots\,\,x_k(J)\Big)=
\Big(\al_{j_1}\,\,\al_{j_2}\,\,\ldots\,\,\al_{j_k}\Big)M_J^{-1}.
\end{equation*}
\item\label{equivcohr_prp} With $x_i$ and $u_j$ defined by \e_ref{ebdle_e},
\begin{equation}\label{Xequivcohom_e}
H^*_{\T^N}\!\left(X_M^{\tau}\right)=\frac{\Q[\al]\left[x_1,x_2,\ldots,x_k,u_1,u_2,\ldots,u_N\right]}
{\left(u_j-\sum\limits_{i=1}^km_{ij}x_i+\al_j,~~1\!\le j\!\le N\right)+
\left(\prod\limits_{j\in J}u_j\!: J\!\in\!\cE^{\tau}_M\right)}.
\end{equation}
If in addition $P\!\in\!H^*_{\T^N}(X^{\tau}_M)$, then
$P\!=\!0$ if and only if $P(J)\!=\!0$ for all $J\!\in\!\V$.
\end{enumerate}\end{prp}

\begin{proof}
\ref{erestr_prp} This follows from \e_ref{u_e} and \e_ref{fprestr_e}.\\
\ref{equivcohr_prp} By Remark~\ref{gencoh_rmk}, 
there exists $B\!\subset\!(\Z^{\ge 0})^k$ such that
$\{\nH^{\bp}:\bp\!\in\!B\}$ is a $\Q$-basis for $H^*(\X)$. 
The~map 
\begin{equation*}
H^*\!(X^{\tau}_M)\ni \nH^{\bp}\lra x^{\bp}\in H^*_{\T^N}\!(X^{\tau}_M)\qquad\forall\,\bp\!\in\!B
\end{equation*}
defines a cohomology extension of the fiber for the fiber bundle
$E\T^N\!\times_{\T^N}\!X^{\tau}_M\!\lra\!(\P^{\i})^N$.
Thus, by the Leray-Hirsch Theorem \cite[Chapter~5]{Spa}, the map
$$H^*_{\T^N}\!\otimes\!H^*(X^{\tau}_M)\ni
P\otimes \nH^{\bp}\lra Px^{\bp}\in H^*_{\T^N}(X^{\tau}_M)\qquad\forall\,\bp\!\in\!B$$
is an isomorphism of vector spaces.
The relations in~\e_ref{Xequivcohom_e} hold by~\e_ref{u_e} and~\e_ref{0prod_e}.
We show below that there are no other relations and simultaneously 
verify the last claim.

Suppose $P\!\in\!H^*_{\T^N}(X^{\tau}_M),P(J)\!=\!0$ for all $J\!\in\!\V$.
By \ref{S} in Definition~\ref{toric_dfn}, 
any element $P$ of $H^*_{\T^N}(X^{\tau}_M)$
is a polynomial in $u_1,\ldots,u_N$ with coefficients in $\Q[\al]$.
If $J\!\in\!\V$ and $j\!\in\![N]\!-\!J$, then
\begin{equation}\label{ual_e}
u_j(J)\Big|_{\begin{subarray}{c}\al_{i}=0\,\forall\,i\in J\end{subarray}}\!=\!-\al_j
\end{equation}
by \e_ref{u_e} and \ref{erestr_prp}.
By \e_ref{fprestr_e} and \e_ref{ual_e},
whenever $u_{i_1}^{a_1}\ldots u_{i_s}^{a_s}$ is a monomial appearing in $P$ and
$J\!\in\!\V$, $\{i_1,\ldots,i_s\}\!\cap\!J\!\neq\!\emptyset$.
This shows that 
$$P\!\in\!H^*_{\T^N}(X^{\tau}_M),P(J)\!=\!0~~\forall\,J\!\in\!\V
\qquad\Lra\qquad P\!\in\!\bH',$$
where $\bH'$ is the ideal
$$\bH'\equiv\Big(u_{i_1}\ldots u_{i_s}:\{i_1,\ldots,i_s\}\!\cap\!J\!\neq\!\emptyset~
        \forall\,J\!\in\!\V\Big)\!\subset\!\Q[\al][u_1,\ldots,u_N].$$
Since $\bH'\!\subseteq\!(\prod_{j\in J} u_j:
J\!\in\!\cE^{\tau}_M)$ by Lemma~\ref{moment_lmm}\ref{charts_lmm2},
$$P\!\in\!H^*_{\T^N}(X^{\tau}_M),P(J)\!=\!0~~~\forall\,J\!\in\!\V
\qquad\Lra\qquad P\!\in\!\left(\prod\limits_{j\in J} u_j:
J\!\in\!\cE^{\tau}_M\right).$$
By \e_ref{0prod_e}, this implies that $P\!=\!0\in H^*_{\T^N}(X^{\tau}_M)$
if $P(J)\!=\!0$ for all $J\!\in\!\V$.
\end{proof}

For every $J\!\in\!\V$, let 
\begin{equation*}
\phi_J\!\equiv\!\prod\limits_{\begin{subarray}{c}j\in[N]-J\end{subarray}}u_j.
\end{equation*}
By \e_ref{TXses_e} and \e_ref{fprestr_e}, 
\begin{equation} \label{tan_e}
\phi_J(J)=\E\left(T_{[J]}X^{\tau}_M\right), \qquad
\phi_J(I)=0~~~\forall~I\in \V\!-\!\{J\}.
\end{equation}
Thus, by the Localization Theorem \e_ref{ABo_e}, 
\begin{equation}\label{phiprop_e}
\int_{X^{\tau}_M}P\phi_J=P(J)\qquad\forall\,P\!\in\!H^*_{\T^N}(X^{\tau}_M),J\!\in\!\V,
\end{equation}
i.e.~$\phi_J$ is the equivariant Poincar\'{e} dual of the point $[J]\!\in\!\X$.

\subsection{Examples}\label{eg_sec}

\begin{eg}[\emph{the complex projective space $\P^{N-1}$ with the standard action of $\T^N$}]
\label{proj_eg}
If 
$$M\equiv(1,\ldots,1)\in\R^N  \qquad\hbox{and}\qquad \tau\in\R^{>0},$$ 
then
$$\mu_M:\C^N\lra\R,~~~ \mu_M(z)\!=\!|z_1|^2\!+\!\ldots+\!|z_N|^2,
\qquad
P^{\tau}_M=\left\{v\!\in\!\left(\R^{\ge 0}\right)^N\!:v_1\!+\!\ldots\!+\!v_N\!=\!\tau\right\},$$
$(M,\tau)$ is a minimal toric pair, $\wt{X}_M^{\tau}\!=\!\C^N\!-\!\{0\}$,
$$X_M^{\tau}=\P^{N-1}\cong(S^{2n-1}(\sqrt{\tau}))/S^1, \qquad
H_{\T^N}^*(\P^{N-1})\cong \Q[\al_1,\ldots,\al_N][x]\big/\prod_{k=1}^N(x\!-\!\al_k).$$
\end{eg}

\begin{eg}[\textit{the Hirzebruch surfaces} $\F_k\!\equiv\!\P\left(\cO_{\P^1}\oplus\cO_{\P^1}(k)\right)$]
\label{Hirzebruch_eg}
If $k\!\ge\!0$,$$M\!\equiv\!\begin{pmatrix}0&0&1&1\\1&1&0&k\end{pmatrix},\qquad
\tau\!\equiv\!\begin{pmatrix}1\\k+1\end{pmatrix},$$
then 
\begin{align*}
(t_1,t_2)&\cdot(z_1,z_2,z_3,z_4)\!\equiv\!\left(t_2z_1,t_2z_2,t_1z_3,t_1t_2^kz_4\right),\\
\mu_M&:\C^4\lra\R^2,\quad\mu_M(z)\equiv
\begin{pmatrix}|z_3|^2+|z_4|^2\\|z_1|^2+|z_2|^2+k|z_4|^2\end{pmatrix},\\
P^{\tau}_M&=\left\{v\!\in\!\left(\R^{\ge 0}\right)^4\!:
v_3\!+\!v_4\!=\!1,\,v_1\!+\!v_2\!+\!kv_4\!=\!k\!+\!1\right\},
\end{align*}
$\left(M,\tau\right)$ is a minimal toric pair,
$$\wt{X}^{\tau}_M=\C^4-\Big(\C^2\!\times\!0\!\cup\!0\!\times\!\C^2\Big),\qquad X^{\tau}_M=\wt{X}_M^{\tau}\big/\T^2.$$
The map 
\begin{equation}\label{Fkisom_e}
X_M^{\tau}\isomto\F_k,\qquad
[z_1,z_2,z_3,z_4]\lra\left[[z_1,z_2],z_3,\left((z_1,z_2)^{\otimes k}\lra z_4\right)\right],
\end{equation}
is a $\T^4$-equivariant biholomorphism
with respect to the action of $\T^4$ on $\F_k$ given by
\begin{align*}
(t_1,t_2,t_3,t_4)\cdot \left[[z_1,z_2],z_3,\varphi\right]\!\equiv\!
\left[[t_1z_1,t_2z_2],t_3z_3,(t_1y_1,t_2y_2)^{\otimes k}\lra 
     t_4\varphi\left(\left(y_1,y_2\right)^{\otimes k}\right)\right],\\
\forall [z_1,z_2]\in\P^1,z_3\in\C,\varphi\in\cO_{\P^1}(k)\big|_{[z_1,z_2]}.
\end{align*}
By Proposition~\ref{nequivcoh_prp},
\begin{equation*}\begin{split}
H^*\left(\F_k\right) &=\frac{\Q\left[\nH_1,\nH_2,\nU_1,\nU_2,\nU_3,\nU_4\right]}
{\left(\nU_1-\nH_2,\nU_2-\nH_2,\nU_3-\nH_1,\nU_4-\nH_1-k\nH_2\right)+\left(\nU_1\nU_2,\nU_3\nU_4\right)}\\
&\cong
\frac{\Q\left[\nH_1,\nH_2\right]}{\left(\nH_2^2,\nH_1\left(\nH_1+k\nH_2\right)\right)}.
\end{split}\end{equation*}
Since the toric hypersurfaces $D_2$ and $D_3$ defined by \e_ref{torichyp_e}
intersect at one point,
$$\nH_1\nH_2= \nU_2\nU_3=1,\qquad \nH_1^2\!=\!-k\nH_1\nH_2\!=\!-k.$$
The isomorphism \e_ref{Fkisom_e} maps $D_4$ onto $E_0$
and  $D_3$ onto $E_{\i}$, where
\begin{align*}
E_0&\equiv\textnormal{ image of the section }(1,0)\textnormal{ in }\F_k,\\
E_{\i}&\equiv\textnormal{ closure of the image of }(0,\si)\textnormal{ in }\F_k,
\end{align*}
where $\si$ is any non-zero holomorphic section of $\cO_{\P^1}(k)$. 

Since $\V=\{(1,3),(1,4),(2,3),(2,4)\}$, 
by Corollary~\ref{fixed_crl}, the $\T^4$-fixed points in $\X$ are
$$[1,0,1,0],[1,0,0,1],[0,1,1,0],\quad\textnormal{and}\quad[0,1,0,1],$$
while the closed $\T^4$-fixed curves are all 4 toric 
hypersurfaces $D_1$, $D_2$, $D_3$, and~$D_4$.
By Proposition~\ref{equivcoh_prp}\ref{equivcohr_prp},
$$H^*_{\T^4}\left(\F_k\right)\!\equiv\!
\frac{\Q[\al_1,\al_2,\al_3,\al_4][x_1,x_2]}
{\Big((x_2-\al_1)(x_2-\al_2),(x_1-\al_3)(x_1+kx_2-\al_4)\Big)}.$$
\end{eg}

\begin{eg}[\textit{products}]
\label{prod_eg}
Let $(M_1,\tau_1)$ and $(M_2,\tau_2)$ be (minimal) toric pairs, 
where $M_j$ is a $k_j\!\times\!N_j$ matrix.
Define\begin{equation*}
M_1\oplus M_2\!\equiv\! 
\begin{pmatrix} M_1 &\text{\large 0}\\ 
       \text{\large 0}&M_2  \end{pmatrix}.
\end{equation*}
Then, $(M_1\oplus M_2,(\tau_1,\tau_2))$ is a (minimal) toric pair,
$$P^{(\tau_1,\tau_2)}_{M_1\oplus M_2}=P^{\tau_1}_{M_1}\times P^{\tau_2}_{M_2},
\qquad\hbox{and}\qquad
\wt{X}^{(\tau_1,\tau_2)}_{M_1\oplus M_2}
  =\wt{X}^{\tau_1}_{M_1} \times \wt{X}^{\tau_2}_{M_2}.$$
The projections $\pi_j\!:\!\C^{N_1+N_2}\!\!\lra\!\C^{N_j}$ 
induce a $\T^{N_1+N_2}$-equivariant biholomorphism
$$X^{(\tau_1,\tau_2)}_{M_1\oplus M_2}  \ni[z] \stackrel{\sim}{\lra} 
\Big([\pi_1(z)],[\pi_2(z)]\Big)\in  X^{\tau_1}_{M_1}\times X_{M_2}^{\tau_2},$$
where the action of $\T^{N_1+N_2}$ on $X^{\tau_1}_{M_1}\!\times\!X^{\tau_2}_{M_2}$
is the product of the standard actions of 
$\T^{N_1}$ on $X^{\tau_1}_{M_1}$ and of $\T^{N_2}$ on $X^{\tau_2}_{M_2}$.
By \e_ref{Vset_e} and Lemma~\ref{moment_lmm}\ref{polytope_lmm}\ref{vertcoresp_lmm}, 
$$\mathscr{V}^{(\tau_1,\tau_2)}_{M_1\oplus M_2}=
   \mathscr{V}^{\tau_1}_{M_1} \times \mathscr{V}^{\tau_2}_{M_2}\,.$$
Thus, by Corollary~\ref{fixed_crl}\ref{fpt_crl}, the $\T^{N_1+N_2}$-fixed points of $X^{(\tau_1,\tau_2)}_{M_1\oplus M_2}$
are the points $([I_1],[I_2])$ for all $I_j\!\in\!\mathscr{V}^{\tau_j}_{M_j}$, 
with $[I_j]$ defined by \e_ref{fpt_e}.
By Corollary~\ref{fixed_crl}\ref{fc_crl} and the second statement in Lemma~\ref{moment_lmm}\ref{Mrestr_lmm}, the closed $\T^{N_1+N_2}$-fixed curves in $X^{(\tau_1,\tau_2)}_{M_1\oplus M_2}$ 
are all curves of the form $C_1\!\times\![I_2]$ and $[I_1]\!\times\!C_2$,
where $C_j$ is any closed $\T^{N_j}$-fixed curve in
$X^{\tau_j}_{M_j}$ and $I_j\!\in\!\mathscr{V}^{\tau_j}_{M_j}$
is arbitrary.

In particular, $\P^{N_1-1}\!\times\!\ldots\!\times\!\P^{N_s-1}$ 
is given by the minimal toric pair
\begin{equation}\label{prodproj_e}
\begin{BMAT}(0pt,20pt,10pt){c}{cc}
  \vphantom{\rule{3pt}{25pt}}\\
   s \text{ rows } 
   \end{BMAT}\!\!\!\!
   \begin{BMAT}(1pt,1pt,30pt){c}{cc}
    \vphantom{\rule{0mm}{-3pt}}\\
      \left\lbrace \vphantom{\rule{10pt}{39pt}} \right.
                  \end{BMAT} \!\!\!\!
   \begin{BMAT}(@,12pt,10pt){c}{cc}
   \begin{BMAT}(1pt,60pt,1pt){ccc}{c}
\,\,\,\overbrace{}^{N_1\text{ columns}}&
    \hphantom{\rule{2mm}{5pt}}&
      \overbrace{}^{N_s\text{ columns}}\,
   \end{BMAT}\\
   \left(
   \begin{BMAT}(@,12pt,10pt){c0c0c}{c0c0c0c0c}
\begin{BMAT}(0pt,1pt,1pt){ccccc}{c}1&1&\ldots&1&1\end{BMAT}& 
\begin{BMAT}(0pt,1pt,1pt){c}{c} \ldots\end{BMAT}&
  \begin{BMAT}(0pt,1pt,1pt){ccccc}{c}0&0&\ldots&0&0\end{BMAT}\\  
 \begin{BMAT}(0pt,1pt,1pt){ccccc}{c}0&0&\ldots&0&0\end{BMAT} &
 \begin{BMAT}(0pt,1pt,1pt){ccccc}{c}&\quad&\ldots&\quad&\end{BMAT}& 
 \begin{BMAT}(0pt,1pt,1pt){ccccc}{c}0&0&\ldots&0&0\end{BMAT}\\ 
  \vdots & \ddots &\vdots\\
   \begin{BMAT}(0pt,1pt,1pt){ccccc}{c}0&0&\ldots&0&0\end{BMAT}&&
\begin{BMAT}(0pt,1pt,1pt){ccccc}{c}0&0&\ldots&0&0\end{BMAT}\\
  \begin{BMAT}(0pt,1pt,1pt){ccccc}{c}0&0&\ldots&0&0\end{BMAT}&
\begin{BMAT}(0pt,1pt,1pt){c}{c} \ldots\end{BMAT}&
\begin{BMAT}(0pt,1pt,1pt){ccccc}{c}1&1&\begin{BMAT}(0pt,1pt,1pt){c}{c} \ldots\end{BMAT}&1&1\end{BMAT}
   \end{BMAT}
    \right)\end{BMAT},   
   \quad\tau=
    \begin{pmatrix}\tau_1\\\tau_2\\\vdots\\\tau_s\end{pmatrix}\in(\R^{>0})^s.
\end{equation}

By Proposition~\ref{equivcoh_prp}\ref{equivcohr_prp},
\begin{equation}\label{projcoh_e}
H^*_{\T^N}\!\left(\P^{N_1-1}\!\times\!\ldots\!\times\!\P^{N_s-1}\right)=
\frac{\Q[\al^{(i)}_j,1\!\le\!i\!\le\!s,1\!\le\!j\!\le\!N_i][x_1,\ldots,x_s]}
{\left(\prod\limits_{j=1}^{N_i}\left(x_i-\al^{(i)}_j\right),1\!\le\!i\!\le\!s\right)}.
\end{equation}
\end{eg}

\begin{rmk} Let $\si:[N]\lra[N]$ be a permutation
and
$(M,\tau)$ be a (minimal) toric pair.
Let $$M^{\si}\!\equiv\!\left(m_{i\si(j)}\right)_{\begin{subarray}{c}1\le i\le k\\1\le j\le N\end{subarray}}
\equiv M\!\circ\!\Id^{\si}$$
be the matrix obtained from $M$ by permuting its columns as dictated by $\si$.
Then $\left(M^{\si},\tau\right)$ is a (minimal) toric pair as well and $\Id^{\si^{-1}}$
induces a biholomorphism between $X_M^{\tau}$ and $X_{M^{\si}}^{\tau}$ 
(since $\mu_{M^{\si}}\!=\!\mu_M\!\circ\!\Id^{\si}$)
equivariant with respect to
$$\T^N\lra\T^N,\quad\left(t_1,t_2,\ldots,t_N\right)\lra\left(t_{\si(1)},t_{\si(2)},\ldots,t_{\si(N)}\right).$$
\end{rmk}
In particular, taking $k\!=\!0$ in Example~\ref{Hirzebruch_eg} gives - via \e_ref{prodproj_e} - $\P^1\!\times\!\P^1$ as expected, since the corresponding matrices differ by a permutation of columns.

\section{Explicit Gromov-Witten formulas}
\label{nequiv_sec}
For the remaining part of this paper, $\X$ is the compact projective manifold 
defined by \e_ref{toricman_e}, where $(M,\tau)$ is a minimal toric pair
as in Definition~\ref{toric_dfn}.
Theorem~\ref{Y_thm} in Section~\ref{stat_sec} below computes the one-point GW generating functions
$\Zp_{\eta}$ and $\Zpp_{\eta}$ of \e_ref{Zeta1ptdfn_e} if $\eta\!\in\!H^*(\X)$ is of the form
$\eta\!=\!\nH^{\bp}$, where $\{\nH_1,\ldots,\nH_k\}$ is the basis for $H^2(\X;\Z)$
referred to in Proposition~\ref{nequivcoh_prp} and
$$\nH^{\bp}\equiv\nH_1^{p_1}\ldots\nH_k^{p_k}\qquad\forall\,\bp=(p_1,\ldots,p_k)\!\in\!(\Z^{\ge0})^k.$$
We denote
\begin{equation}\label{Zp1pt_e}
\Zp_{\bp}\equiv \Zp_{\nH^{\bp}}\qquad\textnormal{and}\qquad \Zpp_{\bp}\equiv \Zpp_{\nH^{\bp}}.
\end{equation}
Section~\ref{construction_sec} constructs the explicit formal power series 
in terms of which $\Zp_{\bp}$ and $\Zpp_{\bp}$ are expressed in Theorem~\ref{Y_thm}.
Throughout this construction, which extends the constructions in 
\cite[Section~2.3]{bcov0} and \cite[Sections~2,3.1]{bcov0_ci} 
from $\P^{n-1}$ to an arbitrary toric manifold $\X$, we assume that 
\begin{equation}\label{nonnegnu_e}
\nu_E(\bd)\ge0\,\qquad\forall\,\bd\!\in\!\La,
\end{equation}
with $\nu_E$ as in \e_ref{L(d)_e}, and identify 
$$H_2(\X;\Z)\cong\Z^k$$ via the basis $\{\nH_1,\ldots,\nH_k\}$.
Via this identification $\La\!\hookrightarrow\!\Z^k$, with $\La$ as in \e_ref{La_e}.
\subsection{Notation and construction of explicit power series}
\label{construction_sec}

If $R$ is a ring, we denote by $R\Lau{\h}$ the ring of formal Laurent series in 
$\h^{-1}$ with finite principal part:
$$R\Lau{\h}\equiv R[[\h^{-1}]]+R[\h].$$
Given $f,g\!\in\!R\Lau{\h}$ and $s\!\in\!\Z^{\ge 0}$, we write
$$f\cong g ~~~\mod \h^{-s} \qquad\textnormal{if}\qquad 
f\!-\!g\in
R[\h]+\left\{\sum_{i=1}^{s-1}a_i\h^{-i}\!:a_i\!\in\!R~~\forall\,i\!\in\![s\!-\!1]\right\}.$$
If $R$ is a field, we view $R(\h)$ as a subring of $R\Lau{\h}$ by associating to
each element of $R(\h)$ its Laurent series at $\h^{-1}\!=\!0$.

With the line bundles $L^{\pm}_i$ as in \e_ref{splitE_e} and 
$\nU_j$ as in \e_ref{bdle_e} and $\bd\!\in\!H_2(\X;\Z)$, we define 
\begin{equation}\label{L(d)_e}\begin{split}
D_j(\bd)&\equiv\blr{\nU_j,\bd}, \qquad
L_i^{\pm}(\bd)\equiv\blr{c_1(L_i^{\pm}),\bd},\\
\nu_E(\bd)&\equiv\sum_{j=1}^ND_j(\bd)\!-\!\sum_{i=1}^aL_i^+(\bd)\!+\!\sum_{i=1}^bL_i^-(\bd)\,.\,\footnotemark
\end{split}\end{equation}
If in addition $Y\!\subset\!X$ is a one-dimensional submanifold,\footnotetext{\label{nuE_f}By \e_ref{splitE_e} and \e_ref{ctan_e},
$\nu_E(\bd)\!=\!\blr{c_1(T(E^-|_Y)),\bd}$ if 
$Y$ is a smooth complete intersection defined by a holomorphic section of~$E^+$
and $T(E^-|_Y)$ is the tangent bundle of the total space of $E^-|_Y$.} let
$$D_j(Y)\equiv D_j\left([Y]_{\X}\right),\qquad
L_i^{\pm}(Y)\equiv L_i^{\pm}\left([Y]_{\X}\right),$$
where $[Y]_{\X}\in H_2(\X;\Z)$ is the homology class represented by~$Y$.
By \e_ref{ctan_e}, our assumption \e_ref{nonnegnu_e}, and
Footnote~\ref{dual_f}, 
$$c_1(T\X)-\sum\limits_{i=1}^ac_1(L_i^+)+\sum\limits_{i=1}^bc_1(L_i^-)\in\ov\cK^{\tau}_M.$$
Thus, if $E\!\neq\!E^+$, then $\X$ is Fano.
In this case, the Cone Theorem~\cite[Theorem~1.5.33]{La}
implies that the closed $\R$-cone of curves is a polytope spanned by classes of rational curves.
By \cite[Proposition~1.4.28]{La} and \cite[Theorem~1.4.23(i)]{La}, this closed cone is
the $\R$-cone spanned by $\La$.
Thus, $L_i^-(\bd)\!<\!0$ for all $\bd\!\in\!\La\!-\!\{0\}$
and all $i\!\in\![b]$.\footnote{In the notation of \cite{La},
$N^1(\X)_{\R}\!=\!H^{1,1}(\X)\!\cap\!H^2(\X;\R)$ as can be seen
from Poincar\'{e} Duality,  Lefschetz Theorem on
$(1,1)$-classes, and Hard Lefschetz Theorem.}

Let $R$ be a ring.
Similarly to Section~\ref{neqres_sec}, 
we denote by  $R[[\La\!-\!0]]$ and $R[[\La;\nu_E\!=\!0]]$ 
the subalgebras of $R[[\La]]$ given~by 
\begin{equation*}\begin{split}
R[[\La\!-\!0]]&\equiv\left\{\sum_{\bd\in\La}a_{\bd}Q^{\bd}\in R[[\La]]:
a_{\b0}=0\right\}, \\
R[[\La;\nu_E\!=\!0]]&\equiv\left\{\sum_{\bd\in\La}a_{\bd}Q^{\bd}\in R[[\La]]:
a_{\bd}=0~\hbox{if}~\nu_E(\bd)\!\neq\!0\right\}.
\end{split}\end{equation*}
In some cases, the formal variables whose powers are indexed by $\La$ within $R[[\La]]$
will be denoted by $Q\!\equiv\!(Q_1,\ldots,Q_k)$ as in Section~\ref{neqres_sec}, 
while in other cases the formal variables will
be $q\!\equiv\!(q_1,\ldots,q_k)$.
If $f\!\in\!R[[\La]]$ and $\bd\!\in\!\La$, we write $\LR{f}_{q;\bd}\!\in\!R$
for the coefficient of $q^{\bd}$ in~$f$.
By Proposition~\ref{dimK_prp}, the set
$\{\bs\!\in\!\La:\bd\!-\!\bs\!\in\!\La\}$ is finite for every $\bd\!\in\!\La$;
thus,
$$f\in R[[\La]] ~~~\textnormal{is invertible} \qquad\Llra\qquad
\LR{f}_{q;\b0}\in R~~~\textnormal{is invertible}.$$
If $f\!\equiv\!\sum\limits_{\bd\in\La}f_{\bd}q^{\bd}\!\in\!R[[\La]]$, we define
$$\LR{f}_{\nu_E=0}\equiv\sum_{\begin{subarray}{c}\bd\in\La\\\nu_E(\bd)=0\end{subarray}}
\!\!\!f_{\bd}q^{\bd}\in R[[\La;\nu_E\!=\!0]].$$
Let $\nA\!\!=\!\!(\nA_1,\!\ldots\!,\nA_k)$ be a tuple of formal variables.
If $f\!\equiv\!\!\sum\limits_{\bd\in\La}\!\!f_{\bd}(\nA)q^{\bd}\!\!\in\!\!
R[[\nA]][[\La]]$ and $p\!\ge\!0$,
we write
$$\LR{f}_{\nA;p}\!\equiv\!\sum_{\bd\in\La}
\LR{f_{\bd}(\nA)}_{\nA;p}q^{\bd}\!\in\!R[\nA]_p[[\La]],$$
where $\LR{f_{\bd}(\nA)}_{\nA;p}\!\in\!R[\nA]$ denotes the degree~$p$ homogeneous part of $f_{\bd}(\nA)$
and $R[\nA]_p$ the space of homogeneous polynomials of degree~$p$ in $\nA_1,\ldots,\nA_k$ with coefficients in $R$.
Finally, we write
$$|\bp|\equiv p_1+p_2+\ldots+p_k\qquad\forall\,\bp=(p_1,p_2,\ldots,p_k)\!\in\!(\Z^{\ge 0})^k.$$

For each $\bd\!\in\!\La$, let
\begin{equation}\label{wtU_e1}
U(\bd;\nA,\h)\equiv
\frac{\prod\limits_{\begin{subarray}{c}j\in[N]\\D_j(\bd)<0\end{subarray}}
\prod\limits_{s=D_j(\bd)+1}^0\left(\sum\limits_{i=1}^km_{ij}\nA_i+s\h\right)}
{\prod\limits_{\begin{subarray}{c}j\in[N]\\D_j(\bd)\ge 0\end{subarray}}
\prod\limits_{s=1}^{D_j(\bd)}\left(\sum\limits_{i=1}^km_{ij}\nA_i+s\h\right)}
\in\Q[\nA]\Lau{\h}\,.
\end{equation}
By Proposition~\ref{Pic_prp},
the line bundles $\ga_i^*$ of \e_ref{bdle_e} form a basis for the Picard group of $\X$.
Thus, there are well-defined integers $\ell^{\pm}_{ri}$ such that
\begin{equation}\label{Etensor_e}
L_i^{\pm}=\ga_1^{*\ell^{\pm}_{1i}}\otimes\ldots\otimes\ga_k^{*\ell^{\pm}_{ki}}. 
\end{equation}
With $\nA$ and $\bd$ as above, let
\begin{equation}\label{wtU_e2}\begin{split}
\Ep(\bd;A,\h)&\equiv
\prod_{i=1}^a\prod_{s=1}^{L_i^+(\bd)}\left(\sum_{r=1}^k\ell^+_{ri}\nA_r+s\h\right)
\prod_{i=1}^b\prod_{s=0}^{-L_i^-(\bd)-1}\left(\sum_{r=1}^k\ell^-_{ri}\nA_r-s\h\right)
\in \Z[A,\h],\\
\Epp(\bd;A,\h)&\equiv
\prod_{i=1}^a\prod_{s=0}^{L_i^+(\bd)-1}\left(\sum_{r=1}^k\ell^+_{ri}\nA_r+s\h\right)
\prod_{i=1}^b\prod_{s=1}^{-L_i^-(\bd)}\left(\sum_{r=1}^k\ell^-_{ri}\nA_r-s\h\right)
\in \Z[A,\h]. 
\end{split}\end{equation}
The formal power series computing $\Zp_{\bp}$ and $\Zpp_{\bp}$ in
Theorem~\ref{Y_thm} are obtained from
\begin{equation}\begin{split}\label{Y'_e}
\Yp(\nA,\h,q)&\equiv\sum_{\bd\in\La}q^{\bd}
U(\bd;\nA,\h)\Ep(\bd;\nA,\h)\in \Q[A]\big[\big[\h^{-1},\La\big]\big],\\
\Ypp(\nA,\h,q)&\equiv\sum_{\bd\in\La}q^{\bd}
U(\bd;\nA,\h)\Epp(\bd;\nA,\h)\in \Q[A]\big[\big[\h^{-1},\La\big]\big].
\end{split}\end{equation}
We define
\begin{equation}\label{I0_e}
\Ip_{\b0}(q)\equiv \Yp(\nA,\h,q)~~~\mod\,\h^{-1},\qquad 
\Ipp_{\b0}(q)\equiv \Ypp(\nA,\h,q)~~~\mod\,\h^{-1},
\end{equation}
and so
\begin{equation*}\begin{split}
\Ip_{\b0}(q) &\equiv 1+\de_{b,0}\!\!\!\!\!\!\!\!\!\!
\sum_{\begin{subarray}{c}\bd\in\La-0,~\nu_E(\bd)=0\\
  D_j(\bd)\ge 0~\forall\,j\in[N]\end{subarray}}
\!\!\!\! q^{\bd}\frac{\prod\limits_{i=1}^a\left(L^+_i(\bd)!\right)}
{\prod\limits_{j=1}^N\left(D_j(\bd)!\right)},\\
\Ipp_{\b0}(q)&\equiv 1+\!\!
\sum_{\begin{subarray}{c}\bd\in\La-0,~\nu_E(\bd)=0\\
D_j(\bd)\ge 0~\forall\,j\in[N]\\
L^+_i(\bd)=0\,\forall\,i\in[a]
\end{subarray}}\!\!\!\!q^{\bd}
(-1)^{\sum\limits_{i=1}^bL_i^-(\bd)}
\frac{\prod\limits_{i=1}^b\left(\left(-L_i^-(\bd)\right)!\right)}
{\prod\limits_{j=1}^N\left(D_j(\bd)!\right)}.
\end{split}\end{equation*}

We next describe an operator $\bD^{\bp}$ acting on a subset of $\Q(\nA,\h)[[\La]]$
and certain associated ``structure coefficients" in $\Q[[\La]]$ which occur in the formulas 
for $\Zp_{\bp}$ and~$\Zpp_{\bp}$.
Fix an element $Y(\nA,\h,q)$ of $\Q(\nA,\h)[[\La]]$ such~that for all $\bd\!\in\!\La$
$$\LR{Y(\nA,\h,q)}_{q;\bd}\equiv\frac{f_{\bd}(\nA,\h)}{g_{\bd}(\nA,\h)}$$
for some homogeneous polynomials $f_{\bd}(\nA,\h),g_{\bd}(\nA,\h)\!\in\!\Q[\nA,\h]$
satisfying
\begin{equation}\label{assumeS_e}
f_{\b0}(\nA,\h)=g_{\b0}(\nA,\h)\,, \qquad
\deg\,f_{\bd}-\deg\,g_{\bd}=-\nu_E(\bd),\qquad
g_{\bd}\big|_{\nA=0}\neq 0\quad\forall\,\bd\!\in\!\La.
\end{equation}
This condition is  satisfied by the power series $\Yp$ and $\Ypp$ 
of \e_ref{Y'_e} and so the construction below applies to $Y\!=\!\Yp$ and $Y\!=\!\Ypp$.
We inductively define
\begin{equation*}
J_p(Y)\in\End_{\Q[[\La;\nu_E=0]]}\big(\Q[[\La;\nu_E\!=\!0]][\nA]_p\big)
~~\forall\,p\!\in\!\Z^{\ge 0}\,, \quad
\bD^{\bp}Y(\nA,\h,q)\in\Q(\nA,\h)[[\La]]~~\forall\,\bp\!\in\!(\Z^{\ge 0})^k
\end{equation*}
satisfying 
\begin{enumerate}[label=(P\arabic*),leftmargin=*] 
\item\label{P_J_pq;0} for every $\bp\!\in\!(\Z^{\ge 0})^k$ with $|\bp|\!=\!p$,
$\LR{\{J_p(Y)\}(\nA^{\bp})}_{q;\b0}=\nA^{\bp}$;
\item\label{P_Dpexp} there exist $\nC^{(\br)}_{\bp,s}\!\equiv\!\nC^{(\br)}_{\bp,s}(Y)\in\Q[[\La]]$
with $\bp,\br\!\in\!(\Z^{\ge0})^k$ and $s\!\in\!\Z^{\ge0}$ such that 
\begin{gather}
\label{bDpexp_e}
\bD^{\bp}Y(\nA,\h,q)=\h^{|\bp|}
\sum_{s=0}^{\i}\sum_{\br\in(\Z^{\ge0)^k}}\!\!\!
\nC^{(\br)}_{\bp,s}(q)\nA^{\br}\h^{-s}\,,\\
\label{degnC_e}
\LR{\nC^{(\br)}_{\bp,s}}_{q;\bd}=0~\textnormal{if}~s\!\neq\!\nu_E(\bd)\!+\!|\br|,~
\LR{\nC^{(\br)}_{\bp,s}}_{\nu_E=0}
=\de_{\bp,\br}\de_{|\br|,s}~\textnormal{if}~s\!\le\!|\bp|,~
\LR{\nC^{(\br)}_{\bp,|\br|}}_{q;\b0}=\de_{\bp,\br}.
\end{gather}
\end{enumerate}
By \e_ref{assumeS_e}, we can define $J_0(Y)\in\Q[[\La;\nu_E\!=\!0]]$ 
and $\bD^{\b0}Y\!\in\!\Q(\nA,\h)[[\La]]$ by
\begin{equation}\label{bD0dfn_e}
\{J_0(Y)\}(1)\equiv Y(\nA,\h,q)~~~\mod\,\h^{-1},\qquad
\bD^{\b0}Y(\nA,\h,q)\equiv\left[\{J_0(Y)\}(1)\right]^{-1}Y(\nA,\h,q).
\end{equation}

Suppose next that $p\!\ge\!0$ and we have constructed an operator $J_p(Y)$ 
and power series $\bD^{\bp'}Y$ for all $\bp'\!\in\!(\Z^{\ge 0})^k$ with $|\bp'|\!=\!p$
satisfying the above properties.
For each $\bp\!\in\!(\Z^{\ge 0})^k$ with $|\bp|\!=\!p\!+\!1$, let
\begin{equation}\begin{split}\label{tiDpdfn_e}
\wt{\bD}^{\bp}Y(\nA,\h,q)&\equiv
\frac{1}{|\supp(\bp)|}\sum\limits_{i\in\supp(\bp)}\left\{\nA_i+\h\,q_i\frac{\nd}{\nd q_i}\right\}
\bD^{\bp-e_i}Y(\nA,\h,q)\in \Q(\nA,\h)[[\La]],\\
\{J_{p+1}(Y)\}(\nA^{\bp})&\equiv
\LR{\wt{\bD}^{\bp}Y(\nA,\h,q)~~\mod\,\h^{-1}}_{\nA;p+1},
\end{split}\end{equation}
where $\{e_1,\ldots,e_k\}$ is the canonical basis for $\Z^k$.
By \ref{P_Dpexp}, 
\begin{equation}\label{expressJ_e}\begin{split}
\{J_{p+1}(Y)\}(\nA^{\bp})&=\frac{1}{|\supp(\bp)|}
\sum_{i\in\supp(\bp)}\left[\sum_{|\br|=p}\nC^{(\br)}_{\bp-e_i,p}\nA_i\!\cdot\!\nA^{\br}+
\sum_{|\br|=p+1}q_i\frac{\nd \nC^{(\br)}_{\bp-e_i,p+1}}{\nd q_i}\nA^{\br}\right]\\
&=\frac{1}{|\supp(\bp)|}
\sum_{|\br|=p+1}\left[
\sum_{i\in\supp(\bp)}\left(\nC^{(\br-e_i)}_{\bp-e_i,p}
+q_i\frac{\nd \nC^{(\br)}_{\bp-e_i,p+1}}{\nd q_i}\right)\right]\nA^{\br},
\end{split}\end{equation}
where we set $\nC^{(\br-e_i)}_{\bp-e_i,p}\!\equiv\!0$ if $i\!\not\in\!\supp(\br)$.
By \e_ref{expressJ_e} and \e_ref{degnC_e},
\begin{equation}\label{Jpprop_e}
\{J_{p+1}(Y)\}(\nA^{\bp})\in\Q[[\La;\nu_E\!=\!0]][\nA]_{p+1}
\quad\hbox{and}\quad
\LR{\{J_{p+1}(Y)\}(\nA^{\bp})}_{q;\b0}=\nA^{\bp}\,; 
\end{equation}
in particular, $J_{p+1}(Y)$ is invertible.
With $\nc_{\bp;\bp'}(q)\!\in\!\Q[[\La;\nu_E\!=\!0]]$ for $\bp,\bp'\!\in\!(\Z^{\ge0})^k$ 
with $|\bp|,|\bp'|\!=\!p\!+\!1$ given~by
\begin{equation}\label{Jpinv_e}
\{J_{p+1}(Y)\}^{-1}(\nA^{\bp})\equiv
\sum_{\bp'\in(\Z^{\ge0})^k,|\bp'|=p+1}\!\!\!\!\!\!\nc_{\bp;\bp'}(q)\nA^{\bp'},
\end{equation}
we define
\begin{equation}\label{bDdfn_e}
\bD^{\bp}Y(\nA,\h,q)\equiv
\sum_{\bp'\in(\Z^{\ge0})^k,|\bp'|=p+1}\!\!\!\!\!\!
\nc_{\bp;\bp'}(q)\wt{\bD}^{\bp'}Y(\nA,\h,q).
\end{equation}
By \e_ref{bDdfn_e} and the inductive assumption \e_ref{bDpexp_e}, 
\begin{gather}
\bD^{\bp}Y(\nA,\h,q)=\h^{p+1}
\sum_{s=0}^{\i}\sum_{\br\in(\Z^{\ge0)^k}}\!\!\!
\nC^{(\br)}_{\bp,s}(q)\nA^{\br}\h^{-s}\,, \qquad\hbox{where}\notag\\
\label{nCrec_e}
\nC^{(\br)}_{\bp,s}=\sum_{\bp'\in(\Z^{\ge0})^k,|\bp'|=p+1}\!
 \frac{\nc_{\bp;\bp'}}{|\supp(\bp')|}
\sum_{i\in\supp(\bp')}\left(\nC^{(\br-e_i)}_{\bp'-e_i,s-1}
+q_i\frac{\nd \nC^{(\br)}_{\bp'-e_i,s}}{\nd q_i}\right),
\end{gather}
where we set $\nC^{(\br-e_i)}_{\bp'-e_i,s-1}\!=\!0$
if $i\!\notin\!\supp(\br)$ or $s\!=\!0$.
By the first property in \e_ref{degnC_e} with $\bp$ replaced by $\bp'\!-\!e_i$
with $|\bp'|\!=\!p\!+\!1$ and $i\!\in\!\supp(\bp')$,  
$\nC^{(\br)}_{\bp,s}$ satisfies this property as well.
By the second property in~\e_ref{degnC_e} with $\bp$ replaced by $\bp'\!-\!e_i$
with $|\bp'|\!=\!p\!+\!1$ and $i\!\in\!\supp(\bp')$,
$\LR{\nC^{(\br)}_{\bp,s}}_{\nu_E=0}\!=\!0$ if $s\!\le\!p$.
Since $\nC^{(\br)}_{\bp,p+1}\!=\!\de_{\bp,\br}$ whenever $|\br|\!=\!p\!+\!1$ by~\e_ref{nCrec_e} and~\e_ref{Jpinv_e},
$\nC^{(\br)}_{\bp,s}$ also satisfies the second property in~\e_ref{degnC_e}.
By the second statement in~\e_ref{Jpprop_e} and~\e_ref{Jpinv_e},
$\LR{\nc_{\bp;\bp'}}_{q;\b0}\!=\!\de_{\bp,\bp'}$.
Thus, by the third property in~\e_ref{degnC_e} with $\bp$ replaced by $\bp'\!-\!e_i$
with $|\bp'|\!=\!p\!+\!1$ and $i\!\in\!\supp(\bp')$,
$\nC^{(\br)}_{\bp,s}$ satisfies the last property in~\e_ref{degnC_e} as well.

Define $\wt\nC^{(r)}_{\bp,\bs}\!\equiv\!\wt{\nC}^{(r)}_{\bp,\bs}(Y)\!\in\!\Q[[\La]]$
for $\bp,\bs\!\in\!(\Z^{\ge0})^k$ and $r\!\in\!\Z^{\ge0}$ with $|\bs|\!\le\!|\bp|\!-\!r$
and $r\!\le\!|\bp|$ by
\begin{equation}\label{neqtiC_e}
\sum_{t=0}^r\sum_{\begin{subarray}{c}\bs\in(\Z^{\ge0})^k\\ |\bs|\le|\bp|-t\end{subarray}}\!\!
\wt\nC^{(t)}_{\bp,\bs}\nC_{\bs,|\br|+r-t}^{(\br)}
=\de_{\bp,\br}\de_{r,0}
\qquad\forall~\br\!\in\!(\Z^{\ge0})^k,\,|\br|\le|\bp|\!-\!r.
\end{equation}
Equations \e_ref{neqtiC_e} indeed uniquely determine $\wt{\nC}^{(r)}_{\bp,\bs}$,
since 
\begin{equation}\label{neqtiC_e2}
\sum_{t=0}^r\sum_{\begin{subarray}{c}\bs\in(\Z^{\ge0})^k\\ |\bs|\le|\bp|-t\end{subarray}}\!\!
\wt\nC^{(t)}_{\bp,\bs}\nC_{\bs,|\br|+r-t}^{(\br)}
=\sum_{t=0}^{r-1}\sum_{\begin{subarray}{c}\bs\in(\Z^{\ge0})^k\\ |\bs|\le|\bp|-t\end{subarray}}
\!\!\wt{\nC}^{(t)}_{\bp,\bs}\nC_{\bs,|\br|+r-t}^{(\br)}
+\sum_{\begin{subarray}{c}\bs\in(\Z^{\ge0})^k\\ |\bs|<|\br|\end{subarray}}
\wt{\nC}^{(r)}_{\bp,\bs}\nC_{\bs,|\br|}^{(\br)}
+\wt{\nC}^{(r)}_{\bp,\br},
\end{equation}
as follows from~\e_ref{degnC_e}.
By~\e_ref{neqtiC_e} together with the first and third statements in \e_ref{degnC_e}, 
\begin{equation}\label{neqtiCq;0_e}
\LR{\wt{\nC}^{(r)}_{\bp,\bs}}_{q;\b0}=\de_{\bp,\bs}\de_{r,0}\,.
\end{equation}
By \e_ref{neqtiC_e}, \e_ref{neqtiC_e2}, and induction on $|\bs|$,
\begin{equation}\label{tiC0_e}
\wt{\nC}^{(0)}_{\bp,\bs}(q)=\de_{\bp,\bs}\qquad\forall\,\bp,\bs\!\in\!(\Z^{\ge 0})^k
\quad\textnormal{with}
\quad|\bs|\!\le\!|\bp|.
\end{equation}
By \e_ref{neqtiC_e}, \e_ref{neqtiC_e2}, the first statement in~\e_ref{degnC_e},  
and induction on $|\bs|$ and $r$,
\begin{equation}\label{neqtiCdeg_e}
\LR{\wt{\nC}^{(r)}_{\bp,\bs}}_{q;\bd}=0\qquad\textnormal{if}\quad \nu_E(\bd)\neq r.
\end{equation}

\begin{rmk}\label{J0Y'_rmk}
With $\Yp$, $\Ypp$ as in \e_ref{Y'_e} and $\Ip_{\b0}$, $\Ipp_{\b0}$ as in \e_ref{I0_e},
\begin{alignat*}{2}
\left\{J_0(\Yp)\right\}(1)&=\Ip_{\b0}(q),&\qquad 
\bD^{\b0}\Yp(\nA,\h,q)&=\frac{1}{\Ip_{\b0}(q)}\!\Yp(\nA,\h,q),\\
\left\{J_0(\Ypp)\right\}(1)&=\Ipp_{\b0}(q),&\qquad 
\bD^{\b0}\Ypp(\nA,\h,q)&=\frac{1}{\Ipp_{\b0}(q)}\!\Ypp(\nA,\h,q),
\end{alignat*}
by \e_ref{bD0dfn_e}.
\end{rmk}

Define
\begin{equation*}
\left\{\nA+\h\,q\frac{\nd}{\nd q}\right\}^{\bp}\equiv\left\{\nA_1\!+\!\h\,q_1\frac{\nd}{\nd q_1}\right\}^{p_1}\ldots
\left\{\nA_k\!+\!\h\,q_k\frac{\nd}{\nd q_k}\right\}^{p_k}\qquad\forall\,\bp=(p_1,\ldots,p_k)\!\in\!(\Z^{\ge 0})^k.
\end{equation*}

\begin{rmk}\label{simpleDp_rmk}
If $\nu_{E}(\bd)\!>\!0$ for all $\bd\!\in\!\La\!-\!\{0\}$ and 
$Y(\nA,\h,q)\!\in\!\Q(\nA,\h)[[\La]]$ satisfies \e_ref{assumeS_e}, then
$J_p(Y)\!=\!\Id$ for all $p\!\in\!\Z^{\ge0}$ by \ref{P_J_pq;0} above.
Along with the first equation in~\e_ref{tiDpdfn_e}, \e_ref{Jpinv_e}, \e_ref{bDdfn_e}, 
and induction on~$|\bp|$, this implies that 
\begin{equation*}
\bD^{\bp}Y(\nA,\h,q)=\left\{\nA\!+\!\h\,q\frac{\nd}{\nd q}\right\}^{\bp}Y(\nA,\h,q)
\end{equation*}
for all $\bp\!\in\!(\Z^{\ge 0})^k$. 
\end{rmk}

\begin{rmk}\label{partDp_rmk}
Suppose $p^*\!\in\!\Z^{\ge0}$, $Y(\nA,\h,q)\!\in\!\Q(\nA,\h)[[\La]]$
satisfies \e_ref{assumeS_e}, and
$$\deg_{\h}f_{\bd}(\nA,\h)-\deg_{\h}g_{\bd}(\nA,\h)<-p^*\qquad\forall\,\bd\!\in\!\La\!-\!\{0\}.$$
By the same reasoning as in Remark~\ref{simpleDp_rmk}, 
we again find that
$$J_p(Y)=\Id, \qquad
\bD^{\bp}Y(\nA,\h,q)\equiv\left\{\nA\!+\!\h\,q\frac{\nd}{\nd q}\right\}^{\bp}Y(\nA,\h,q),$$
for all $p\!\in\!\Z^{\ge0}$ and $\bp\!\in\!(\Z^{\ge 0})^k$ such that 
$p,|\bp|\le\!p^*$.
\end{rmk}

\begin{rmk}\label{projDp_rmk}
Let $(M,\tau)$ be the toric pair of Example~\ref{proj_eg} with 
$N\!=\!n$ so that $\X\!=\!\P^{n-1}$. 
Let $$E\equiv\bigoplus\limits_{i=1}^{a}
\cO_{\P^{n-1}}(\ell^+_i)\oplus\bigoplus\limits_{i=1}^{b}\cO_{\P^{n-1}}(\ell^-_i)$$
with $a,b\!\ge\!0$, $\ell^+_i\!>\!0$ for all $i\!\in\![a]$, $\ell^-_i\!<\!0$ for all $i\!\in\![b]$, 
and $\sum\limits_{i=1}^{a}\ell^+_i\!-\!\sum\limits_{i=1}^{b}\ell^-_i\!\le\!n$.
Thus, 
$$\nu_E(d)=\left(n-\sum_{i=1}^{a}\ell^+_i+\sum_{i=1}^{b}\ell^-_i\right)d$$ 
for all $d\!\in\!\Z^{\ge 0}$.
By \e_ref{Y'_e},
\begin{equation*}\begin{split}
\Yp(\nA,\h,q)&=\sum_{d=0}^{\i}q^d
\frac{\prod\limits_{i=1}^a\prod\limits_{s=1}^{\ell^+_id}\left(\ell^+_i\nA\!+\!s\h\right)
\prod\limits_{i=1}^b\prod\limits_{s=0}^{-\ell_i^-d-1}\left(\ell_i^-\nA\!-\!s\h\right)}
{\prod\limits_{s=1}^d\left(\nA\!+\!s\h\right)^n},\\
\Ypp(\nA,\h,q)&=\sum_{d=0}^{\i}q^d
\frac{\prod\limits_{i=1}^a\prod\limits_{s=0}^{\ell^+_id-1}\left(\ell^+_i\nA\!+\!s\h\right)
\prod\limits_{i=1}^b\prod\limits_{s=1}^{-\ell_i^-d}\left(\ell_i^-\nA\!-\!s\h\right)}
{\prod\limits_{s=1}^d\left(\nA\!+\!s\h\right)^n}.
\end{split}\end{equation*}
By Remark~\ref{partDp_rmk},
\begin{alignat*}{3}
J_p(\Yp)&=\Id, &\quad
\bD^{p}\Yp(\nA,\h,q)&=\left\{\nA\!+\!\h\,q\frac{\nd}{\nd q}\right\}^p\Yp(\nA,\h,q)  
&\qquad &\forall\,p\!<\!b,\\
J_p(\Ypp)&=\Id, &\quad
\bD^{p}\Ypp(\nA,\h,q)&=\left\{\nA\!+\!\h\,q\frac{\nd}{\nd q}\right\}^p\Ypp(\nA,\h,q)
&\qquad &\forall\,p\!<\!a.
\end{alignat*}
If $\sum\limits_{i=1}^a\ell^+_i\!-\!\sum\limits_{i=1}^b\ell^-_i\!<\!n$, then
$$J_p(\Yp),J_p(\Ypp)\!=\!\Id, \quad
\bD^{p}\Yp=\left\{\nA\!+\!\h\,q\frac{\nd}{\nd q}\right\}^p\Yp,\quad
\bD^{p}\Ypp=\left\{\nA\!+\!\h\,q\frac{\nd}{\nd q}\right\}^p\Ypp$$
for all $p$ by Remark~\ref{simpleDp_rmk}.
If $\sum\limits_{i=1}^a\ell^+_i\!-\!\sum\limits_{i=1}^b\ell^-_i\!=\!n$, 
then we follow \cite[(1.1)]{bcov0_ci} and set
\begin{equation}\begin{split}\label{FIpproj_e}
F(\nw,q)&\equiv\sum_{d=0}^{\i}q^d\frac{\prod\limits_{i=1}^a\prod\limits_{r=1}^{\ell^+_id}\left(\ell^+_i\nw\!+\!r\right)
\prod\limits_{i=1}^b\prod\limits_{r=1}^{-\ell^-_id}\left(\ell_i^-\nw\!-\!r\right)}
{\prod\limits_{r=1}^d\left(\nw\!+\!r\right)^n},\\
\bM F(\nw,q)&\equiv\left\{1\!+\!\frac{q}{\nw}\frac{\nd}{\nd q}\right\}\left(\frac{F(\nw,q)}{F(0,q)}\right),\quad
I_p(q)\equiv\bM^pF(0,q).
\end{split}\end{equation}
By \e_ref{tiDpdfn_e} and \e_ref{bDdfn_e} above,
\begin{alignat*}{3}
J_p(\Yp)&=I_{p-b}(q)\Id, &\quad
\bD^{p}\Yp(\nA,\h,q)&=\nA^p\frac{1}{I_{p-b}(q)}\bM^{p-b}F\left(\frac{\nA}{\h},q\right)
&\qquad &\forall\,p\!\ge\!b,\\
J_p(\Ypp)&=I_{p-a}(q)\Id, &\quad
\bD^{p}\Ypp(\nA,\h,q)&=\nA^p\frac{1}{I_{p-a}(q)}\bM^{p-a}F\left(\frac{\nA}{\h},q\right)
&\qquad &\forall\,p\!\ge\!a.
\end{alignat*}
\end{rmk}

\subsection{Statements}
\label{stat_sec}
The statements and proofs of the theorems below rely on the
one-point mirror formula \e_ref{Z'Y'_e} below, which is proved in \cite{LLY3}.
We begin by defining the mirror map occurring in this formula.

For each $i\!\in\![k]$, let
\begin{equation}\label{fi_e}
f_i(q)\!\equiv\!\frac{1}{\Ip_{\b0}(q)}\!\!\!\!\sum\limits_{\begin{subarray}{c}\bd\in\La\\\nu_E(\bd)=0\end{subarray}}\!\!\!\!\!q^{\bd}\frac{\part\left\{U(\bd;\nA,1)\Ep(\bd;\nA,1)\right\}}{\part \nA_i}\Big|_{\nA=0}\!\in\!\Q[[\La\!-\!0]],
\end{equation}
with $\Ip_{\b0}(q)$ defined by \e_ref{I0_e}.
The \sf{mirror map} is the change of variables $\!q\lra\!Q$, where
\begin{equation}\label{mirrmap_e}
(Q_1,\ldots,Q_k)=\left(q_1\ne^{f_1(q)},\ldots,q_k\ne^{f_k(q)}\right).
\end{equation}
Finally, let
\begin{equation}\label{G_e}
G(q)\!\equiv\!\frac{\de_{b,0}}{\Ip_{\b0}(q)}\!\!\!\!\sum_{\begin{subarray}{c}\bd\in\La,\,\,\nu_E(\bd)=1\\D_j(\bd)\ge 0\,\forall\,j\in[N]\end{subarray}}\!\!\!\!\!\!
q^{\bd}\frac{\prod\limits_{i=1}^a\left(L^+_i(\bd)!\right)}
{\prod\limits_{j=1}^N\left(D_j(\bd)!\right)}\!\in\!\Q[[\La\!-\!0]].
\end{equation}
\begin{thm}\label{Y_thm}
If $\nu_E(\bd)\!\ge\!0$ for all $\bd\!\in\!\La$, then $\Zp_{\bp}$ and $\Zpp_{\bp}$ of \e_ref{Zp1pt_e} and \e_ref{Zeta1ptdfn_e} are given by
\begin{equation*}\begin{split}
\Zp_{\bp}(\h,Q)&=\ne^{-\frac{1}{\h}\left[G(q)+\sum\limits_{i=1}^k\nH_if_i(q)\right]}
\Yp_{\bp}(\nH,\h,q)\!\in\!H^*(\X)[\h^{-1}][[\La]],\\
\Zpp_{\bp}(\h,Q)&=\ne^{-\frac{1}{\h}\left[G(q)+\sum\limits_{i=1}^k\nH_if_i(q)\right]}
\Ypp_{\bp}(\nH,\h,q)\!\in\!H^*(\X)[\h^{-1}][[\La]],
\end{split}\end{equation*}
where 
\begin{equation}\begin{split}\label{Yp_e}
\Yp_{\bp}(\nA,\h,q)&\equiv\bD^{\bp}\Yp(\nA,\h,q)+\sum_{r=1}^{|\bp|}\sum_{|\bs|=0}^{|\bp|-r}\tinCp^{(r)}_{\bp,\bs}(q)\h^{|\bp|-r-|\bs|}
\bD^{\bs}\Yp(\nA,\h,q)\!\in\!\Q(\nA,\h)[[\La]],\\
\Ypp_{\bp}(\nA,\h,q)&\equiv\bD^{\bp}\Ypp(\nA,\h,q)+\sum_{r=1}^{|\bp|}\sum_{|\bs|=0}^{|\bp|-r}\tinCpp^{(r)}_{\bp,\bs}(q)\h^{|\bp|-r-|\bs|}
\bD^{\bs}\Ypp(\nA,\h,q)\!\in\!\Q(\nA,\h)[[\La]],
\end{split}\end{equation}
with 
$\tinCp^{(r)}_{\bp,\bs}\equiv\tinC^{(r)}_{\bp,\bs}(\Yp)$ and $\tinCpp^{(r)}_{\bp,\bs}\equiv\wt{\nC}^{(r)}_{\bp,\bs}(\Ypp)$
defined by \e_ref{neqtiC_e},
$Q$ and $q$ related by the mirror map \e_ref{mirrmap_e}, $G$ and $f_i$ given by \e_ref{G_e} and \e_ref{fi_e},
and the operator $\bD^{\bp}$ defined by \e_ref{tiDpdfn_e} and \e_ref{bDdfn_e}.

If $|\bp|\!<\!b$, $\bD^{\bp}\Yp=\{\nA\!+\!\h\,q\frac{\nd}{\nd q}\}^{\bp}\Yp$ and 
$\tinCp^{(r)}_{\bp,\bs}\!=\!0$ for all
$r\!\in\![|\bp|]$.
If $|\bp|\!<\!a$ and $L_i^+(\bd)\!\ge\!1$ for all $i\!\in\![a]$ and $\bd\!\in\!\La\!-\!\{0\}$,
then $\bD^{\bp}\Ypp=\{\nA\!+\!\h\,q\frac{\nd}{\nd q}\}^{\bp}\Ypp$ and
$\tinCpp^{(r)}_{\bp,\bs}\!=\!0$ for all
$r\!\in\![|\bp|]$.
\end{thm} 
This follows from Theorem~\ref{cY_thm} together with \e_ref{ctCvstinC_e} and \ref{EP_noneq} below;
Theorem~\ref{cY_thm} is an equivariant version of Theorem~\ref{Y_thm}.
\begin{rmk}\label{weightedbD_rmk}
In the inductive construction of $\bD^{\bp}Y$ with $Y=\Yp$ or $Y=\Ypp$, the first equation in \e_ref{tiDpdfn_e}
may be replaced by
$$
\wt{\bD}^{\bp}Y(\nA,\h,q)\!\equiv\!\sum_{i\in\supp(\bp)}c_{\bp;i}\left\{\nA_i\!+\!\h\,q_i\frac{\nd}{\nd q_i}\right\}
\bD^{\bp-e_i}Y(\nA,\h,q)\!\in\!\Q(\nA,\h)[[\La]]$$
for any tuple $(c_{\bp;i})_{i\in\supp(\bp)}$ of rational numbers with $\sum\limits_{i\in\supp(\bp)}c_{\bp;i}\!=\!1$.
The endomorphism $J_{p+1}(Y)$ and the power series
$\bD^{\bp}Y$ defined by the second equation in \e_ref{tiDpdfn_e} and \e_ref{bDdfn_e} in terms of the new ``weighted''
$\wt{\bD}^{\bp}Y$ satisfy \ref{P_J_pq;0} and \ref{P_Dpexp} by the same arguments as
in the case when
$c_{\bp;i}\!=\!\frac{1}{|\supp(\bp)|}$ for all $i\!\in\!\supp(\bp)$.
Therefore, \e_ref{neqtiC_e} continues to define power series $\wt{\nC}^{(r)}_{\bp,\bs}(Y)$
in terms of the ``new weighted'' $\nC^{(\br)}_{\bp,\bs}(Y)$.
The resulting ``weighted'' power series $Y_{\bp}$ of \e_ref{Yp_e}
do not depend on the ``weights'' $c_{\bp;i}$ as elements of $H^*(\X)\Lau{\h}[[\La]]$; this follows from Remark~\ref{weightedfD_rmk}. 
\end{rmk}

\begin{crl}\label{nuE0_crl}
If $\nu_E(\bd)\!=\!0$ or $\nu_E(\bd)\!>\!|\bp|$ for all $\bd\!\in\!\La\!-\!\{0\}$, then
$$\Zp_{\bp}(\h,Q)=\ne^{-\frac{1}{\h}\left[G(q)+\sum\limits_{i=1}^k\nH_if_i(q)\right]}\!\bD^{\bp}\Yp(\nH,\h,q),\quad
\Zpp_{\bp}(\h,Q)=\ne^{-\frac{1}{\h}\left[G(q)+\sum\limits_{i=1}^k\nH_if_i(q)\right]}\!\bD^{\bp}\Ypp(\nH,\h,q),
$$
with $Q$ and $q$ related by the mirror map \e_ref{mirrmap_e} and $G$ and $f_i$ given by \e_ref{G_e} and \e_ref{fi_e}.
\end{crl}
This follows from Theorem~\ref{Y_thm} together with \e_ref{neqtiCdeg_e}.

Let $\pr_i\!:\!\X\!\times\!\X\!\lra\!\X$ denote the projection onto the $i$-th component.
\begin{crl}\label{2pt_crl}
Let $g_{\bp\bs}\!\in\!\Q$ be such that $\sum\limits_{|\bp|+|\bs|=N-k}\!\!\!g_{\bp\bs}\pr_1^*\nH^{\bp}\pr_2^*\nH^{\bs}$
is the Poincar\'{e} dual to the diagonal class in $\X$, where $N\!-\!k$ is the complex dimension of $\X$.
If $N\!>\!k$ and $\nu_E(\bd)\!>\!N\!-\!k$ for all $\bd\!\in\!\La\!-\!\{0\}$, then the two-point function $\Zp$
of \e_ref{Z2ptdfn_e} is given by
\begin{equation*}
\Zp(\h_1,\h_2,q)=\frac{1}{\h_1+\h_2}\sum_{|\bp|+|\bs|=N-k}\!\!\!\!
\!\!g_{\bp\bs}
\pr_1^*\left\{\nH\!+\!\h_1\,q\frac{\nd}{\nd q}\right\}^{\bp}\Yp(\nH,\h_1,q)\pr_2^*
\left\{\nH\!+\!\h_2\,q\frac{\nd}{\nd q}\right\}^{\bs}\Ypp(\nH,\h_2,q).
\end{equation*}
\end{crl}
This follows from Theorem~\ref{Z2pt_thm}, Corollary~\ref{nuE0_crl}, and Remark~\ref{simpleDp_rmk}.
\begin{rmk}\label{untwZ_rmk2}
If 
$$P(\nA)\equiv\frac{\prod\limits_{i=1}^a\left(\sum\limits_{r=1}^k\ell^+_{ri}\nA_r\right)}
{\prod\limits_{i=1}^b\left(\sum\limits_{r=1}^k\ell^-_{ri}\nA_r\right)}\in\Q[\nA],$$
then
$$Z^*(\h_1,\h_2,Q)=\Zp^*(\h_1,\h_2,Q)\pr_1^*P(\nH),$$
where $\pr_1\!:\!\X\!\times\!\X\!\lra\!\X$ is the projection onto the first component, while $\Zp^*$ 
and $Z^*$ are as in Remark~\ref{untwZ_rmk}.
Via Theorem~\ref{Z2pt_thm}, this expresses the two-point function $Z^*$ in terms of the one-point functions $\Zp_{\eta}$, $\Zpp_{\eta}$. In this case and if $\nu_E(\bd)\!\ge\!0$ for all $\bd\!\in\!\La$, $Z^*$ can be computed explicitly in terms of
$\Yp$ and $\Ypp$ via Theorem~\ref{Y_thm}.

We next use an idea from \cite{CoZ} to express $Z^*$ in terms of one-point GW generating functions and then show
how to compute the latter in the $b\!>\!0$ case.
If $\pr_i\!:\!\X\!\times\!\X\!\lra\!\X$
and
$g_{\bp\bs}\!\in\!\Q$ are as in Corollary~\ref{2pt_crl},
then
\begin{equation}\label{untwZ_e2}
Z^*(\h_1,\h_2,Q)\!=\!\frac{1}{\h_1+\h_2}\sum\limits_{|\bp|+|\bs|=N-k}
g_{\bp\bs}\!
\left[\pr_1^*\nH^{\bp}\pr_2^*Z^*_{\bs}(\h_2,Q)\!+\!\pr_1^*Z^*_{\bp}(\h_1,Q)\pr_2^*\Zpp_{\bs}(\h_2,Q)\right],
\end{equation}
where
\begin{equation*}
Z^*_{\bp}(\h,Q)\equiv\sum_{\bd\in\La-0}Q^{\bd}
\ev_{1*}\left[\frac{\E(\cV_E)\ev_2^*\nH^{\bp}}{\h\!-\!\psi_1}\right]
\!\in\!H^*\!\left(\X\right)[\h^{-1}][[\La]]
\end{equation*}
and $\ev_1:\ov\M_{0,2}(\X,\bd)\!\lra\!\X$.
This follows from \e_ref{euntwZ_e2}.

We next assume that $b\!>\!0$ and $\nu_E(\bd)\!\ge\!0$ for all $\bd\!\in\La$ and express $Z^*_{\bp}(\h,Q)$ 
in terms of explicit power series. Along with \e_ref{untwZ_e2} and Theorem~\ref{Y_thm}, this will conclude the computation of $Z^*$.

With $U(\bd;\nA,\h)$ given by \e_ref{wtU_e1}, we define
\begin{equation}\label{newY_e}
\hY(\nA,\h,q)\!\equiv\!\!\sum_{\bd\in\La}\!q^{\bd}U(\bd;\nA,\h)
\prod\limits_{i=1}^a\prod\limits_{s=1}^{L_i^+(\bd)}\left(\sum\limits_{r=1}^k\ell^+_{ri}\nA_r\!+\!s\h\right)
\prod\limits_{i=1}^b\prod\limits_{s=1}^{-L_i^-(\bd)}\left(\sum\limits_{r=1}^k\ell^-_{ri}\nA_r\!-\!s\h\right).
\end{equation}
As $\hY$ satisfies \e_ref{assumeS_e}, we may define 
$\bD^{\bp}\hY$ and $\wt{\nC}^{(r)}_{\bp,\bs}\!\equiv\!\wt{\nC}^{(r)}_{\bp,\bs}(\hY)$ by \e_ref{bDdfn_e} and \e_ref{neqtiC_e}. We define $\hY_{\bp}(\nA,\h,q)$
by the right-hand side of \e_ref{Yp_e} above, with $\Yp$ replaced by $\hY$ and $\tinCp^{(r)}_{\bp,\bs}$ by
$\tinC^{(r)}_{\bp,\bs}$.

Let
\begin{equation}\label{wtY_e}
\wt{Y}^*(\nA,\h,q)\!\equiv\!
\sum\limits_{\bd\in\La-0}q^{\bd}U(\bd;\nA,\h)
\prod\limits_{i=1}^a\prod\limits_{s=1}^{L^+_i(\bd)}\!\!\!\left(\sum\limits_{r=1}^k\ell^+_{ri}\nA_r\!+\!s\h\right)
\prod\limits_{i=1}^b\prod\limits_{s=1}^{-L_i^-(\bd)-1}\!\!\!\left(\sum\limits_{r=1}^k\ell^-_{ri}\nA_r\!-\!s\h\right).
\end{equation}
We define $\nE^{(\br)}_{\bp,\bs}\!\in\!\Q[[\La]]$ by 
\begin{equation}\label{nEcoeff_e}
\left\{\nA+\h\,q\frac{\nd}{\nd q}\right\}^{\bp}\wt{Y}^*(\nA,\h,q)\cong
\sum_{s=0}^{|\bp|-b}\sum_{|\br|=0}^{|\bp|-b-s}\nE^{(\br)}_{\bp,s}\nA^{\br}\h^s\quad\mod\,\h^{-1}.
\end{equation}
It follows that $\LR{\nE^{(\br)}_{\bp,\bs}}_{q;\bd}\!=\!0$ unless~$|\bp|\!=\!b\!+\!s\!+\nu_E(\bd)\!+\!|\br|$.
Whenever $b\!\ge\! 2$,
\begin{equation}\begin{split}\label{Zpstar_e1}
Z^*_{\bp}(\h,q)\!=\!
e(E^+)
\left[\left\{\nH+\h\,q\frac{\nd}{\nd q}\right\}^{\bp}\wt{Y}^*(\nH,\h,q)\!-\!\!\sum_{s=0}^{|\bp|-b}\sum_{|\br|=0}^{|\bp|-b-s}\!\!\nE^{(\br)}_{\bp,s}\h^s\hY_{\br}(\nH,\h,q)\right].
\end{split}\end{equation}
If $b\!=\!1$ and $Q$ and $q$ are related by the mirror map \e_ref{mirrmap_e},
\begin{equation}\begin{split}\label{Zpstar_e2}
Z^*_{\bp}(\h,Q)\!=\!
e(E^+)\ne^{-\frac{e(E^-)f_0(q)}{\h}}\!\!
\left[\left\{\nH+\h\,q\frac{\nd}{\nd q}\right\}^{\bp}\wt{Y}^*(\nH,\h,q)\!-\!\!\!\sum_{s=0}^{|\bp|-b}\sum_{|\br|=0}^{|\bp|-b-s}\!\!\!\!\nE^{(\br)}_{\bp,s}\h^s\hY_{\br}(\nH,\h,q)\right]\\-
\frac{e(E^+)\nH^{\bp}f_0(q)}{\h}\sum_{n=0}^{\i}\frac{1}{(n+1)!}\left[-\frac{e(E^-)f_0(q)}{\h}\right]^n,
\end{split}\end{equation}
where
\begin{equation}\label{f0_e}
f_0(q)\equiv\sum_{\begin{subarray}{c}\bd\in\La-0,\,\nu_E(\bd)=0\\
D_j(\bd)\ge 0\,\forall\,j\in [N]\end{subarray}}\!q^{\bd}
(-1)^{L_1^-(\bd)+1}\left(-L_1^-(\bd)\!-\!1\right)!
\frac{\prod\limits_{i=1}^a\left(L_i^+(\bd)!\right)}{\prod\limits_{j=1}^N\left(D_j(\bd)\right)!}.~\footnotemark
\end{equation}
\footnotetext{In this case, $f_i(q)=\ell^-_{i1}f_0(q)$ with $\ell^-_{i1}\!\in\!\Z$ given by \e_ref{Etensor_e}.}
Identities~\e_ref{Zpstar_e1} and \e_ref{Zpstar_e2} follow by setting $\al\!=\!0$ in \e_ref{cZpstar_e1} and \e_ref{cZpstar_e2}.

As in \cite{CoZ}, if $\X\!=\!\P^{n-1}$ and $b\!\ge\!2$, \e_ref{Zpstar_e1} can be be replaced by a simpler formula in terms
of the power series $F(\nw,q)$ in \e_ref{FIpproj_e} above.
Assume that $E\!\lra\!\P^{n-1}$ is as in Remark~\ref{projDp_rmk} and
$\sum\limits_{i=1}^{a}\ell^+_i\!-\!\sum\limits_{i=1}^{b}\ell^-_i\!=\!n$.
Similarly to \cite[Section~2.3]{CoZ},
$$Z^*_p(\h,q)=\frac{e(E^+)}{e(E^-)}\times
\begin{cases}
\left\{\nH\!+\!\h\,q\frac{\nd}{\nd q}\right\}^p\Yp(\nH,\h,q)\!-\!\nH^p,\quad&\hbox{if}\quad p\!<\!b,\\
\nH^p\frac{\bM^{p-b}F(\frac{\nH}{\h},q)}{I_{p-b}(q)}\!-\!\nH^p,\quad&\hbox{if}\quad p\!\ge\!b,
\end{cases}$$  
where the right-hand side should be first simplified in $\Q(\nH,\h)[[q]]$ to eliminate division by $\nH$
and only afterwards viewed as an element in $H^*(\P^{n-1})[\h^{-1}][[q]]$.
This follows from Remarks~\ref{projfDp_rmk} and \ref{projDp_rmk} together with Theorem~\ref{cY_thm}.
By Theorem~\ref{Y_thm} and Remark~\ref{projDp_rmk},
$$\Zpp_p(\h,q)=\begin{cases}
\left\{\nH+\h\,q\frac{\nd}{\nd q}\right\}^p\Ypp(\nH,\h,q),\quad&\hbox{if}\quad p\!<\!a,\\
\nH^p\frac{M^{p-a}F(\frac{\nH}{\h},q)}{I_{p-a}(q)},\quad&\hbox{if}\quad p\!\ge\!a.
\end{cases}$$
The last two displayed equations together with \e_ref{untwZ_e2} imply that
\begin{equation}\label{localCY_e}
\sum\limits_{d=1}^{\i}\!dq^d\int_{\ov\M_{0,2}(\P^{n-1},d)}\!\!\!\!e(\cV_E)\ev_1^*\nH^{c_1}
\ev_2^*\nH^{c_2}\!=\!\frac{\prod\limits_{i=1}^a\ell_i^+}{\prod\limits_{i=1}^b\ell_i^-}\left(I_{c_1+1-b}(q)\!-\!1\right)\quad\textnormal{and}\quad I_{c_1+1-b}\!=\!I_{c_2+1-b},
\end{equation}
whenever $c_1\!+c_2\!=\!n\!-\!2\!-\!a\!+\!b$
and with $I_p(q)$ defined by \e_ref{FIpproj_e} if $p\!\ge\!0$ and $I_p(q)\!\equiv\!1$ if $p\!<\!0$.
\end{rmk}

\section{Equivariant theorems}
\label{equiv_sec}

In this section we introduce equivariant versions of the GW generating functions 
$\Zp$, $\Zp_{\eta}$, and $\Zpp_{\eta}$ of \e_ref{Z2ptdfn_e} and \e_ref{Zeta1ptdfn_e}. We then present theorems
about them which imply the non-equivariant statements of Section~\ref{stat_sec}.

With $\al\!\equiv\!(\al_1,\ldots,\al_N)$ denoting the $\T^N$-weights of Section~\ref{equiv-intro_sec},
$H^*_{\T^N}(\X)$ is generated over $\Q[\al]$ by $\{x_1,\ldots,x_k\}$;
see Proposition~\ref{equivcoh_prp}.
The classes $x_i$ of \e_ref{ebdle_e} satisfy
$$H^2_{\T^N}(\X)\!\ni\! x_i\xrightarrow{\textnormal{restriction}} \nH_i\!\in\!H^2(\X)\qquad\forall\,i\!\in\![k],\quad
\E(\ga_i^*)=x_i\qquad\forall\,i\!\in\![k],$$
where $\E(\ga_i^*)$ is defined by the lift \e_ref{trivact_e} of the action of $\T^N$ on $\X$ to the line bundle $\ga_i^*$. 
Let
$$x\equiv(x_1,\ldots,x_k),\quad x^{\bp}\!\equiv\!x_1^{p_1}\cdot\!\ldots\!\cdot x_k^{p_k}
\qquad\forall\,\bp=(p_1,\ldots,p_k)\!\in\!\left(\Z^{\ge 0}\right)^k.$$

The action of $\T^N$ on $\X$ induces an action on $\ov\M_{0,m}(\X,\bd)$ which lifts to an action on the vector orbi-bundles $\cV_E$, $\cVp_E$, and $\cVpp_E$ of \e_ref{cVdfn_e} and \e_ref{cV'dfn_e}.  
It also lifts to an action on the universal cotangent line bundle to the $i$-th marked point
whose equivariant Euler class will also be denoted by $\psi_i$.
The evaluation maps $\ev_i\!:\!\ov\M_{0,m}(\X,\bd)\!\lra\!\X$ are $\T^N$-equivariant.

With $\ev_1,\ev_2\!:\!\ov\M_{0,3}(\X,\bd)\!\lra\!\X$ denoting the evaluation maps at the
first two marked points, let
\begin{equation}\label{eZ2ptdfn_e}
\cZp(\h_1,\h_2,Q)\equiv\frac{\h_1\h_2}{\h_1\!+\!\h_2}
\sum_{\bd\in\La}\!Q^{\bd}\!\left(\ev_1\!\times\!\ev_2\right)_*\!
\left[\!\frac{\E(\cVp_E)}{\left(\h_1\!-\!\psi_1\right)\left(\h_2\!-\!\psi_2\right)}\!\right].
\end{equation}
With $\ev_1,\ev_2\!:\!\ov\M_{0,2}(\X,\bd)\!\lra\!\X$ denoting the evaluation maps at the
two marked points and for all $\eta\!\in\!H^*_{\T^N}(\X)$, let
\begin{equation}\begin{split}\label{eZeta1ptdfn_e}
\cZp_{\eta}(\h,Q)&\equiv \eta\!+\!\sum_{\bd\in\La-0}Q^{\bd}
\ev_{1*}\left[\frac{\E(\cVp_E)\ev_2^*\eta}{\h\!-\!\psi_1}\right]
\!\in\!H^*_{\T^N}\!\left(\X\right)[[\h^{-1},\La]],\\
\cZpp_{\eta}(\h,Q)&\equiv \eta\!+\!\sum_{\bd\in\La-0}Q^{\bd}
\ev_{1*}\left[\frac{\E(\cVpp_E)\ev_2^*\eta}{\h\!-\!\psi_1}\right]
\!\in\!H^*_{\T^N}\!\left(\X\right)[[\h^{-1},\La]].
\end{split}\end{equation}
In the $\eta\!=\!x^{\bp}$ cases, these are equivariant versions of $\Zp_{\bp}$ and $\Zpp_{\bp}$
in \e_ref{Zp1pt_e}: 
\begin{equation}\begin{split}\label{eZp1pt_e}
\cZp_{\bp}(\h,Q)\equiv\cZp_{x^{\bp}}(\h,Q),\,
\cZpp_{\bp}(\h,Q)\equiv\cZpp_{x^{\bp}}(\h,Q)
\!\in\!H^*_{\T^N}\!\left(\X\right)[[\h^{-1},\La]].
\end{split}\end{equation}
In particular, $\cZp_{\b0}\!\equiv\!\cZp_1$, with $\b0\!\in\!\Z^k$ and $1\!\in\!H^*_{\T^N}(\X)$.

Section~\ref{econstruction_sec} below constructs the explicit formal power series in terms of which
$\cZp_{\bp}$ and $\cZpp_{\bp}$ are expressed in Theorem~\ref{cY_thm}.
Throughout this construction, which extends the constructions in \cite[Section~2.3]{bcov0}
and \cite[Section~3.1]{bcov0_ci} from $\P^{n-1}$ to an arbitrary toric manifold $\X$, we assume that 
$\nu_E(\bd)\!\ge\!0$ for all $\bd\!\in\!\La$ and identify $H_2(\X;\Z)\!\cong\!\Z^k$
via the basis $\{\nH_1,\ldots,\nH_k\}$ of $H^2(\X;\Z)$.
Via this identification $\La\!\hookrightarrow\!\Z^k$.

\subsection{Construction of equivariant power series}
\label{econstruction_sec}
We begin by defining equivariant versions $\cYp$ and $\cYpp$ of the power series $\Yp$ and $\Ypp$ in \e_ref{Y'_e}
as these will compute $\cZp_{\bp}$ and $\cZpp_{\bp}$ in Theorem~\ref{cY_thm}.
We consider the lift \e_ref{trivact_e} of the $\T^N\!$-action on $\X$ to the line bundles
$L_i^{\pm}$ of \e_ref{splitE_e} and \e_ref{Etensor_e} so that
\begin{equation}\label{la_e}
\la_i^{\pm}\equiv\E(L_i^{\pm})\!=\!\sum\limits_{r=1}^k\ell^{\pm}_{ri}x_r.
\end{equation}
An equivariant version of the power series $U(\bd;\nA,\h)$ in \e_ref{wtU_e1} is given by
\begin{equation}\label{ewtU_e1}
u(\bd;\nA,\h)\!\equiv\!
\frac{\prod\limits_{\begin{subarray}{c}j\in[N]\\D_j(\bd)<0\end{subarray}}\prod\limits_{s=D_j(\bd)+1}^0
\left(\sum\limits_{i=1}^km_{ij}\nA_i\!-\!\al_j\!+\!s\h\right)}
{\prod\limits_{\begin{subarray}{c}j\in[N]\\D_j(\bd)\ge 0\end{subarray}}\prod\limits_{s=1}^{D_j(\bd)}\left(\sum\limits_{i=1}^km_{ij}\nA_i\!-\!\al_j\!+\!s\h\right)}
\in\Q[\al,\nA]\Lau{\h}.
\end{equation}
By \e_ref{u_e}, 
\begin{equation}\label{uYgeom_e}
u(\bd;x,\h)\!=\!\frac{\prod\limits_{\begin{subarray}{c}j\in[N]\\D_j(\bd)<0\end{subarray}}\prod\limits_{s=D_j(\bd)+1}^0
\left(u_j\!+\!s\h\right)}
{\prod\limits_{\begin{subarray}{c}j\in[N]\\D_j(\bd)\ge 0\end{subarray}}\prod\limits_{s=1}^{D_j(\bd)}\left(u_j\!+\!s\h\right)}.
\end{equation}
With $\Ep(\bd;\nA,\h)$ and $\Epp(\bd;\nA,\h)$ as in \e_ref{wtU_e2}, let 
\begin{equation}\begin{split}\label{cY_e}
\cYp(\nA,\h,q)&\equiv\sum_{\bd\in\La}\!q^{\bd}u(\bd;\nA,\h)\Ep(\bd;\nA,\h)\in\Q[\al,\nA][[\h^{-1},\La]],\\
\cYpp(\nA,\h,q)&\equiv\sum_{\bd\in\La}\!q^{\bd}
u(\bd;\nA,\h)\Epp(\bd;\nA,\h)\in\Q[\al,\nA][[\h^{-1},\La]].
\end{split}\end{equation}

In the above definitions of $\cYp$ and $\cYpp$ and throughout the construction below, the torus weights $\al$ should be thought of as formal variables,
in the same way in which $\nA$ of Section~\ref{construction_sec} are formal variables.
With $\nA$ replaced by $x$, $\cYp$ and $\cYpp$ become well-defined elements in $H^*_{\T^N}(\X)[[\h^{-1},\La]]$.
However, this is irrelevant for the purposes of this subsection and becomes relevant only when we use $\cYp$ and $\cYpp$ in the formulas for $\cZp_{\bp}$ and $\cZpp_{\bp}$.

As before, $\Q_{\al}\!\equiv\!\Q(\al)$.
We next describe an operator $\fD^{\bp}$ acting on a subset of
$\Q_{\al}(\nA,\h)[[\La]]$ and certain associated ``equivariant structure coefficients''
in $\Q[\al][[\La]]$ which occur in the formulas for $\cZp_{\bp}$ and $\cZpp_{\bp}$.
Fix an element $\cY(\nA,\h,q)\!\in\!\Q_{\al}(\nA,\h)[[\La]]$ such that for all $\bd\!\in\!\La$
$$\LR{\cY(\nA,\h,q)}_{q;\bd}\equiv\frac{f_{\bd}(\nA,\h)}{g_{\bd}(\nA,\h)}$$
for some homogeneous polynomials $f_{\bd}(\nA,\h),g_{\bd}(\nA,\h)\!\in\!\Q[\al,\nA,\h]$
satisfying
\begin{equation}\begin{split}\label{assumecS_e}
f_{\b0}(\nA,\h)=g_{\b0}(\nA,\h),\qquad
\deg\,f_{\bd}-\deg\,g_{\bd}=-\nu_E(\bd),\qquad
g_{\bd}\left|_{\begin{subarray}{c}\nA=0\\\al=0\end{subarray}}\right.\!\neq\!0\quad\forall\,\bd\in\!\La.
\end{split}\end{equation}
This condition is satisfied by the power series $\cYp$ and $\cYpp$
of \e_ref{cY_e} and so the construction below applies to $\cY\!=\!\cYp$ and $\cY\!=\!\cYpp$.
We inductively define $\fD^{\bp}\cY(\nA,\h,q)$ in $\Q_{\al}(\nA,\h)[[\La]]$
satisfying
\begin{enumerate}[label=(EP\arabic*),leftmargin=*] 
\item\label{EP_noneq} with $\bD^{\bp}$ defined in Section~\ref{construction_sec},
$$\fD^{\bp}\cY(\nA,\h,q)\Big|_{\al=0}=\bD^{\bp}\left(\cY(\nA,\h,q)\Big|_{\al=0}\right);$$
\item\label{EP_Dpexp} there exist $\cC^{(\br)}_{\bp,s}\!\equiv\!\cC^{(\br)}_{\bp,s}(\cY)\!\in\!\Q[\al][[\La]]$
with $\bp,\br\!\in\!(\Z^{\ge 0})^k$, $s\!\in\!\Z^{\ge0}$, such that $\LR{\cC^{(\br)}_{\bp,s}}_{q;\bd}$ is a homogeneous polynomial in $\al$ of degree $-\nu_E(\bd)\!-\!|\br|\!+\!s$,
\begin{equation}\label{fDpexp_e}
\fD^{\bp}\cY(\nA,\h,q)=\h^{|\bp|}\sum\limits_{s=0}^{\i}\sum\limits_{\br\in(\Z^{\ge 0})^k}\cC^{(\br)}_{\bp,s}(q)
\nA^{\br}\h^{-s},
\end{equation}
\begin{equation}\label{EP_Dpq;0_e}
\LR{\cC^{(\br)}_{\bp,s}}_{q;\b0}\!\!\!\!=\!\de_{\bp,\br}\de_{|\br|,s}\quad\forall\,\bp,\br\!\in\!(\Z^{\ge0})^k,\,s\!\in\!\Z^{\ge 0}.
\end{equation}
\end{enumerate}
By \e_ref{assumecS_e}, \e_ref{bD0dfn_e}, and since $\LR{\{J_0(\cY\big|_{\al=0})\}(1)}_{q;\b0}\!=\!1$
by \ref{P_J_pq;0},
we can define
\begin{equation}\label{fD0dfn_e}
\fD^{\b0}\cY(\nA,\h,q)\equiv\left[\{J_0(\cY\big|_{\al=0})\}(1)\right]^{-1}\cY(\nA,\h,q)\!\in\!\Q_{\al}(\nA,\h)[[\La]].
\end{equation}

Suppose next that $p\!\ge\!0$ and we have constructed power series $\fD^{\bp'}\cY(\nA,\h,q)$
for all
$\bp'\!\in\!(\Z^{\ge 0})^k$ with $|\bp'|\!=\!p$
satisfying the above properties.
For each $\bp\!\in\!(\Z^{\ge 0})^k$ with $|\bp|\!=\!p\!+\!1$, let
\begin{equation}\begin{split}\label{fDdfn_e}
\wt{\fD}^{\bp}\cY(\nA,\h,q)\!&\equiv\!\frac{1}{|\supp(\bp)|}
\sum_{i\in\supp(\bp)}\left\{\nA_i\!+\!\h\,q_i\frac{\nd}{\nd q_i}\right\}\fD^{\bp-e_i}\cY(\nA,\h,q)\!\in\!\Q_{\al}(\nA,\h)[[\La]],\\
\fD^{\bp}\cY(\nA,\h,q)\!&\equiv\!\!\!\!\sum_{\bp'\in(\Z^{\ge 0})^k,|\bp'|=p+1}\!\!\!\nc_{\bp;\bp'}(q)
\wt{\fD}^{\bp'}\cY(\nA,\h,q),
\end{split}\end{equation}
where $\nc_{\bp;\bp'}(q)\!\in\!\Q[[\La;\nu_E\!=\!0]]$ are defined by \e_ref{Jpinv_e}
with $Y\!\equiv\!\cY\big|_{\al=0}$ and where
$\{e_1,\ldots,e_k\}$ is the standard basis of $\Z^k$.
Since \ref{EP_noneq} holds with $\bp$ replaced by any $\bp'$ with $|\bp'|\!=\!p$,
$$\wt{\fD}^{\bp}\cY(\nA,\h,q)\big|_{\al=0}=\wt{\bD}^{\bp}\left(\cY(\nA,\h,q)\big|_{\al=0}\right)$$
by \e_ref{tiDpdfn_e}; thus, by the second equation in \e_ref{fDdfn_e} and \e_ref{bDdfn_e},
$\fD^{\bp}\cY$ satisfies \ref{EP_noneq}.
It is immediate to verify that $\fD^{\bp}\cY(\nA,\h,q)$ admits an expansion as in \e_ref{fDpexp_e}.
Since $\LR{\nc_{\bp;\bp'}}_{q;\b0}\!=\!\de_{\bp,\bp'}$ by the second statement in~\e_ref{Jpprop_e} and~\e_ref{Jpinv_e}
and \e_ref{EP_Dpq;0_e} holds for $\bp\!-\!e_i$ with $i\!\in\!\supp(\bp)$
instead of $\bp$, \e_ref{EP_Dpq;0_e} also holds for $\bp$ with $|\bp|\!=\!p\!+\!1$.

By \ref{EP_noneq}, \e_ref{bDpexp_e}, and \e_ref{fDpexp_e},
\begin{equation}\label{cCvsnC_e}
\cC^{(\br)}_{\bp,s}(\cY)\big|_{\al=0}=\nC^{(\br)}_{\bp,s}\left(\cY\Big|_{\al=0}\right),
\end{equation}
with $\nC^{(\br)}_{\bp,s}$ as in \ref{P_Dpexp}.

Define
$\ctC^{(r)}_{\bp,\bs}\!\equiv\!\ctC^{(r)}_{\bp,\bs}(\cY)\!\in\!\Q[\al][[\La]]$ 
for $\bp,\bs\!\in\!(\Z^{\ge 0})^k$ and $r\!\in\!\Z^{\ge 0}$ with $|\bs|\!\le\!|\bp|\!-\!r$ and $r\!\le\!|\bp|$
by
\begin{equation}\label{eqtiC_e}
\sum_{t=0}^r\sum_{\begin{subarray}{c}
\bs\in(\Z^{\ge 0})^k\\|\bs|\le|\bp|-t\end{subarray}}
\ctC^{(t)}_{\bp,\bs}\cC_{\bs,|\br|+r-t}^{(\br)}=\de_{\bp,\br}\de_{r,0}\qquad
\forall\,\br\!\in\!(\Z^{\ge 0})^k,\,|\br|\!\le\!|\bp|\!-\!r.
\end{equation}
Equations \e_ref{eqtiC_e} indeed uniquely determine $\ctC^{(r)}_{\bp,\bs}$,
since 
\begin{equation}\label{eqtiC_e2}
\sum_{t=0}^r\sum_{\begin{subarray}{c}
\bs\in(\Z^{\ge 0})^k\\
|\bs|\le|\bp|-t\end{subarray}}
\ctC^{(t)}_{\bp,\bs}\cC^{(\br)}_{\bs,|\br|+r-t}=
\sum_{t=0}^{r-1}\sum_{\begin{subarray}{c}\bs\in(\Z^{\ge 0})^k\\|\bs|\le|\bp|-t\end{subarray}}
\ctC^{(t)}_{\bp,\bs}\cC_{\bs,|\br|+r-t}^{(\br)}
+
\sum_{\begin{subarray}{c}\bs\in(\Z^{\ge 0})^k\\
|\bs|<|\br|\end{subarray}}\ctC^{(r)}_{\bp,\bs}\cC_{\bs,|\br|}^{(\br)}
+\ctC^{(r)}_{\bp,\br}.
\end{equation}
This follows from
\begin{equation}\label{cCprop_e}
\cC^{(\br)}_{\bp,|\br|}=
\de_{\bp,\br}\quad\textnormal{if}\quad|\br|\!\le\!|\bp|,\end{equation}
which in turn follows from \e_ref{cCvsnC_e}, the second equation in \e_ref{degnC_e}, and the first property in \ref{EP_Dpexp}.
By \e_ref{eqtiC_e} and \e_ref{EP_Dpq;0_e},
\begin{equation}\label{eqtiCq;0_e}
\LR{\ctC^{(r)}_{\bp,\bs}}_{q;\b0}=\de_{\bp,\bs}\de_{r,0}.
\end{equation}
By \e_ref{cCvsnC_e}, \e_ref{neqtiC_e}, \e_ref{neqtiC_e2}, \e_ref{eqtiC_e}, \e_ref{eqtiC_e2}, and induction,
\begin{equation}\label{ctCvstinC_e}
\ctC^{(r)}_{\bp,\bs}(\cY)\big|_{\al=0}=\wt{\nC}^{(r)}_{\bp,\bs}\left(\cY\big|_{\al=0}\right).
\end{equation}
By \e_ref{eqtiCq;0_e} in the $\bd\!=\!\b0$ case and \e_ref{eqtiC_e}, \e_ref{eqtiC_e2}, \ref{EP_Dpexp}, and induction
in all other cases,
$\LR{\ctC^{(r)}_{\bp,\bs}(q)}_{q;\bd}$ is a degree $r\!-\!\nu_E(\bd)$ homogeneous polynomial in $\al$.
In particular, $\ctC^{(0)}_{\bp,\bs}(q)\!\in\!\Q[[\La]]$.
This together with \e_ref{ctCvstinC_e} and \e_ref{tiC0_e} implies that,
\begin{equation}\label{ctC0_e}
\ctC^{(0)}_{\bp,\bs}(q)=\de_{\bp,\bs}\qquad\forall\bp,\bs\!\in\!(\Z^{\ge 0})^k\quad\textnormal{with}\quad|\bs|\!\le\!|\bp|.
\end{equation}
\begin{rmk}\label{symC_rmk}
If columns $i$ and $j$ of $M$ are equal, then $u(\bd;\nA,\h)$ of \e_ref{ewtU_e1} is invariant under the permutation
$$\al_{i}\lra\al_j,\quad\al_j\lra\al_i,\quad\al_k\lra\al_k\quad\forall\,k\neq i,j,$$
for all $\bd\!\in\!\La$. In this case, 
the structure coefficients $\ctC^{(r)}_{\bp,\bs}$ are also invariant under this permutation by \e_ref{eqtiC_e2}.

In particular, if $X^{\tau}_M\!=\!\P^{N_1-1}\!\times\!\ldots\!\times\P^{N_s-1}$ as in \e_ref{prodproj_e} and the $\T^N$-weights are
$\al^{(i)}_j$ with $1\!\le\!i\!\le s$ and $1\!\le\!j\!\le\!N_i$ as in \e_ref{projcoh_e}, then the structure coefficients
$\ctC^{(r)}_{\bp,\bs}$ are symmetric in $\al^{(i)}_1,\ldots,\al^{(i)}_{N_i}$ for each $i\!=\!1,\ldots,s$.
In the case of $\P^{N-1}$, this remark was used in the computation of the genus~1 GW invariants of Calabi-Yau complete intersections in
\cite{bcov1} and then in \cite{bcov1_ci}.

\end{rmk}
\begin{rmk}\label{fD0Y_rmk}
By \e_ref{fD0dfn_e}, $\cYp\big|_{\al=0}\!=\!\Yp$,
$\cYpp\big|_{\al=0}\!=\!\Ypp$, and Remark~\ref{J0Y'_rmk},
$$\fD^{\b0}\cYp(\nA,\h,q)=\frac{1}{\Ip_{\b0}(q)}\cYp(\nA,\h,q),\qquad
\fD^{\b0}\cYpp(\nA,\h,q)=\frac{1}{\Ipp_{\b0}(q)}\cYpp(\nA,\h,q).$$
\end{rmk}
\begin{rmk}\label{simplefDp_rmk}
If $\nu_E(\bd)\!>\!0$ for all $\bd\!\in\!\La\!-\!\{0\}$ and $\cY(\nA,\h,q)\!\in\!\Q_{\al}(\nA,\h)[[\La]]$
satisfies \e_ref{assumecS_e}, then
$$\fD^{\bp}\cY(\nA,\h,q)=\left\{\nA\!+\!\h\,q\frac{\nd}{\nd q}\right\}^{\bp}\cY(\nA,\h,q)\qquad\forall\,\bp\!=\!(p_1,\ldots,p_k)\!\in\!(\Z^{\ge 0})^k.$$
This follows by induction on $|\bp|$ from \e_ref{fDdfn_e} since
$\nc_{\bp;\bp'}(\cY\big|_{\al=0})\!=\!\de_{\bp,\bp'}$
with $\nc_{\bp;\bp'}$ defined by \e_ref{Jpinv_e}.
The latter follows since $J_p(\cY\big|_{\al=0})\!=\!\Id$ by Remark~\ref{simpleDp_rmk}.
\end{rmk}
\begin{rmk}\label{partfDp_rmk}
Suppose $p^*\!\in\!\Z^{\ge 0}$, $\cY(\nA,\h,q)\!\in\!\Q_{\al}(\nA,\h)[[\La]]$ satisfies \e_ref{assumecS_e}, and
$$\deg_{\h}f_{\bd}(\nA,\h)-\deg_{\h}g_{\bd}(\nA,\h)\!<\!-p^*\qquad\forall\,\bd\!\in\!\La\!-\!0.$$
By the same reasoning as in Remark~\ref{simplefDp_rmk}, but using Remark~\ref{partDp_rmk} instead of \ref{simpleDp_rmk},
\begin{equation}\label{partfDp_e}
\fD^{\bp}\cY(\nA,\h,q)\!=\!\left\{\nA\!+\!\h\,q\frac {\nd}{\nd q}\right\}^{\bp}\cY(\nA,\h,q)\qquad\textnormal{if}\quad|\bp|\!\le\!p^*.
\end{equation}
By \e_ref{eqtiC_e}, \e_ref{eqtiC_e2}, and \e_ref{ctC0_e},
\begin{equation}\label{eqtiC_e3}
\cC^{(\br)}_{\bp,|\br|+r}+\!\!
\sum_{t=1}^{r-1}\sum_{\begin{subarray}{c}\bs\in(\Z^{\ge 0})^k\\|\bs|\le|\bp|-t\end{subarray}}
\ctC^{(t)}_{\bp,\bs}\cC_{\bs,|\br|+r-t}^{(\br)}\!\!
+\!\!
\sum_{\begin{subarray}{c}\bs\in(\Z^{\ge 0})^k\\|\bs|<|\br|\end{subarray}}
\ctC^{(r)}_{\bp,\bs}\cC^{(\br)}_{\bs,|\br|}
+\ctC^{(r)}_{\bp,\br}=0\end{equation}
if $r\!\ge\!1$ and $|\br|\!\le\!|\bp|\!-\!r$.
By \e_ref{partfDp_e} and \e_ref{EP_Dpq;0_e}, $$\fD^{\bp}\cY(\nA,\h,q)\cong\nA^{\bp}\quad\mod\,\h^{-1}\qquad\textnormal{if}\quad|\bp|\!\le\!p^*.$$ 
This together with \e_ref{fDpexp_e} implies that whenever $|\bp|\!\le\!p^*$,
\begin{equation}\label{partfDp_e2}
\cC^{(\br)}_{\bp,|\br|+r}=0\qquad\textnormal{if}\quad\,r\!\ge\!1\quad\textnormal{and}\quad
|\br|\!\le\!|\bp|\!-\!r.
\end{equation}
Finally, using \e_ref{eqtiC_e3}, \e_ref{partfDp_e2}, and induction, we find that
$$\ctC^{(r)}_{\bp,\bs}\!=\!0\qquad\textnormal{if}\quad{r\!\ge\!1},\quad|\bp|\!\le\!p^*,\quad|\bs|\!\le\!|\bp|\!-\!r.$$
\end{rmk}
\begin{rmk}\label{projfDp_rmk}
Let $(M,\tau)$ be the toric pair
of Example~\ref{proj_eg} with $N\!=\!n$ so that $\X\!=\!\P^{n-1}$ and $E\!\lra\!\P^{n-1}$
be as in Remark~\ref{projDp_rmk}.
By \e_ref{cY_e},
\begin{equation*}\begin{split}
\cYp(\nA,\h,q)&=\sum_{d=0}^{\i}q^d
\frac{\prod\limits_{i=1}^a\prod\limits_{s=1}^{\ell^+_id}\left(\ell^+_i\nA\!+\!s\h\right)
\prod\limits_{i=1}^b\prod\limits_{s=0}^{-\ell_i^-d-1}\left(\ell_i^-\nA\!-\!s\h\right)}{\prod\limits_{j=1}^n\prod\limits_{s=1}^d\left(\nA\!-\!\al_j\!+\!s\h\right)},\\
\cYpp(\nA,\h,q)&=\sum_{d=0}^{\i}q^d
\frac{\prod\limits_{i=1}^a\prod\limits_{s=0}^{\ell^+_id-1}\left(\ell^+_i\nA\!+\!s\h\right)
\prod\limits_{i=1}^b\prod\limits_{s=1}^{-\ell_i^-d}\left(\ell_i^-\nA\!-\!s\h\right)}{\prod\limits_{j=1}^n
\prod\limits_{s=1}^d\left(\nA\!-\!\al_j\!+\!s\h\right)}.
\end{split}\end{equation*}
By Remark~\ref{partfDp_rmk},
\begin{align*}
\fD^{p}\cYp(\nA,\h,q)&=\left\{\nA\!+\!\h\,q\frac{\nd}{\nd q}\right\}^p\cYp(\nA,\h,q)
&\quad\textnormal{and}&\quad&
\ctC^{(r)}_{p,s}(\cYp)&=0\qquad\forall\,p\!<\!b,\,\,1\!\le\!r\!\le\!p,\\
\fD^{p}\cYpp(\nA,\h,q)&=\left\{\nA\!+\!\h\,q\frac{\nd}{\nd q}\right\}^p\cYpp(\nA,\h,q)
&\quad\textnormal{and}&\quad&
\ctC^{(r)}_{p,s}(\cYpp)&=0
\qquad\forall\,p\!<\!a,\,\,1\!\le\!r\!\le\!p.
\end{align*}
If $\sum\limits_{i=1}^a\ell^+_i\!-\!\sum\limits_{i=1}^b\ell^-_i\!<\!n$, then
$$\fD^{p}\cYp=\left\{\nA\!+\!\h\,q\frac{\nd}{\nd q}\right\}^p\cYp,\quad\fD^{p}\cYpp=
\left\{\nA\!+\!\h\,q\frac{\nd}{\nd q}\right\}^p\cYpp,$$
for all $p$ by Remark~\ref{simplefDp_rmk}.
If $\sum\limits_{i=1}^a\ell^+_i\!-\!\sum\limits_{i=1}^b\ell^-_i\!=\!n$, then 
\begin{align*}
\fD^b\cYp&=\frac{1}{I_0(q)}\left\{\nA\!+\!\h\,q\frac{\nd}{\nd q}\right\}^b\!\cYp,\quad&
\fD^{p}\cYp&=\frac{1}{I_{p-b}(q)}\left\{\nA\!+\!\h\,q\frac{\nd}{\nd q}\right\}\fD^{p-1}\cYp\qquad
\forall\,p\!>\!b,\\
\fD^a\cYpp&=\frac{1}{I_0(q)}\left\{\nA\!+\!\,h\,q\frac{\nd}{\nd q}\right\}^a\cYpp,\quad&
\fD^{p}\cYpp&=\frac{1}{I_{p-a}(q)}\left\{\nA\!+\!\h\,q\frac{\nd}{\nd q}\right\}\fD^{p-1}\cYpp\qquad
\forall\,p\!>\!a,
\end{align*}
by \e_ref{fDdfn_e} and Remark~\ref{projDp_rmk}.
\end{rmk}

\subsection{Equivariant statements}
\label{estat_sec}
\begin{thm}\label{eZ2pt_thm}
Suppose $(M,\tau)$ is a minimal toric pair and
$\pr_i\!:\!\X\!\times\!\X\!\lra\!\X$ is the projection onto the $i$-th component. If
$\eta_j,\wceta_j\!\in\!H^*_{\T^N}(\X)$ are such that
$$\sum_{j=1}^s
\pr_1^*\eta_j\,\pr_2^*\wceta_j\in
H^{2(N-k)}_{\T^N}\left(\X\!\times\!\X\right)$$ is the equivariant Poincar\'{e} dual of the diagonal,
then
\begin{equation}\label{eZ2pt_e}
\cZp(\h_1,\h_2,Q)=\frac{1}{\h_1\!+\!\h_2}
\sum_{j=1}^s\pr_1^*\cZp_{\eta_j}(\h_1,Q)\,\pr_2^*\cZpp_{\wceta_j}(\h_2,Q).
\end{equation}
\end{thm}

\begin{crl}\label{eproj_crl}
Let $(M,\tau)$ be the minimal toric pair \e_ref{prodproj_e} so that 
$\X\!\!=\!\!\prod\limits_{i=1}^s\P^{N_i-1}, 
N\!\!=\!\!\!\sum\limits_{i=1}^sN_i$, and $H^*_{\T^N}\!\!\left(\prod\limits_{i=1}^s\P^{N_i-1}\!\!\right)$
is given by \e_ref{projcoh_e}.
Let $\pr_j\!:\!\X\!\times\!\X\!\lra\!\X$ denote the projection onto the j-th component.
For all $i\!\in\![s]$ and $r\!\in\!\Z^{\ge0}$, denote by $\si^{(i)}_r$
the $r$-th elementary symmetric polynomial in $\al^{(i)}_1\!,\ldots,\al^{(i)}_{N_i}$.
Then, 
\begin{equation*}
\cZp(\h_1,\h_2,Q)\!=\!\frac{1}{\h_1\!+\!\h_2}\!\!
\sum\limits_{\begin{subarray}{c}r_i+a_i+b_i=N_i-1\,\forall\,i\in[s]\\
r_i,a_i,b_i\ge0\,\forall\,i\in[s]\end{subarray}}\!\!
(-1)^{\sum\limits_{i=1}^sr_i}\si_{r_1}^{(1)}\ldots\si_{r_s}^{(s)}
\pr_1^*\cZp_{(a_1,\ldots,a_s)}(\h_1,Q)\,\pr_2^*\cZpp_{(b_1,\ldots,b_s)}(\h_2,Q).
\end{equation*}
\end{crl}
This follows from Theorem~\ref{eZ2pt_thm} as the equivariant Poincar\'{e} dual to the diagonal in 
$\prod\limits_{i=1}^s\P^{N_i-1}$ is
$$\sum\limits_{\begin{subarray}{c}r_i+a_i+b_i=N_i-1\,\forall\,i\in[s]\\
r_i,a_i,b_i\ge0\,\forall\,i\in[s]\end{subarray}}\!\!
(-1)^{\sum\limits_{i=1}^sr_i}\si_{r_1}^{(1)}\ldots\si_{r_s}^{(s)}\pr_1^*(x_1^{a_1}\ldots x_s^{a_s})\,\pr_2^*(x_1^{b_1}\ldots
x_s^{b_s}).$$
\begin{thm}\label{cY_thm}
Let $(M,\tau)$ be a minimal toric pair.
If $\nu_E(\bd)\!\ge\!0$ for all $\bd\!\in\!\La$, then $\cZp_{\bp}$ and $\cZpp_{\bp}$
of \e_ref{eZp1pt_e} and \e_ref{eZeta1ptdfn_e} are given by
\begin{equation}\begin{split}\label{ZpYp_e}
\cZp_{\bp}(\h,Q)&=\ne^{-\frac{1}{\h}\left[G(q)+\sum\limits_{i=1}^kx_if_i(q)+\sum\limits_{j=1}^N\al_jg_j(q)\right]}
\cYp_{\bp}(x,\h,q)\!\in\!H^*_{\T^N}(\X)\Lau{\h}[[\La]],\\
\cZpp_{\bp}(\h,Q)&=\ne^{-\frac{1}{\h}\left[G(q)+\sum\limits_{i=1}^kx_if_i(q)+\sum\limits_{j=1}^N\al_jg_j(q)\right]}
\cYpp_{\bp}(x,\h,q)\!\in\!H^*_{\T^N}(\X)\Lau{\h}[[\La]],
\end{split}\end{equation}
where 
\begin{equation}\begin{split}\label{cYp_e}
\cYp_{\bp}(x,\h,q)&\equiv\fD^{\bp}\cYp(x,\h,q)+\sum_{r=1}^{|\bp|}\sum_{|\bs|=0}^{|\bp|-r}\ticCp^{(r)}_{\bp,\bs}(q)\h^{|\bp|-r-|\bs|}
\fD^{\bs}\cYp(x,\h,q),\\
\cYpp_{\bp}(x,\h,q)&\equiv\fD^{\bp}\cYpp(x,\h,q)+\sum_{r=1}^{|\bp|}\sum_{|\bs|=0}^{|\bp|-r}\ticCpp^{(r)}_{\bp,\bs}(q)\h^{|\bp|-r-|\bs|}
\fD^{\bs}\cYpp(x,\h,q),
\end{split}\end{equation}
with $\ticCp^{(r)}_{\bp,\bs}\equiv\ticC^{(r)}_{\bp,\bs}(\cYp)$ and
$\ticCpp^{(r)}_{\bp,\bs}\equiv\ctC^{(r)}_{\bp,\bs}(\cYpp)$ defined by \e_ref{eqtiC_e},
$Q$ and $q$ related by the mirror map \e_ref{mirrmap_e}, $G$, $f_i$ and $g_j\!\in\!\Q[[\La\!-\!0;\nu_E\!=\!0]]$~\footnotemark given by \e_ref{G_e}, \e_ref{fi_e}, 
and \e_ref{gj_e}, and the operator $\fD^{\bp}$ defined by \e_ref{fDdfn_e}.
The coefficient of $q^{\bd}$ within each of $\ticCp^{(r)}_{\bp,\bs}$ and $\ticCpp^{(r)}_{\bp,\bs}$ is
a degree $r\!-\!\nu_E(\bd)$ homogeneous polynomial in $\al_1,\ldots,\al_N$.

If $|\bp|\!<\!b$, $\fD^{\bp}\cYp=\{\nA\!+\!\h\,q\frac{\nd}{\nd q}\}^{\bp}\cYp$ and $\ticCp^{(r)}_{\bp,\bs}\!=\!0$ for all $r\!\in\![|\bp|]$.
If $|\bp|\!<\!a$ and $L_i^+(\bd)\!\ge\!1$ for all $i\!\in\![a]$ and $\bd\!\in\!\La\!-\!\{0\}$,
then
$\fD^{\bp}\cYpp=\{\nA\!+\!\h\,q\frac{\nd}{\nd q}\}^{\bp}\cYpp$ and
$\ticCpp^{(r)}_{\bp,\bs}\!=\!0$ for all $r\!\in\![|\bp|]$.
\end{thm}
\footnotetext{Furthermore, $g_j\!=\!0$ if $b\!>\!0$ or $D_j(\bd)\!\in\!\{-1,0\}$ for all $\bd\!\in\!\La$ with $\nu_E(\bd)\!=\!0$.}
\begin{crl}\label{enuE0_crl}
If $\bp\!\in\!(\Z^{\ge 0})^k$ and $\max(|\bp|,1)\!<\!\nu_E(\bd)\!$ for all $\bd\!\in\!\La\!-\!\{0\}$, then
\begin{equation*}\begin{split}
\cZp_{\bp}(\h,q)=\left\{x+\h\,q\frac{\nd}{\nd q}\right\}^{\bp}\cYp(x,\h,q),\quad
\cZpp_{\bp}(\h,q)=\left\{x+\h\,q\frac{\nd}{\nd q}\right\}^{\bp}\cYpp(x,\h,q).
\end{split}\end{equation*}
\end{crl}
This follows from Theorem~\ref{cY_thm} and Remark~\ref{simplefDp_rmk}.

\begin{crl}\label{e2pt_crl}
Let $g_{\bp\bs}\!\in\!\Q[\al]$ be homogeneous polynomials such that $\sum\limits_{|\bp|+|\bs|\le N-k}\!g_{\bp\bs}\pr_1^*x^{\bp}\pr_2^*x^{\bs}$
is the equivariant Poincar\'{e} dual to the diagonal in $\X$, where $N\!-\!k$ is the complex dimension of $\X$. 
If $N\!>\!k$ and $\nu_E(\bd)\!>\!N\!-\!k$ for all $\bd\!\in\!\La\!-\!\{0\}$, then the two-point function $\cZp$
of \e_ref{eZ2ptdfn_e} is given by
\begin{equation*}
\cZp(\h_1,\h_2,q)=\frac{1}{\h_1+\h_2}\sum_{|\bp|+|\bs|\le N-k}
g_{\bp\bs}
\pr_1^*\left\{x\!+\!\h_1\,q\frac{\nd}{\nd q}\right\}^{\bp}\cYp(x,\h_1,q)\pr_2^*
\left\{x\!+\!\h_2\,q\frac{\nd}{\nd q}\right\}^{\bs}\cYpp(x,\h_2,q).
\end{equation*}
\end{crl}
This follows from Theorem~\ref{eZ2pt_thm} and Corollary~\ref{enuE0_crl}.
\begin{rmk}\label{weightedfD_rmk}
In the inductive construction of $\fD^{\bp}\cY$ with $\cY\!=\!\cYp$ or $\cY\!=\!\cYpp$, the first equation in \e_ref{fDdfn_e}
may be replaced by
$$\wt{\fD}^{\bp}\cY(\nA,\h,q)\!\equiv\!\sum_{i\in\supp(\bp)}c_{\bp;i}\left\{\nA_i\!+\!\h\,q\frac{\nd}{\nd q}\right\}
\fD^{\bp-e_i}\cY(\nA,\h,q)\!\in\!\Q_{\al}(\nA,\h)[[\La]],$$
for any tuple $(c_{\bp;i})_{i\in\supp(\bp)}$ of rational numbers with $\sum\limits_{i\in\supp(\bp)}c_{\bp;i}\!=\!1$.
The power series $\fD^{\bp}\cY$ defined by the second equation in \e_ref{fDdfn_e} in terms of the ``new weighted''
$\wt{\fD}^{\bp}\cY$ satisfy \ref{EP_noneq} and \ref{EP_Dpexp} with $\bD^{\bp}$
correspondingly ``weighted'' as in Remark~\ref{weightedbD_rmk}.
This follows by the same arguments as in the case when $c_{\bp;i}\!=\!\frac{1}{|\supp(\bp)|}$ for all $i\!\in\!\supp(\bp)$.
Therefore, \e_ref{eqtiC_e} continues to define power series $\ctC^{(r)}_{\bp,\bs}(\cY)$ in terms of the ``new weighted''
$\cC^{(\br)}_{\bp,\bs}(\cY)$.
The resulting ``weighted'' power series $\cY_{\bp}$ of \e_ref{cYp_e} do not depend on the ``weights''
$c_{\bp;i}$ as elements of $H^*_{\T^N}(\X)\Lau{\h}[[\La]]$ by the proof of Theorem~\ref{cY_thm}
outlined in Section~\ref{pfsoutline_sec}.
\end{rmk}
\begin{rmk}\label{euntwZ_rmk2}
We define an equivariant version of $Z^*$ in \e_ref{untwZ_e}.
Let
\begin{equation}
\cZ^*(\h_1,\h_2,Q)\!\equiv\!\sum_{\bd\in\La-0}\!Q^{\bd}\!\left(\ev_1\!\times\!\ev_2\right)_*\!
\left[\!\frac{\E(\cV_E)}{\left(\h_1\!-\!\psi_1\right)\left(\h_2\!-\!\psi_2\right)}\!\right],
\end{equation}
where $\ev_1,\ev_2\!:\!\ov\M_{0,2}(\X,\bd)\!\lra\!\X$. 
Since $\E(\cVp_E)\ev_1^*\E(E^+)\!=\!\E(\cV_E)\ev_1^*\E(E^-)$,
$$\cZp^*(\h_1,\h_2,Q)\pr_1^*\E(E^+)=\cZ^*(\h_1,\h_2,Q)\pr_1^*\E(E^-),$$
where $\cZp^*$ is obtained from $\cZp$ by disregarding the $Q^{\b0}$ term and $\pr_1\!:\!\X\!\times\!\X\!\lra\!\X$ is the projection
onto the first component. This together with Theorem~\ref{eZ2pt_thm} expresses $\cZ^*$
in terms of $\cZp_{\eta}$ and $\cZpp_{\eta}$ in the $E\!=\!E^+$ case.

Using an idea from \cite[Section~3.3]{CoZ}, we derive a formula for $\cZ^*$ in terms of one-point GW generating functions that holds in all cases.
Following \cite{CoZ}, we then show
how to express the latter in terms of explicit power series if $b\!>\!0$.
If $\pr_i\!:\!\X\!\times\!\X\!\lra\!\X$ and $g_{\bp\bs}\!\in\!\Q[\al]$ are as in Corollary~\ref{e2pt_crl},
then
\begin{equation}\label{euntwZ_e2}
\cZ^*(\h_1,\h_2,Q)\!=\!\frac{1}{\h_1+\h_2}\sum\limits_{|\bp|+|\bs|\le N-k}
g_{\bp\bs}\!
\left[\pr_1^*x^{\bp}\pr_2^*\cZ^*_{\bs}(\h_2,Q)\!+\!\pr_1^*\cZ^*_{\bp}(\h_1,Q)\pr_2^*\cZpp_{\bs}(\h_2,Q)\right],
\end{equation}
where
\begin{equation*}
\cZ^*_{\bp}(\h,Q)\equiv\sum_{\bd\in\La-0}Q^{\bd}
\ev_{1*}\left[\frac{\E(\cV_E)\ev_2^*x^{\bp}}{\h\!-\!\psi_1}\right]
\!\in\!H^*_{\T^N}\!\left(\X\right)[[\h^{-1},\La]].
\end{equation*}
This follows from Theorem~\ref{eZ2pt_thm}, using that
$\pr_1^*(\E(E^+)/\E(E^-))\!(\cZp\!-\!\llbr{\cZp\rrbr}_{Q;\b0})
=\!\cZ^*$, and
\begin{equation}\label{euntwZloc_e}
\pr_1^*\left(\frac{\E(E^+)}{\E(E^-)}\right)\sum\limits_{|\bp|+|\bs|\le N-k}g_{\bp\bs}\pr_1^*x^{\bp}
\pr_2^*(\cZpp_{\bs}(\h,Q)-x^{\bs})=
\sum\limits_{|\bp|+|\bs|\le N-k}g_{\bp\bs}\pr_1^*x^{\bp}\pr_2^*
\cZ^*_{\bs}(\h,Q).
\end{equation}
In the $\X\!=\!\P^{n-1}$ case, \e_ref{euntwZ_e2} is the GW version of \cite[(3.28)]{CoZ} and the proof of the $\X\!=\!\P^{n-1}$
case of \e_ref{euntwZloc_e} in \cite{CoZ}
extends to the case of an arbitrary toric manifold.

We give another proof of \e_ref{euntwZloc_e}, using
the Virtual Localization Theorem \e_ref{virloc_e} on $\ov\M_{0,2}(\X,\bd)$ as in Section~\ref{Sp_sec}.
We prove that \e_ref{euntwZloc_e} holds when restricted to $[I]\!\times\![J]$
for arbitrary $I,J\!\in\!\V$.
The left-hand side of \e_ref{euntwZloc_e} restricted to $[I]\!\times\![J]$ is
\begin{equation}\label{euntwZLHS_e}
\frac{\E(E^+)}{\E(E^-)}\Big|_{[I]}\sum_{|\bp|+|\bs|\le N-k}g_{\bp\bs}x^{\bp}\Big|_{[I]}
\sum\limits_{\bd\in\La-0}Q^{\bd}
\int_{\left[\ov\M_{0,2}(\X,\bd)\right]^{vir}}\frac{\E(\cVpp_E)\ev_2^*x^{\bs}\ev_1^*\phi_J}{\h\!-\!\psi_1}.
\end{equation}
The right-hand side of \e_ref{euntwZloc_e} restricted to $[I]\!\times\![J]$ is
\begin{equation}\label{euntwZRHS_e}
\sum\limits_{|\bp|+|\bs|\le N-k}g_{\bp\bs}x^{\bp}\Big|_{[I]}\sum\limits_{\bd\in\La-0}Q^{\bd}
\int_{\left[\ov\M_{0,2}(\X,\bd)\right]^{vir}}
\frac{\E(\cV_E)\ev^*_2x^{\bs}\ev_1^*\phi_J}{\h\!-\!\psi_1}.
\end{equation}
Since $\phi_I$ is the equivariant Poincar\'{e} dual of $[I]$,
\begin{equation*}
\sum\limits_{\bp,\bs}g_{\bp\bs}\pr_1^*x^{\bp}\pr_2^*x^{\bs}\big|_{[I]\times[J]}=
\int_{\De(\X)}\pr_1^*\phi_I\pr_2^*\phi_J=\int_{\X}\phi_I\phi_J=\phi_I(J)=0\qquad
\forall\,I\neq J\!\in\!\V,
\end{equation*}
where $\De(\X)\!\subset\!\X\!\times\!\X$ denotes the diagonal.
Thus, by the Virtual Localization Theorem~\e_ref{virloc_e}, a graph
$\Ga$ as in Section~\ref{Sp_sec} may contribute to \e_ref{euntwZLHS_e} or \e_ref{euntwZRHS_e} only if
its second marked point is mapped into $[I]$. 
Finally, \e_ref{euntwZloc_e} follows from the above since
$$\frac{\E(E^+)}{\E(E^-)}\Big|_{[I]}\E(\cVpp_E)\Big|_{\cZ_{\Ga}}=\E(\cV_E)\Big|_{\cZ_{\Ga}}$$
whenever $\cZ_{\Ga}\!\subset\!\ov\M_{0,2}(\X,\bd)$ is the $\T^N$-pointwise fixed locus corresponding to a graph $\Ga$
whose second marked point is mapped into $[I]$.

We next assume that $b\!>\!0$ and $\nu_E(\bd)\!\ge\!0$ for all $\bd\!\in\La$ and express $\cZ^*_{\bp}(\h,Q)$ 
in terms of explicit power series. Along with \e_ref{euntwZ_e2} and Theorem~\ref{cY_thm}, this will conclude the computation of $\cZ^*$. 

We define
\begin{equation}\label{newcY_e}
\hcY(\nA,\h,q)\!\equiv\!\!\sum_{\bd\in\La}\!q^{\bd}
u(\bd;\nA,\h)\prod\limits_{i=1}^a\prod\limits_{s=1}^{L^+_i(\bd)}\left(\sum_{r=1}^k\ell^+_{ri}\nA_r\!+\!s\h\right)
\prod\limits_{i=1}^b\prod\limits_{s=1}^{-L_i^-(\bd)}\left(\sum_{r=1}^k\ell^-_{ri}\nA_r\!-\!s\h\right).
\end{equation}
As $\hcY$ satisfies \e_ref{assumecS_e}, we may define $\fD^{\bp}\hcY$ and $\ctC^{(r)}_{\bp,\bs}\!\equiv\!\ctC^{(r)}_{\bp,\bs}(\hcY)$
by \e_ref{fDdfn_e} and \e_ref{eqtiC_e}.
We define $\hcY_{\bp}(\nA,\h,q)$ by the right-hand side of \e_ref{cYp_e} above, with $\cYp$ replaced by $\hcY$
and $\ticCp^{(r)}_{\bp,\bs}$ by $\ctC^{(r)}_{\bp,\bs}$. 
Let 
\begin{equation}\label{wtcY_e}
\wt{\cY}^*(\nA,\h,q)\equiv
\sum\limits_{\bd\in\La-0}q^{\bd}u(\bd;\nA,\h)
\prod\limits_{i=1}^a\prod\limits_{s=1}^{L^+_i(\bd)}\left(\sum_{r=1}^k\ell^+_{ri}\nA_r\!+\!s\h\right)
\prod\limits_{i=1}^b\prod\limits_{s=1}^{-L_i^-(\bd)-1}\left(\sum_{r=1}^k\ell^-_{ri}\nA_r\!-\!s\h\right).
\end{equation}
Define $\cE^{(\br)}_{\bp,\bs}\!\in\!\Q[\al][[\La]]$ by 
\begin{equation}\label{Ecoeff_e}
\left\{\nA+\h\,q\frac{\nd}{\nd q}\right\}^{\bp}\wt{\cY}^*(\nA,\h,q)\cong
\sum_{s=0}^{|\bp|-b}\sum_{|\br|=0}^{|\bp|-b-s}\cE^{(\br)}_{\bp,s}\nA^{\br}\h^s\quad\mod\,\h^{-1}.
\end{equation}
It follows that $\LR{\cE^{(\br)}_{\bp,\bs}}_{q;\bd}$
is a degree~$|\bp|\!-\!b\!-\!s\!-\nu_E(\bd)\!-\!|\br|$ homogeneous polynomial in~$\al$.
Then,
\begin{equation}\label{euntwZ_e3}
\cYp_{\bp}(x,\h,Q)=
\left\{x\!+\!\h\,q\frac{\nd}{\nd q}\right\}^{\bp}\cYp(x,\h,q)-
\E(E^-)\sum_{s=0}^{|\bp|-b}\sum_{|\br|=0}^{|\bp|-b-s}\cE^{(\br)}_{\bp,s}\h^s\hcY_{\br}(x,\h,q),
\end{equation}
where $\cYp_{\bp}$ is defined by \e_ref{cYp_e}; see Section~\ref{pfsoutline_sec} for a proof of \e_ref{euntwZ_e3}.

Whenever $b\!\ge\! 2$,
\begin{equation}\begin{split}\label{cZpstar_e1}
\cZ^*_{\bp}(\h,q)\!=\!
\E(E^+)
\left[\left\{x+\h\,q\frac{\nd}{\nd q}\right\}^{\bp}\wt{\cY}^*(x,\h,q)\!-\!\!\sum_{s=0}^{|\bp|-b}\sum_{|\br|=0}^{|\bp|-b-s}\!\!\cE^{(\br)}_{\bp,s}\h^s\hcY_{\br}(x,\h,q)\right].
\end{split}\end{equation}
If $b\!=\!1$, 
\begin{equation}\begin{split}\label{cZpstar_e2}
\cZ^*_{\bp}(\h,Q)\!=\!
\E(E^+)\ne^{-\frac{\E(E^-)f_0(q)}{\h}}\!\!
\left[\left\{x+\h\,q\frac{\nd}{\nd q}\right\}^{\bp}\wt{\cY}^*(x,\h,q)\!-\!\!\!\sum_{s=0}^{|\bp|-b}\sum_{|\br|=0}^{|\bp|-b-s}\!\!\!\!\cE^{(\br)}_{\bp,s}\h^s\hcY_{\br}(x,\h,q)\right]\\-
\frac{\E(E^+)x^{\bp}f_0(q)}{\h}\sum_{n=0}^{\i}\frac{1}{(n+1)!}\left[-\frac{\E(E^-)f_0(q)}{\h}\right]^n,
\end{split}\end{equation}
with $Q$ and $q$ related by the mirror map~\e_ref{mirrmap_e} and $f_0(q)\!\in\!\Q[[\La]]$ given by \e_ref{f0_e}.
Equations~\e_ref{cZpstar_e1} and \e_ref{cZpstar_e2} follow from $\cZ_{\bp}^*\!=\!\frac{\E(E^+)}{\E(E^-)}\left(\cZp_{\bp}\!-\!x^{\bp}\right)$,
\e_ref{ZpYp_e}, and \e_ref{euntwZ_e3}.
\end{rmk}
\section{Proofs}
\label{pfs_sec}
\subsection{Outline}
\label{pfsoutline_sec}
In this section, we prove Theorems~\ref{eZ2pt_thm} and \ref{cY_thm} and the identity \e_ref{euntwZ_e3}.
The proofs of the two theorems are in the spirit of the proof of mirror symmetry in \cite{Gi_equiv} and \cite{Gi_mirr} but with a twist.
Similarly to \cite{Gi_equiv} and \cite{Gi_mirr}, our argument revolves around the restrictions on power series imposed by certain
recursivity and polynomiality conditions. The concept of $C$-recursivity was first introduced in \cite{Gi_equiv} in the $\X\!=\!\P^{n-1}$
case, extended to an arbitrary $\X$ in \cite{Gi_mirr}, and re-defined in \cite{bcov0}; all these definitions involve an explicit collection $C$ of structure coefficients. Our concept of $C$-recursivity introduced in Definition~\ref{rec_dfn}
extends the notion of $C$-recursivity with an arbitrary collection of structure coefficients from the $\X\!=\!\P^{n-1}$
case considered in \cite{bcov0_ci} to an arbitrary $\X$.
The concept of (self-) polynomiality introduced in \cite{Gi_equiv} in the $\X\!=\!\P^{n-1}$ case and extended to an arbitrary $\X$
in \cite{Gi_mirr} was modified into the concept of mutual polynomiality for a pair of power series in the $\X\!=\!\P^{n-1}$
case in \cite{bcov0}; we extend the latter to an arbitrary $\X$ in Definition~\ref{MPC_dfn}.
By Proposition~\ref{unique_prp}, which extends \cite[Proposition~2.1]{bcov0} from the $\X\!=\!\P^{n-1}$ case
to an arbitrary $\X$, $C$-recursivity and mutual polynomiality impose severe restrictions on power series, more severe than the
restrictions imposed by recursivity and self-polynomiality as discovered in \cite{Gi_equiv}.

Analogous to \cite{bcov0} and \cite{bcov0_ci}, the proof of Theorem~\ref{cY_thm} relies on the one-point mirror
theorem of \cite{LLY3}. We begin by stating it.
The coefficient of $\al_j/\h$ for $j\!\in\![N]$ in the Laurent expansion of $\frac{1}{\Ip_{\b0}(q)}\cYp\big|_{\nA=0}$ at $\h^{-1}\!=\!0$
is given by
\begin{equation}\begin{split}\label{gj_e}
g_j(q)&\equiv\frac{\de_{b,0}}{\Ip_{\b0}(q)}\left[\sum_{\begin{subarray}{c}\bd\in\La,\,\nu_E(\bd)=0\\D_s(\bd)\ge 0\,\,\forall\,s\in[N]
\end{subarray}}\!\!q^{\bd}\,\frac{\prod\limits_{i=1}^a\left[L_i^+(\bd)!\right]}{\prod\limits_{r=1}^N\left[D_r(\bd)!\right]}\left(\sum_{s=1}^{D_j(\bd)}\frac{1}{s}\right)\right.\\
&\,\,\qquad\qquad+\left.\sum_{\begin{subarray}{c}\bd\in\La,\,\nu_E(\bd)=0\\D_j(\bd)<-1\\
D_s(\bd)\ge 0\,\,\forall\,s\in[N]-\{j\}\end{subarray}}\!\!\!\!\!\!\!
q^{\bd}(-1)^{D_j(\bd)}\left[-D_j(\bd)-1\right]!\frac{\prod\limits_{i=1}^a\left[L_i^+(\bd)!\right]}
{\prod\limits_{\begin{subarray}{c}s\in [N]-\{j\}\end{subarray}}\left[D_s(\bd)!\right]}\right].
\end{split}\end{equation}
\noindent By \cite[Theorem~4.7]{LLY3} together with \cite[Section~5.2]{LLY3},
if $\nu_E(\bd)\!\ge\!0$ for all $\bd\!\in\!\La$, then
\begin{equation}\label{Z'Y'_e}
\cZp_{\b0}(\h,Q)=\frac{1}{\Ip_{\b0}(q)}\ne^{-\frac{1}{\h}\left[G(q)+\sum\limits_{i=1}^kx_if_i(q)+\sum\limits_{j=1}^N\al_jg_j(q)\right]}
\cYp(x,\h,q),
\end{equation}
with $Q$ and $q$ related by the mirror map \e_ref{mirrmap_e}, $G$, $f_i$, and $g_j$ defined by \e_ref{G_e}, \e_ref{fi_e}, and
\e_ref{gj_e}.\footnote{See Appendix~\ref{LLY_a} for 
the correspondence between the relevant notation in
\cite{LLY3} and ours and detailed references within \cite{LLY3} indicating
how \cite[Theorem~4.7]{LLY3} together with \cite[Section~5.2]{LLY3} implies \e_ref{Z'Y'_e}.}
\begin{rmk}\label{cYcoeff_rmk}
By \e_ref{cY_e}, \e_ref{G_e}, and \e_ref{fi_e}, 
$$\Ip_{\b0}(q)G(q)\equiv\!\LR{\cYp(\nA,\h,q)\Big|_{\begin{subarray}{c}\al=0\\\nA=0\end{subarray}}}_{\h^{-1};1}\quad
\textnormal{and}\quad
\Ip_{\b0}(q)f_i(q)\equiv\!\LR{\cYp(\nA,\h,q)}_{\frac{\nA_i}{\h};1}\qquad\forall\,i\!\in\![k],$$
where $\LR{\quad}_{\h^{-1};1}$ and $\LR{\quad}_{\frac{\nA_i}{\h};1}$ denote the coefficients of $\h^{-1}$
and $\frac{\nA_i}{\h}$ respectively within the Laurent expansion around $\h^{-1}\!=\!0$ of the power series inside of the brackets. Thus,
$$\LR{\cYp(\nA,\h,q)}_{\h^{-1};1}\equiv \Ip_{\b0}(q)\left[G(q)+
\sum\limits_{i=1}^k\nA_if_i(q)+\sum_{j=1}^N\al_jg_j(q)\right].$$
\end{rmk}

Some of the proofs in this section also hold if we replace
$\Q$ by any field $R\!\supseteq\!\Q$.
Given such a field $R$,
let
\begin{equation*}
R_{\al}\equiv\Q_{\al}\otimes_{\Q}R=R[\al_1,\ldots,\al_N]_{\langle P:P\in\Q[\al]-0\rangle}\quad\textnormal{and}\quad
H^*_{\T^N}(\X;R)\equiv H^*_{\T^N}(\X)\otimes_{\Q}R.
\end{equation*}
An element in $H^*_{\T^N}\!(\X;R)\Lau{\h}[[\La]]$ admits a
lift to an element in $\!R[\al,x]\Lau{\h}[[\La]]$ and an element in $\!R[\al,x]\Lau{\h}[[\La]]$
induces an element in $H^*_{\T^N}\!(\X;R)\Lau{\h}[[\La]]$
via Proposition~\ref{equivcoh_prp}. 
Given $Y(\h,Q)\!\in\!H^*_{\T^N}(\X;R)\Lau{\h}[[\La]]$ and $J\!\in\!\V$, we write
$$Y(\h,Q)\big|_{[J]}\quad\textnormal{or}\quad Y(\h,Q)\big|_J\quad\textnormal{or}\quad Y(x(J),\h,Q)\in R[\al]\Lau{\h}[[\La]]$$
for the power series obtained from $Y$ by replacing each coefficient of $\h^{s} Q^{\bd}$ in $Y$ by its
image via the restriction map $\cdot(J)$ of \e_ref{dfnerestr_e}.

In proving Theorem \ref{cY_thm},
we follow the steps outlined in \cite[Section~1.3]{bcov0} and used for proving
\cite[Theorem~1.1]{bcov0}:
\begin{enumerate}[label=(M\arabic*),leftmargin=*]
\item\label{admis_item} if $R\!\supseteq\!\Q$ is any field, $Y,Z\!\in\!H^*_{\T^N}(\X;R)\Lau{\h}[[\La]]$,
$Z(\h,Q)$ is $C$-recursive in the sense of Definition~\ref{rec_dfn}
and satisfies the mutual polynomiality condition (MPC) of Definition~\ref{MPC_dfn}
with respect to $Y(\h,Q)$,
the transforms of $Z(\h,Q)$ of Lemma~\ref{admis_lmm} are also $C$-recursive
and satisfy the MPC with respect to appropriate transforms of $Y(\h,Q)$;
\item\label{unique_item} if $R\!\supseteq\!\Q$ is any field,
$Z\!\in\!H^*_{\T^N}(\X;R)\Lau{\h}[[\La]]$
is recursive in the sense of Definition~\ref{rec_dfn} and $(Y,Z)$ satisfies the MPC
for some $Y\!\in\!H^*_{\T}(\X;R)\Lau{\h}[[\La]]$ with $\LR{Y(\h,Q)\big|_I}_{Q;\b0}\!\in\!R_{\al}^*$ 
for all $I\!\in\!\V$, then 
$Z$ is determined by its `mod $\h^{-1}$ part' (see Proposition~\ref{unique_prp});
\item\label{Crec_item} $\cYp_{\bp}$ of \e_ref{cYp_e} and $\cZp_{\ze}$ of \e_ref{eZeta1ptdfn_e} are $\fCp$-recursive in the sense of Definition~\ref{rec_dfn} with $\fCp$ given by \e_ref{fC_e}, while
$\cYpp_{\bp}$ of \e_ref{cYp_e} and $\cZpp_{\ze}$ of \e_ref{eZeta1ptdfn_e} are $\fCpp$-recursive
with $\fCpp$ given by \e_ref{fC_e};
\item\label{MPC_item} $(\cYp,\cYpp_{\bp})$, $(\cYpp,\cYp_{\bp})$,
$(\cZp_1,\cZpp_{\ze})$, and $(\cZpp_1,\cZp_{\ze})$ satisfy the MPC;
\item\label{equalmod_item} the two sides of~\e_ref{ZpYp_e} viewed as  powers series in $\h^{-1}$,  agree mod~$\h^{-1}$. 
\end{enumerate}
The proof of Theorem~\ref{eZ2pt_thm} described below follows the same ideas and extends the proof of
\cite[(1.17)]{bcov0}.

Claims \ref{Crec_item} and \ref{MPC_item} concerning $\cZp_{\ze}$ and $\cZpp_{\ze}$
follow from Lemmas~\ref{ZCrec_lmm} and \ref{ZMPC_lmm}, since by the string equation of \cite[Section~26.3]{MirSym} and \e_ref{Zeb_e},
\begin{align*}
\cZp_{\ze}(\h,Q)\!=\!\h\cZp_{\eta,\be}(\h,Q)\quad\textnormal{and}\quad
\cZpp_{\ze}(\h,Q)\!=\!\h\cZpp_{\eta,\be}(\h,Q),
\end{align*}
if $m\!=\!3$, $\be_2\!=\!\be_3\!=\!0$, $\eta_2\!=\!\ze$, and $\eta_3\!=\!1$.

By Lemmas~\ref{YCrec_lmm}, \ref{YMPC_lmm}, and \ref{MPCsym_lmm},
$\cYp$ is $\fCp$-recursive and $\cYpp$ is $\fCpp$-recursive, while
$(\cYp,\cYpp)$ and $(\cYpp,\cYp)$ satisfy the MPC.
This together with the admissibility of transforms~\ref{deriv_lmm} and \ref{multpoly_lmm}
of Lemma~\ref{admis_lmm} proves claims \ref{Crec_item} and \ref{MPC_item} for $\cYp_{\bp}$ and $\cYpp_{\bp}$.

Claims \ref{Crec_item} and \ref{MPC_item} together with
\e_ref{Z'Y'_e}, the admissibility of transforms \ref{multexp_lmm} and \ref{comp_lmm}
of Lemma~\ref{admis_lmm}, and Proposition~\ref{unique_prp},
prove that
verifying \e_ref{ZpYp_e} amounts to showing that
the two sides of each of these equations agree $\mod \h^{-1}$;
this is in turn equivalent to \e_ref{eqtiC_e}.

Lemma~\ref{MPCsym_lmm}, Lemma~\ref{admis_lmm}, and Proposition~\ref{unique_prp}
are proved in Section~\ref{admis_sec};
the preparations for this section and the ones following it are made in Section~\ref{fixednot_sec}.
Lemmas~\ref{ZCrec_lmm} and \ref{ZMPC_lmm}
are proved in Sections~\ref{ZCrec_sec} and \ref{ZMPC_sec}, respectively.
Both proofs rely on the Virtual Localization Theorem~\cite[(1)]{GraPa}.
The localization data provided by \cite{Sp} is presented in Section~\ref{Sp_sec}.
Lemmas~\ref{YCrec_lmm} and \ref{YMPC_lmm} are proved in Section~\ref{Y_sec}.

\begin{proof}[Proof of \e_ref{euntwZ_e3}]
Define $E^-_{\bp}\!\in\!\Z$ with $\bp\!\in\!(\Z^{\ge 0})^k$ by
$$\prod\limits_{i=1}^b\left(\sum_{r=1}^k\ell^-_{ri}\nA_r\right)\equiv\sum\limits_{\bp\in(\Z^{\ge 0})^k}E^-_{\bp}\nA^{\bp}.$$
By \e_ref{cY_e} and \e_ref{newcY_e},
\begin{equation}\label{cY-_e}
\E(E^-)\hcY(x,\h,q)=\sum\limits_{\bp\in(\Z^{\ge 0})^k}E^-_{\bp}\left\{x\!+\!\h\,q\frac{\nd}{\nd q}\right\}^{\bp}\cYp(x,h,q).
\end{equation}
Since $\cYp$ is $\fCp$-recursive by Lemma~\ref{YCrec_lmm} and $(\cYpp,\cYp)$ satisfies the MPC by Lemmas~\ref{MPCsym_lmm}
and \ref{YMPC_lmm}, $\E(E^-)\hcY$ is $\fCp$-recursive and $(\cYpp,\E(E^-)\hcY)$ satisfies the MPC
by \e_ref{cY-_e} and Lemma~\ref{admis_lmm}\ref{deriv_lmm}.
This together with Lemma~\ref{admis_lmm}\ref{deriv_lmm}\ref{multpoly_lmm} implies that the right-hand side of \e_ref{euntwZ_e3} is $\fCp$-recursive and satisfies the MPC with respect to
$\cYpp$. Since $\cYp_{\bp}$ also satisfies these two properties by \ref{Crec_item} and \ref{MPC_item}, the claim follows from \ref{unique_item} and the fact that both sides of \e_ref{euntwZ_e3} are congruent to
$x^{\bp}$ modulo $\h^{-1}$. The latter follows from the fact that $\cYp_{\bp}(x,h,Q)$ and $\hcY_{\bp}(x,h,Q)$ are congruent to $x^{\bp}$ modulo $\h^{-1}$ by \e_ref{eqtiC_e}
together with \e_ref{cY_e}, \e_ref{wtcY_e}, and \e_ref{Ecoeff_e}.
\end{proof}
\begin{proof}[Proof of Theorem~\ref{eZ2pt_thm}]
By \e_ref{eZ2ptdfn_e}, \e_ref{phiprop_e}, and \e_ref{push-fwd_e},
\begin{equation}\label{initial_e}
(\h_1\!+\!\h_2)\cZp(\h_1,\h_2,Q)\Big|_{[I]\times[J]}\!=\!
\h_1\h_2\sum_{\bd\in\La}Q^{\bd}\int_{\left[\ov\M_{0,3}(\X,\bd)\right]^{vir}}
\frac{\E(\cVp_E)\ev_1^*\phi_I\ev_2^*\phi_J}
{(\h_1\!-\!\psi_1)(\h_2\!-\!\psi_2)}
\end{equation}
for all $I,J\!\in\V$.
Applying Lemmas~\ref{ZCrec_lmm} and \ref{ZMPC_lmm} for $\cZp_{\eta,\be}(\h_1,Q)$ with 
$$m\!=\!3,\quad\be_2\!=\!n,\quad\be_3\!=\!0,\quad\eta_2\!=\!\phi_J,\quad
\eta_3\!=\!1,$$
along with Lemma~\ref{admis_lmm}\ref{multpoly_lmm},
we obtain that the coefficient of $\h_2^{-n}$ in
$(\h_1\!+\!\h_2)\cZp(\h_1,\h_2,Q)$ is $\fCp$-recursive with $\fCp$ given by \e_ref{fC_e}
and satisfies the MPC
with respect to $\cZpp_1(\h_1,Q)$ for all $n\!\ge\!0$.
Using this, Proposition~\ref{equivcoh_prp}\ref{equivcohr_prp}, \ref{Crec_item}, \ref{MPC_item}, and \ref{unique_item},
it follows that in order to prove
\e_ref{eZ2pt_e} it suffices to show that 
\begin{equation}\label{2ptpf_e}
(\h_1\!+\!\h_2)\cZp(\h_1,\h_2,Q)\Big|_{[I]\times[J]}\cong\sum_{j=1}^s\cZp_{\eta_j}(\h_1,Q)\Big|_{[I]}\cZpp_{\wceta_j}(\h_2,Q)\Big|_{[J]}\quad\mod\h_1^{-1}
\end{equation}
for all $I,J\!\in\!\V$.
By \e_ref{initial_e} and the string equation, the left-hand side of \e_ref{2ptpf_e} $\mod \h_1^{-1}$ is 
\begin{equation}\begin{split}\label{LHS_e}
\phi_I(J)+\h_2\sum_{\bd\in\La-0}Q^{\bd}\int_{\left[\ov\M_{0,3}(\X,\bd)\right]^{vir}}
\frac{\E(\cVp_E)\ev_1^*\phi_I\ev_2^*\phi_J}
{\h_2\!-\!\psi_2}\qquad\qquad\qquad\\\qquad\qquad\quad=
\De_*1\Big|_{[I]\times[J]}+\sum_{\bd\in\La-0}Q^{\bd}\int_{\left[\ov\M_{0,2}(\X,\bd)\right]^{vir}}
\frac{\E(\cVp_E)\ev_1^*\phi_I\ev_2^*\phi_J}
{\h_2\!-\!\psi_2},
\end{split}\end{equation}
where $\De\!:\!\X\!\lra\!\X\!\times\!\X$, $\De[z]\!\equiv\!([z],[z])$.
The right-hand side of \e_ref{2ptpf_e} $\mod \h_1^{-1}$ is 
\begin{equation}\label{RHS_e}
\sum_{j=1}^s\eta_j\Big|_{[I]}\cZpp_{\wceta_j}(\h_2,Q)\Big|_{[J]}.
\end{equation}
Applying Lemmas~\ref{ZCrec_lmm} and \ref{ZMPC_lmm} for $\cZpp_{\eta,\be}(\h_2,Q)$ with 
$$m\!=\!3,\quad\be_2\!=\!\be_3\!=\!0,\quad\eta_2\!=\!\phi_I,\quad
\eta_3\!=\!1,$$ 
along with Lemma~\ref{admis_lmm}\ref{multpoly_lmm},
we obtain that \e_ref{LHS_e}
is the restriction to $[J]$ of a $\fCpp$-recursive formal power series which satisfies the MPC with respect to $\cZp_1(\h_2,Q)$. Since \e_ref{RHS_e} also satisfies these two properties, by Proposition~\ref{unique_prp} the power series
\e_ref{LHS_e} and \e_ref{RHS_e} agree if and only if they agree $\mod {\h_2^{-1}}$.
The latter is the case since \e_ref{RHS_e} $\mod \h^{-1}$ is the equivariant Poincar\'{e} dual to the diagonal in $\X\!\times\!\X$
restricted to the point $[I]\!\times\![J]$.
\end{proof}

\subsection{Notation for fixed points and curves}
\label{fixednot_sec}
With $\V$ as in \e_ref{Vset_e} and for all $I,J\!\in\!\V$ with $|I\!\cap\!J|\!=\!k\!-\!1$, we denote by 
\begin{equation}\label{IJdfn_e}
\ov{IJ}\equiv\X(I\cup J)\subseteq\X\quad\textnormal{and}\quad
\deg\overline{IJ}\equiv\left[\overline{IJ}\right]_{[\X]}\in\La
\end{equation}
the $\P^1$ passing through the points $[I]$ and $[J]$
and its homology class, respectively; see Corollary~\ref{fixed_crl}.
Given $I\!\in\!\V$ and $j\!\in\![N]\!-\!I$, 
we denote by 
$$\ov{Ij}\equiv\X\left(I\cup\{j\}\right)\quad\textnormal{and}\quad
\deg\ov{Ij}\equiv\left[\ov{Ij}\right]_{[\X]}\in\La$$
the compact one-dimensional complex submanifold of
$\X$ defined by Remark~\ref{submfld_rmk} and its
homology class, respectively.
Since $\X$ admits a K\"{a}hler form,
$$\deg{\overline{IJ}},\,\deg\ov{Ij}\in\La\!-\!\{0\}$$ by \cite[Chapter~0, Section~7]{GH}.
By the last part of Remark~\ref{submfld_rmk}, there exists a unique element $v(I,j)$ of $\V$
such that
$$v\left(I,j\right)\!\neq\!I\textnormal{ and }v\left(I,j\right)\!\subset\!I\!\cup\!\{j\}.$$
Since $v(I,j)\!\cup\!I\!=\!\{j\}\!\cup\!I$, 
$j\!\in\!v\left(I,j\right)$ and $\ov{Ij}\!=\!\overline{Iv(I,j)}$.
Let $\{\widehat{j}\}\!\equiv\!I\!-\!v(I,j)$.

Applying the Localization Theorem~\e_ref{ABo_e} to the integral of $1$ over $\ov{Ij}\!\cong\!\P^1$
and using \e_ref{tan_e} and Corollary~\ref{fixed_crl}, we find that
\begin{equation}\label{ufixed_e} 
u_j(I)\!+\!u_{\widehat{j}}\left(v\left(I,j\right)\right)\!=\!0
\qquad\forall\,I\!\in\!\V,j\!\in\![N]\!-\!I.
\end{equation}

Applying the Localization Theorem~\e_ref{ABo_e} to the integrals of $x_i$, $\la^{\pm}_i$, and $u_s$
over $\ov{Ij}$ and using Corollary~\ref{fixed_crl}, \e_ref{tan_e}, and \e_ref{ufixed_e}, we find that
\begin{alignat}{2}
\label{xfixed_e} x_i(I)\!-\!x_i\left(v\left(I,j\right)\right)&=\lr{\nH_i,\deg\ov{Ij}}u_j(I)
\qquad\forall\,I\!\in\!\V,j\!\in\![N]\!-\!I,i\!\in\![k],\\
\label{Lfixed_e} \la_i^{\pm}(I)\!-\!\la_i^{\pm}\left(v\left(I,j\right)\right)&=L_i^{\pm}\left(\ov{Ij}\right)u_j(I)
\qquad\forall\,I\!\in\!\V,j\!\in\![N]\!-\!I,i\!\in\![a]~ (i\!\in\![b]),\\
\label{ufixed2_e} u_s(I)\!-\!u_s\left(v\left(I,j\right)\right)&=D_s\left(\ov{Ij}\right)u_j(I)\qquad\forall\, I\!\in\!\V,j\!\in\![N]\!-\!I,s\!\in\![N].
\end{alignat}
By \e_ref{ufixed2_e}, \e_ref{ufixed_e}, \e_ref{fprestr_e},
and \e_ref{ual_e},
\begin{equation}\label{Dedge_e}
D_j\left(\ov{Ij}\right)\!=\!D_{\widehat{j}}\left(\ov{Ij}\right)\!=\!1,\qquad
D_s\left(\ov{Ij}\right)\!=\!0\quad\forall\,s\!\in\!I\!\cap\!v(I,j).
\end{equation}
The last five identities are stated in \cite{Gi_mirr}.
\subsection{Recursivity, polynomiality, and admissible transforms}
\label{admis_sec}

As in \cite{Gi_mirr}, we introduce a partial order on $\La$:
if $\bs,\bd\!\in\!\La$, we define 
$\bs\!\preceq\!\bd$ if $\bd\!-\!\bs\!\in\!\La$.
By Proposition~\ref{dimK_prp}, 
\begin{equation}\label{finite_e}
\bd\in\La\qquad\Lra\qquad\left\{\bs\in\La:\bs\preceq\bd\right\}\textnormal{ is finite}.
\end{equation}
This implies that
for every non-empty subset $S$ of $\La$, there exists
$\bd\!\in\!S$ such that
$$\bs\in\La,\bs\prec\bd\qquad\Lra\qquad\bs\notin S.$$

\begin{dfn}\label{rec_dfn}
Let $R\!\supseteq\!\Q$ be any field and
 $C\!\equiv\!(C_{I,j}(d))^{d\ge1}_{\begin{subarray}{c}
I\in\V,\,j\in [N]-I\end{subarray}}$ any collection of elements
of $R_{\al}$. A power series $Z\!\in\!H^*_{\T^N}(\X;R)\Lau{\h}[[\La]]$
is \sf{$C$-recursive} if the following holds: if ${\bd}^*\!\in\!\La$
is such that
$$\LR{Z\left(x(v(I,j)),\h,Q\right)}_{Q;\bd^*-d\cdot\deg\ov{Ij}}\!\in\!R_{\al}(\h)
\quad\forall\,I\!\in\!\V,j\!\in\![N]\!-\!I,d\!\ge\!1,$$
and
$\LR{Z\left(x(v(I,j)),\h,Q\right)}_{Q;\bd^*-d\cdot\deg\ov{Ij}}$ is regular at $\h\!=\!-u_j(I)/d$
for all $I\!\in\!\V$, $j\!\in\![N]\!-\!I$, and $d\!\ge\!1$, then
$$\LR{Z\left(x(I),\h,Q\right)}_{Q;\bd^*}\!-\!
\sum_{d\ge 1}\!\!\!\sum_{\begin{subarray}{c}j\in[N]-I\\d\cdot\deg\ov{Ij}\preceq\bd^*\end{subarray}}\!\!\!\!
\frac{C_{I,j}(d)}{\h+\frac{u_j(I)}{d}}
\LR{Z\left(x\left(v(I,j)\right),\h,Q\right)}_{Q;\bd^*-d\cdot\deg\ov{Ij}}\big|_{\h=-\frac{u_j(I)}{d}}
\!\in\!R_{\al}[\h,\h^{-1}],$$
for all $I\!\in\!\V$.
A power series $Z\!\in\!H^*_{\T^N}(\X;R)\Lau{\h}[[\La]]$ is called \sf{recursive} if it is C-recursive for some collection
$C\!\equiv\!\left(C_{I,j}(d)\right)^{d\ge 1}_{I\in\V,j\in[N]-I}$ of elements of $R_{\al}$.
\end{dfn}
By Remark~\ref{rec_rmk} below, if $Z\!\in\!H^*_{\T^N}(\X;R)\Lau{\h}[[\La]]$ is 
$(C_{I,j}(d))^{d\ge 1}_{I\in\V,j\in[N]-I}$-recursive, then for each $I\!\in\!\V$ 
\begin{equation*}\begin{split}
Z(x(I),\h,Q)=&\sum\limits_{\bd\in\La}\sum\limits_{r=-N_{\bd}}^{N_{\bd}}Z_{I;\bd}^{(r)}\h^{-r}Q^{\bd}\\&+
\sum\limits_{d=1}^{\i}\sum\limits_{\begin{subarray}{c}j\in[N]-I\end{subarray}}\!\!\!\!
\frac{C_{I,j}(d)Q^{d\cdot\deg{\ov{Ij}}}}{\h+\frac{u_j(I)}{d}}
Z\left(x\left(v(I,j)\right),-\frac{u_j(I)}{d},Q\right)
\end{split}\end{equation*}
for some integers $N_{\bd}$ and some $Z_{I;\bd}^{(r)}\!\in\!R_{\al}$.
\begin{rmk}\label{rec_rmk}
Let $R\!\supseteq\!\Q$ be any field.
If $Z\!\in\!H^*_{\T^N}(\X;R)\Lau{\h}[[\La]]$ is recursive, then $Z\big|_I\!\in\!R_{\al}(\h)[[\La]]$
and $\LR{Z(x(v(I,j)),\h,Q)}_{Q;\bd}$ is regular at $\h\!=\!\frac{-u_j(I)}{d}$
for all $I\!\in\!\V$, $\bd\!\in\!\La$, $j\!\in\![N]\!-\!I$, and
$d\!\ge\!1$;
this follows by induction on $\bd\!\in\!\La$. The regularity claim also uses Remark~\ref{non0u_e} below.

The $C$-recursivity is an $R_{\al}$-linear property (that is, if $Z_1$
and $Z_2$ are $C$-recursive, then so is $f_1Z_1+f_2Z_2$ for any $f_1,f_2\!\in\!R_{\al}$).
By Lemma~\ref{admis_lmm}\ref{multpoly_lmm}, $C$-recursivity is actually an $R_{\al}[\h][[\La]]$-linear property.
\end{rmk}
\begin{rmk}\label{non0u_rmk}
For all $I\!\in\!\V$, $j\!\in\![N]\!-\!I$, all $d\!\in\!\Q\!-\!\{1\}$, and all $s\!\in\![N]$,
\begin{equation*}\label{non0u_e}
u_j(I)+d\!\cdot\!u_s(v(I,j))\neq 0.
\end{equation*}
\end{rmk}
\begin{proof}
Assume that 
\begin{equation}\label{pfnon0u_e}
u_j(I)+d\!\cdot\!u_s(v(I,j))=0
\end{equation}
for some $I\!\in\!\V$, $j\!\in\![N]\!-\!I$, $d\!\in\!\Q\!-\!\{1\}$, and $s\!\in\![N]$.
If $d\!=\!0$ or $s\!\in\!v(I,j)$, then $u_j(I)\!=\!0$ by \e_ref{fprestr_e} which contradicts \e_ref{ual_e}.
If $d\!\neq\!0$ and $s\!\in\!(I\!-\!v(I,j))$, then $u_j(I)(1\!-\!d)\!=\!0$ by \e_ref{pfnon0u_e} and \e_ref{ufixed_e},
which again contradicts \e_ref{ual_e}.
If $d\!\neq\!0$ and $s\!\notin\!(I\!\cup\!v(I,j))$, then setting $\al_i\!=\!0$ for all $i\!\in\!(I\!\cup\!v(I,j))$
in \e_ref{pfnon0u_e} and using \e_ref{ual_e}, we find that $-d\al_s\!=\!0$, which is false. 
\end{proof}

For the purposes of Definition~\ref{MPC_dfn} and the transforms~\ref{deriv_lmm} and \ref{comp_lmm} in Lemma~\ref{admis_lmm}
below as well as all statements involving them, we identify $H_2(\X;\Z)$ with $\Z^k$ via the dual basis to $\{\nH_1,\ldots,\nH_k\}$
so that $\La\!\subset\!\Z^k$.
\begin{dfn}\label{MPC_dfn}
For any $Y\!\equiv\!Y(\h,Q),Z\!\equiv\!Z(\h,Q)\!\in\!H^*_{\T^N}(\X;R)\Lau{\h}[[\La]]$,
define $\Phi_{Y,Z}\!\in\!R_{\al}\Lau{\h}[[z,\La]]$ by
$$\Phi_{Y,Z}\left(\h,z,Q\right)\equiv
\sum_{I\in\V}
\frac{\ne^{x(I)\cdot z}}{\prod\limits_{j\in[N]-I}u_j(I)}
Y\left(x(I),\h,Q\ne^{\h z}\right)
Z\left(x(I),-\h,Q\right),$$
where $z\!\equiv\!(z_1,\ldots,z_k),\,
x(I)\cdot z\!\equiv\!\sum\limits_{i=1}^kx_i(I)z_i$, and
$Q\ne^{\h z}\!\equiv\!\left(Q_1\ne^{\h z_1},\ldots,Q_k\ne^{\h z_k}\right)$.

If $Y,Z\!\in\!H^*_{\T^N}(\X;R)\Lau{\h}[[\La]]$, the pair $(Y,Z)$
\sf{ satisfies the mutual polynomiality condition (MPC)} if
$\Phi_{Y,Z}\!\in\!R_{\al}[\h][[z,\La]]$.
\end{dfn}

\begin{prp}\label{unique_prp}
Let $R\!\supseteq\!\Q$ be a field.
Assume that $Z\!\in\!H^*_{\T^N}(\X;R)\Lau{\h}[[\La]]$ is recursive and that
$(Y,Z)$ satisfies the MPC for some
$Y\!\in\!H^*_{\T^N}(\X;R)\Lau{\h}[[\La]]$ with
$$\LR{Y\left(\h,Q\right)\big|_I}_{Q;\b0}\in R_{\al}^*\qquad\forall\,I\in\V.$$
Then,
$Z(\h,Q)\cong 0~~\mod\,\h^{-1}$ if and only if $Z(\h,Q)\!=\!0$.
\end{prp}
\begin{proof}
By the second statement in Proposition~\ref{equivcoh_prp}\ref{equivcohr_prp}, 
$$Z(\h,Q)=0\quad\Llra\quad Z(\h,Q)\big|_{I}=0\qquad\forall\,I\in\V.$$
Set $f_I\!\equiv\!\LR{Y\left(-\h,Q\right)\big|_I}_{Q;\b0}$ and assume that $\LR{Z(\h,Q)\big|_I}_{Q;\bd'}\!=\!0$
for all $0\!\preceq\!\bd'\!\prec\!\bd$ and all $I\!\in\!\V$.
Since $Z$ is recursive and $Z(\h,Q)\cong 0$ modulo $\h^{-1}$, 
$$\LR{Z(\h,Q)\big|_{I}}_{Q;\bd}=\sum\limits_{r={1}}^{N_{\bd}}Z^{(r)}_{I;\bd}\h^{-r}$$
for some $N_{\bd}\!\ge\!0$ and some $Z^{(r)}_{I;\bd}\!\in\!R_{\al}$.
Thus, 
$$\LR{\Phi_{Y,Z}(-\h,z,Q)}_{Q;\bd}\!=\!\!\sum_{I\in\V}\!\frac{\ne^{x(I)\cdot z}}{\prod\limits_{j\in [N]-I}u_j(I)}
f_I\left(\sum\limits_{r=1}^{N_{\bd}}Z^{(r)}_{I;\bd}\h^{-r}\right)\!\in\!R_{\al}[\h][[z]].$$
This implies that
$$\sum\limits_{I\in\V}\frac{\left(x(I)\cdot z\right)^m}{\prod\limits_{j\in [N]-I}u_j(I)}
f_I\left(\sum\limits_{r=1}^{N_{\bd}}Z^{(r)}_{I;\bd}\h^{-r}\right)\!\in\!R_{\al}[\h,z]
\qquad\forall\,m\!\ge\!0.$$
In particular,
$$\sum\limits_{I\in\V}\frac{\left(x(I)\cdot z\right)^m}{\prod\limits_{j\in [N]-I}u_j(I)}
f_IZ^{(r)}_{I;\bd}\!=\!0\qquad\forall\,0\!\le\!m\!\le\!|\V|\!-\!1,\,\forall\,r\!\in\![N_{\bd}].$$
For each $r\!\in\![N_{\bd}]$, this is a linear system in the `unknowns' $f_IZ^{(r)}_{I;\bd}\Big/\!\!\!\!\prod\limits_{j\in [N]-I}u_j(I)$
with $I\!\in\!\V$. Its coefficient matrix has a non-zero Vandermonde determinant, since 
$$x(I)\!\neq\!x(J)\qquad\forall\,I\!\neq\!J\!\in\!\V$$
by Proposition~\ref{equivcoh_prp}\ref{erestr_prp}.
It follows that $Z^{(r)}_{I;\bd}\!=\!0$ for all $I\!\in\!\V$ and all $r\!\in\![N_{\bd}]$. 
\end{proof}
Lemmas~\ref{MPCsym_lmm} and \ref{admis_lmm} below extend
\cite[Lemmas~2.2, 2.3]{bcov0} from the $\X\!=\!\P^{n-1}$ case to an arbitrary $\X$. 
Our proof of the former is completely different from and much simpler than the one in \cite{bcov0}.
For the latter, the arguments in \cite{bcov0} go through with only two significant changes
required.
\begin{lmm}\label{MPCsym_lmm}
Let $R\!\supseteq\!\Q$ be a field and
$Y,Z\!\in\!H^*_{\T^N}(\X;R)\Lau{\h}[[\La]]$.
Then,
$$\Phi_{Y,Z}\!\in\!R_{\al}[\h][[z,\La]]\quad\Llra\quad
\Phi_{Z,Y}\!\in\!R_{\al}[\h][[z,\La]].$$
\begin{proof}
Let $Y_{\bd}(\h)\!\equiv\!\LR{Y(\h,Q)}_{Q;\bd}$ and $Z_{\bd}(\h)\!\equiv\!\LR{Z(\h,Q)}_{Q;\bd}$. 
It follows that $\LR{\Phi_{Y,Z}(\h,z,Q)}_{Q;\bd}$ is
\begin{equation*}\begin{split}
\sum_{\begin{subarray}{c}0\preceq \bd'\preceq \bd\\I\in\V\end{subarray}}\!\!
\frac{\ne^{x(I)\cdot z}}{\prod\limits_{j\in[N]-I}u_j(I)}Y_{\bd'}(\h)\Big|_I\ne^{\h z \bd'}
Z_{\bd-\bd'}(-\h)\Big|_I\!=\!\!\!
\sum_{\begin{subarray}{c}0\preceq\bd'\preceq\bd\\I\in\V\end{subarray}}\!\!
\frac{\ne^{x(I)\cdot z}}{\prod\limits_{j\in[N]-I}u_j(I)}Y_{\bd-\bd'}(\h)\Big|_I\ne^{\h z (\bd-\bd')}
Z_{\bd'}(-\h)\Big|_I,\end{split}\end{equation*}
where $\ne^{\h z}\!\equiv\!(\ne^{\h z_1},\ldots,\ne^{\h z_k})$.
The right-hand side is $\ne^{\h z \bd}$ times $\LR{\Phi_{Z,Y}(-\h,z,Q)}_{Q;\bd}$.
\end{proof}
\end{lmm}
\begin{lmm}\label{admis_lmm}
Let $R\!\supseteq\!\Q$ be any field and
$C\!\equiv\!(C_{I,j}(d))^{d\ge1}_{\begin{subarray}{c}
I\in\V,\,j\in[N]-I\end{subarray}}$ any collection of elements of $R_{\al}$.
Let $Y_1,Y_2,Y_3\!\in\!H^*_{\T^N}(\X;R)\Lau{\h}[[\La]]$.
If $Y_1$ is $C$-recursive and $(Y_2,Y_3)$ satisfies the MPC, then
\begin{enumerate}[label=(\emph{\alph*}),leftmargin=*]
\item\label{deriv_lmm} if $\overline{Y_i}\!\equiv\!
\left\{x_s\!+\!\h\,Q_s\frac{\nd}{\nd Q_s}\right\}Y_i$
for all $i\!$ and $s\!\in\![k]$, then $\overline{Y_1}$ is $C$-recursive
and $\Phi_{Y_2,\overline{Y_3}}\!\in\!R_{\al}[\h][[z,\La]]$;
\item\label{multpoly_lmm} if $f\!\in\!R_{\al}[\h][[\La]]$,
then $fY_1$ is $C$-recursive and 
$\Phi_{Y_2,fY_3}\!\in\!R_{\al}[\h][[z,\La]]$;
\item\label{multexp_lmm} if $f\!\in\!R_{\al}[[\La\!-\!0]]$
and $\overline{Y_i}\!\equiv\!\ne^{f/\h}Y_i$ for all $i$, then 
$\overline{Y_1}$ is $C$-recursive and
$\Phi_{\overline{Y_2},\overline{Y_3}}\!\in\!R_{\al}[\h][[z,\La]]$;
\item\label{comp_lmm} if $f_r\!\in\!R_{\al}[[\La\!-\!0]]$ for all $r\!\in\![k]$ and
$\overline{Y_i}(\h,Q)\!\equiv\!\ne^{f\cdot x/\h}Y_i(\h,Q\ne^{f})$ for all $i$,
where
$f\cdot x\!\equiv\!\sum\limits_{r=1}^kf_rx_r$ and 
$Q\ne^{f}\!\equiv\!(Q_1\ne^{f_1},\ldots,Q_k\ne^{f_k})$,
then $\overline{Y_1}$ is $C$-recursive and 
$\Phi_{\overline{Y_2},\overline{Y_3}}\!\in\!R_{\al}[\h][[z,\La]]$.
\end{enumerate}
\end{lmm}
\begin{proof}
For all $I\!\in\!\V$,
\begin{align*}
\left\{x_s\!(I)+\!\h\,Q_s\frac{\nd}{\nd Q_s}\right\}
\left(\frac{C_{I,j}(d)}{\h\!+\!\frac{u_j(I)}{d}}
Q^{d\cdot\deg\ov{Ij}}
Y_1\left(x\left(v(I,j)\right),-\frac{u_j(I)}{d},Q\right)\right)\!=\qquad\qquad\qquad\qquad\qquad\\
\frac{C_{I,j}(d)}{\h\!+\!\frac{u_j(I)}{d}}Q^{d\cdot\deg\ov{Ij}}
\overline{Y_1}\left(x\left(v(I,j)\right),-\frac{u_j(I)}{d},Q\right)\!+\!
\frac{C_{I,j}(d)}{\h\!+\!\frac{u_j(I)}{d}}Q^{d\cdot\deg\ov{Ij}}\qquad\qquad\qquad\qquad\qquad\\\times
\left(\left(\h\!+\!\frac{u_j(I)}{d}\right)Q_s\frac{\nd}{\nd Q_s}\!+\!
\h\,d\!\cdot\!\deg_s\ov{Ij}\!+\!
x_s(I)-x_s\left(v(I,j)\right)\right)Y_1\left(x\left(v(I,j)\right),-\frac{u_j(I)}{d},Q\right).
\end{align*}
The first claim in \ref{deriv_lmm} now follows from Remark~\ref{rec_rmk} and \e_ref{xfixed_e}.
The second claim in \ref{deriv_lmm} and the claims in \ref{multpoly_lmm}-\ref{comp_lmm} 
follow similarly to the proof of \cite[Lemma~2.3]{bcov0} for 
the $X^{\tau}_M\!=\!\P^{n-1}$ case, using Lemma~\ref{MPCsym_lmm}, Remark~\ref{rec_rmk},
\e_ref{finite_e}, and \e_ref{xfixed_e}.
Equation~\e_ref{xfixed_e} and property~\e_ref{finite_e} are used in the proof of the recursivity claim in \ref{comp_lmm}
when showing that
$$\frac{1}{\h+\frac{u_j(I)}{d}}\left(\ne^{df(Q)\cdot\deg\ov{Ij}+\frac{f(Q)x(I)}{\h}}-\ne^{\frac{-f(Q)x(v(I,j))d}{u_j(I)}}\right)\in R_{\al}[\h,\h^{-1}][[\La]].$$
Property~\e_ref{finite_e} is also used to show that
transforms \ref{multexp_lmm} and \ref{comp_lmm} preserve $H^*_{\T^N}(\X;R_{\al})[\h,\h^{-1}][[\La]]$, that
$$\frac{\ne^{f/\h}-\ne^{-df/u_j(I)}}{\h+\frac{u_j(I)}{d}}\in R_{\al}[\h,\h^{-1}][[\La]],
\quad \ne^{\frac{f(Q\ne^{\h z})-f(Q)}{\h}}\in R_{\al}[\h][[z,Q]],$$
in the case of \ref{multexp_lmm}, and that
$$z_r+\frac{f_r(Q\ne^{\h z})-f_r(Q)}{\h}\in R_{\al}[\h,z][[\La]]\qquad\forall\,r\!\in\![k]$$
in the case of \ref{comp_lmm}.
\end{proof}

\subsection{Torus action on the moduli space of stable maps}
\label{Sp_sec}
An action of $\T^N$ on a smooth projective variety $X$ induces an action on $\ov\M_{0,m}(X,\bd)$
as in Section~\ref{equiv_sec} and an integration along the fiber homomorphism as in Section~\ref{equiv-intro_sec}. The Virtual Localization Theorem~\cite[(1)]{GraPa}
implies that
\begin{equation}\label{virloc_e}
\int_{\left[\ov\M_{0,m}(X,\bd)\right]^{vir}}\eta=\!\!\!\!\!\!
\sum_{F\subseteq\ov\M_{0,m}(X,\bd)^{\T^N}}\int_{[F]^{vir}}
\frac{\eta}{\E(N^{vir}_{F})}\in\Q[\al]\qquad\forall\,\eta\in H^*_{\T^N}\left(\ov\M_{0,m}(X,\bd)\right),
\end{equation}
where the sum runs over the components of the $\T^N\!$ pointwise fixed locus
$$\ov\M_{0,m}(X,\bd)^{\T^N}\subseteq\ov\M_{0,m}(X,\bd).$$
This section describes $\ov\M_{0,m}(X^{\tau}_M,d)^{\T^N}\!\!$,
the equivariant Euler class $\E(N^{vir}_{F})$ of the virtual normal bundle to each component $F$ of $\ov\M_{0,m}(X^{\tau}_M,d)^{\T^N}\!\!$,
and the restriction of $\E(\cV_E)$ to $F$.
We follow \cite{Sp} where the corresponding statements are formulated in the language of fans rather than toric pairs.

If $f\!:\!(\Si,z_1,\ldots,z_m)\!\lra\!\X$ is a $\T^N$\!-fixed stable map,
then the images of its marked points, nodes, contracted components, and ramification points are 
$\T^N$\!-fixed points and so points of the form $[I]$ for some $I\!\in\!\V$
by Corollary~\ref{fixed_crl}\ref{fpt_crl}.
Each non-contracted component $\Si_e$ of $\Si$ maps to a closed $\T^N$\!-fixed curve which is of the form 
$\overline{IJ}$ for some $I,J\!\in\!\V$ with $|I\!\cap\!J|\!=\!k\!-\!1$ by Corollary~\ref{fixed_crl}\ref{fc_crl}.
Since all such curves $\overline{IJ}$ are biholomorphic to $\P^1$ by Corollary~\ref{fixed_crl}\ref{fc_crl},
the map
$$f\Big|_{\Si_e}\!:\Si_{e}\lra\overline{IJ}$$
is a degree $\fd(e)$ covering map ramified only over $[I]$ and $[J]$. 
To each such map we associate a decorated graph as in Definition~\ref{graph_dfn} below;
the vertices of this graph correspond to the nodes and contracted components of $\Si$ or the ramification points of $f$;
the edges $e$ correspond to non-contracted components $\Si_e$ of $\Si$, and $\fd(e)$ describes the degree of $f\big|_{\Si_e}$.
\begin{dfn}\label{graph_dfn}
A genus~$0$ $m$-point decorated graph $\Ga$
is a collection of vertices $\Ver(\Ga)$, edges $\Edg(\Ga)$, and maps
$$\fd:\Edg(\Ga)\lra\Z^{>0},\qquad\fp:\Ver(\Ga)\lra\V,\qquad\dec:[m]\lra\Ver(\Ga)$$
satisfying the following properties:
\begin{enumerate}
\item the underlying graph $(\Ver(\Ga),\Edg(\Ga))$ has no loops;
\item if two vertices $\nv$ and $\nv'$ are connected by an edge, then
$|\fp(\nv)\!\cap\fp(\nv')|\!=\!k\!-\!1$.
\end{enumerate}
Such a decorated graph is said to be of degree $\bd\!\in\!\La$ if 
$$\sum_{\begin{subarray}{c}e\in\Edg(\Ga)\\\partial e=\{\nv,\nv'\}\end{subarray}}
\fd(e)\deg\left(\overline{\fp(\nv)\fp(\nv')}\right)\!=\!\bd,$$
where $\partial e\!\equiv\!\{\nv,\nv'\}$ for an edge $e$ joining vertices $\nv$ and $\nv'$.
\end{dfn}
For a decorated graph $\Ga$ as in Definition~\ref{graph_dfn}, we denote by $\Aut(\Ga)$
the group of automorphisms of $(\Ver(\Ga),\Edg(\Ga))$. It acts naturally on
$\prod\limits_{e\in\Edg(\Ga)}\Z_{\fd(e)}$; let
$$A_{\Ga}\equiv\prod\limits_{e\in\Edg(\Ga)}\Z_{\fd(e)}\rtimes\Aut(\Ga)$$
denote the corresponding semidirect product.

For any $\nv\!\in\!\Ver(\Ga)$, let 
$$\Edg(\nv)\equiv|\{e\in\Edg(\Ga):\nv\in\partial e\}|\quad\textnormal{and}\quad\val(\nv)\equiv|\dec^{-1}(\nv)|+
\Edg(\nv)$$
denote the number of edges to which the vertex $\nv$ belongs and its valence, repectively.
A flag $F$ in $\Ga$ is a pair $(\nv,e)$, where $e$ is an edge and $\nv$ is a
vertex of $e$. For a flag $F\!=\!(\nv,e)$,  let 
$\val(F)\!\equiv\!\val(\nv)$. For a flag $F\!=\!(\nv,e)$, let
$\om_F\!\equiv\!\E(T_{f^{-1}(\fp(\nv))}\P^1)$, where
$f\!:\!\P^1\!\lra\!\overline{\fp(\nv)\fp(\nv')}$
is the degree~$\fd(e)$ cover of $\overline{\fp(\nv)\fp(\nv')}$
corresponding to $e$, $\partial e\!=\!\{\nv,\nv'\}$, and the $\T^N$\!-action on $\P^1$ is induced from the action
on $X^{\tau}_M$ via $f$. If $\{j\}\!\equiv\!\fp(\nv')\!-\!\fp(\nv)$,  
\begin{equation}\label{omF_e}
\om_F=\frac{u_j(\fp(\nv))}{\fd(e)}
\end{equation}
by \e_ref{tan_e}.
If $\nv$ is a vertex that belongs to exactly $2$ edges $e_1$ and $e_2$, then
we write $F_i(\nv)\!\equiv\!(\nv,e_i)$.

Given a decorated graph $\Ga$ as above, let
$$\M_{\Ga}\!\equiv\!\prod_{\nv\in\Ver(\Ga)}
\ov\M_{0,\val(\nv)},$$
where $\ov\M_{0,m}\!\equiv\!\textnormal{point}$, whenever $m\!\le\!2$.
For a flag $F\!=\!(\nv,e)$, let $\psi_F\!\in\!H^2_{\T^N}(\M_{\Ga})$ denote the 
equivariant Euler class of the universal cotangent line bundle on $\M_{\Ga}$
corresponding to $F$ (that is, the pull-back of the $\psi$ class on $\ov\M_{0,\val(\nv)}$
corresponding to $e$).
\begin{prp}[{\cite[Lemma~6.9]{Sp}}]\label{flocus_prp}
There is a morphism $\ga\!:\!\M_{\Ga}\!\lra\!\ov\M_{0,m}(\X;\bd)$
whose image is a component of $\ov\M_{0,m}(\X;\bd)^{\T^N}$
and every such component occurs as the image
of such a morphism corresponding to some degree~$\bd$ decorated graph. 
With $\prod\limits_{e\in\Edg(\Ga)}\!\Z_{\fd(e)}$ acting trivially on
$\M_{\Ga}$, the induced map $$\ga/A_{\Ga}:\M_{\Ga}/A_{\Ga}\lra\ov\M_{0,m}(\X,\bd)$$
identifies $\M_{\Ga}/A_{\Ga}$ with the corresponding component of $\ov\M_{0,m}(\X,\bd)^{\T^N}$.
\end{prp}
\begin{prp}[{\cite[Theorem~7.8]{Sp}}]\label{vnormal_prp}
Let $\Ga$ be a degree~$\bd$ genus~$0$ $m$-point decorated graph and $N_{\Ga}^{vir}$ the virtual normal bundle
to $\ga\!:\!\M_{\Ga}\!\lra\!\ov\M_{0,m}(\X,\bd)$.
Then,
\begin{align*}
\E\left(N_{\Ga}^{vir}\right)\!=\!\prod\limits_{\begin{subarray}{c}\textnormal{flags}\,F\,\textnormal{of}\,\Ga\\
\val(F)\ge 3\end{subarray}}\!\left(\om_F\!-\!\psi_F\right)
\frac{1}{\prod\limits_{\nv\in\Ver(\Ga)}\left[\phi_{\fp(\nv)}(\fp(\nv))\right]^{\Edg(\nv)-1}}
\prod\limits_{\begin{subarray}{c}\nv\in\Ver(\Ga)\\
\val(\nv)=2\\
\dec^{-1}(\nv)=\emptyset\end{subarray}}
\!\!\!\!\!\!\!
\left(\om_{F_1(\nv)}\!+\!\om_{F_2(\nv)}\right)\frac{1}{\prod\limits_{\begin{subarray}{c}\textnormal{flags}\,F\,\textnormal{of}\,\Ga\\\val(F)=1\end{subarray}}\om_F}
\\
\times\!\!\!\prod_{\begin{subarray}{c}e\in\Edg(\Ga)\\\partial e=\{\nv,\nv'\}\end{subarray}}\!\!
\left(
\frac{(-1)^{\fd(e)}\left(\fd(e)!\right)^2\left(u_j(I)\right)^{2\fd(e)}}{(\fd(e))^{2\fd(e)}}\!\!\!\!\!
\prod_{\begin{subarray}{c}r\in[N]-\left(I\cup\{j\}\right)\end{subarray}}\!\!
\frac{\prod\limits_{s=0}^{\fd(e)D_r(\ov{Ij})}\!\!\!\left(u_{r}(I)\!-\!\frac{s}{\fd(e)}u_j(I)\right)}
{\prod\limits_{s=\fd(e)D_r(\ov{Ij})+1}^{-1}\!\!\!\left(u_{r}(I)\!-\!\frac{s}{\fd(e)}u_j(I)\right)}
\right)\left|_{\begin{subarray}{c}I=\fp(\nv)\\\{j\}=\fp(\nv')-I\end{subarray}}\right..
\end{align*}
\end{prp}
By \e_ref{ufixed_e} and \e_ref{ufixed2_e},
\begin{equation*}\begin{split}
(-1)^{\fd(e)}\left(u_j(I)\right)^{2\fd(e)}\left|_{\begin{subarray}{c}I=\fp(\nv)\\\{j\}=\fp(\nv')-I\end{subarray}}\right.=
u^{\fd(e)}_{\fp(\nv')-\fp(\nv)}(\fp(\nv))u^{\fd(e)}_{\fp(\nv)-\fp(\nv')}(\fp(\nv')),\qquad\qquad\\
u_{r}(I)\!-\!\frac{s}{\fd(e)}u_j(I)\left|_{\begin{subarray}{c}I=\fp(\nv)\\\{j\}=\fp(\nv')-I\end{subarray}}\right.=
\begin{cases}
\frac{
\left[\fd(e)D_r\left(\ov{\fp(\nv)\fp(\nv')}\right)-s\right]u_r(\fp(\nv))+su_r(\fp(\nv'))}
{\fd(e)D_r\left(\ov{\fp(\nv)\fp(\nv')}\right)}\quad&\hbox{if}\quad D_r\left(\ov{\fp(\nv)\fp(\nv')}\right)\!\neq\!0,\\
u_r\left(\fp(\nv)\right)=u_r\left(\fp(\nv')\right)\quad&\hbox{if}\quad D_r\left(\ov{\fp(\nv)\fp(\nv')}\right),s\!=\!0;
\end{cases}
\end{split}\end{equation*}
so the edge contributions to $\E(N^{vir}_{\Ga})$ in Proposition~\ref{vnormal_prp} are indeed symmetric in the vertices of each edge.

Let $f\!:\!(\P^1,z_1,\ldots,z_m)\!\lra\!\overline{IJ}$ be a $\T^N$\!-fixed
stable map. Thus, $f$ is a degree $d$ cover of $\overline{IJ}$ for some $d\!\in\!\Z^{>0}$.
By \e_ref{cVdfn_e},
$$\cV_E\big|_{[\P^1,z_1,\ldots,z_m,f]}=H^0\left(\P^1,f^*E^+\right)\oplus H^1\left(\P^1,f^*E^-\right).$$
By \cite[Exercise~27.2.3]{MirSym} together with \e_ref{ufixed_e} and \e_ref{Lfixed_e},
and with $\{j\}\equiv J\!-\!I$,
\begin{equation}\label{cV0_e}
\E(\cV_E)\big|_{[\P^1,z_1,\ldots,z_m,f]}=\prod_{i=1}^a\prod_{s=0}^{dL_i^+(\overline{IJ})}
\left[\la_i^+(I)-\frac{s}{d}u_j(I)\right]
\prod_{i=1}^b\prod_{s=dL_i^-(\ov{IJ})+1}^{-1}\left[\la^-_i(I)-\frac{s}{d}u_j(I)\right].
\end{equation}
By \e_ref{Lfixed_e},
\begin{equation*}\begin{split}\la_i^+(I)\!-\!\frac{s}{d}u_{J-I}(I)&=\begin{cases}
\frac{\left[dL_i^+(\ov{IJ})-s\right]\la_i^+(I)+s\la_i^+(J)}
{dL_i^+(\ov{IJ})}\quad&\hbox{if}\quad L_i^+(\ov{IJ})\!\neq\!0,\\
\la_i^+(I)=\la_i^+(J)\quad&\hbox{if}\quad L_i^+(\ov{IJ})\!=\!s\!=\!0,
\end{cases}\\
\la_i^-(I)\!-\!\frac{s}{d}u_{J-I}(I)&=
\frac{\left[dL_i^-(\ov{IJ})-s\right]\la_i^-(I)+s\la_i^-(J)}{dL_i^-(\ov{IJ})}.
\end{split}\end{equation*}

\subsection{Recursivity for the GW power series}
\label{ZCrec_sec}
For all $d\!\in\!\Z^{>0},I\!\in\!\V,j\!\in\![N]\!-\!I$, let
\begin{equation}\label{tifC_e}
\wt{\fC}_{I,j}(d)\equiv \frac{(-1)^dd^{2d-1}}{\left(d!\right)^2}\frac{1}{\left[u_j(I)\right]^{2d-1}}
\prod\limits_{\begin{subarray}{c}r\in[N]-(I\cup\{j\})\end{subarray}}\!\frac{\prod\limits_{s=dD_r(\ov{Ij})+1}^0
\left[u_r(I)\!-\!\frac{s}{d}u_j(I)\right]}
{\prod\limits_{s=1}^{ d D_r(\ov{Ij})}
\left[u_r(I)\!-\!\frac{s}{d}u_j(I)\right]}
\in\Q_{\al},
\end{equation}
\begin{equation}\begin{split}\label{fC_e}
\fCp_{I,j}(d)&\equiv\wt{\fC}_{I,j}(d)
\!\prod_{i=1}^a\!\prod_{s=1}^{dL_i^+(\ov{Ij})}\!\left[\la_i^+(I)\!-\!\frac{s}{d}u_j(I)\right]
\prod_{i=1}^b\!\prod_{s=0}^{-dL_i^-(\ov{Ij})-1}\!\left[\la_i^-(I)\!+\!\frac{s}{d}u_j(I)\right]\!\in\!\Q_{\al},\\
\fCpp_{I,j}(d)&\equiv
\wt{\fC}_{I,j}(d)
\prod\limits_{i=1}^a\!\prod\limits_{s=0}^{dL_i^+(\ov{Ij})-1}\left[\la_i^+(I)\!-\!\frac{s}{d}u_j(I)\right]
\prod\limits_{i=1}^b\prod\limits_{s=1}^{-dL_i^-(\ov{Ij})}\left[\la_i^-(I)\!+\!\frac{s}{d}u_j(I)\right]\!\in\!\Q_{\al}.
\end{split}\end{equation}

\begin{lmm}\label{ZCrec_lmm}
If $m\!\ge\!3$, $\ev_j\!:\!\ov\M_{0,m}(\X,\bd)\!\lra\!\X$
is the evaluation map at the $j$-th marked point,
$\eta_j\!\in\!H^*_{\T^N}\!(\X)$ and $\be_j\!\in\!\Z^{\ge 0}$ for
$j\!=\!2,\ldots,m$, then the power series
\begin{equation}\label{Zeb_e}\begin{split}
\cZp_{\eta,\be}(\h,Q)&\equiv\sum_{\bd\in\La}
Q^{\bd}\ev_{1*}\left[
\frac{\E\left(\cVp_E\right)}{\h\!-\!\psi_1}
\prod\limits_{j=2}^m\left(\psi_j^{\be_j}\ev_j^*\eta_j\right)\right]
\!\in\!H^*_{\T^N}(\X)\Lau{\h}[[\La]]\quad\textnormal{and}\\
\cZpp_{\eta,\be}(\h,Q)&\equiv\sum_{\bd\in\La}
Q^{\bd}\ev_{1*}\left[
\frac{\E\left(\cVpp_E\right)}{\h\!-\!\psi_1}
\prod\limits_{j=2}^m\left(\psi_j^{\be_j}\ev_j^*\eta_j\right)\right]
\!\in\!H^*_{\T^N}(\X)\Lau{\h}[[\La]]
\end{split}\end{equation}
are $\fCp$- and $\fCpp$-recursive, respectively, with $\fCp$ and $\fCpp$ given by \e_ref{fC_e}.
\end{lmm}
\begin{proof}
This is obtained by applying the Virtual Localization Theorem~\e_ref{virloc_e} on $\ov\M_{0,m}(\X,\bd)$, using Section~\ref{Sp_sec}, and extending the
proof of \cite[Lemma~1.1]{bcov0} from the case of a positive line bundle over $\P^{n-1}$  to that of a split vector bundle $E\!=\!E^+\!\oplus\!E^-$ as in \e_ref{splitE_e} over an arbitrary symplectic toric manifold $\X$.
By \e_ref{phiprop_e}, \e_ref{push-fwd_e}, \e_ref{virloc_e}, and the second equation in \e_ref{tan_e}, a decorated graph may contribute to $\cZp_{\eta,\be}(\h,Q)(I)$ and $\cZpp_{\eta,\be}(\h,Q)(I)$  
only if $\fp(\dec(1))\!=\!I$. 
There are thus two types of contributing graphs:
the $A_I$ and the $B_I$ graphs, where $I\!\in\!\V$.
In an $A_I$ graph the first marked point is attached to a vertex $\nv_0$ of valence $2$, while in
a $B_I$ graph the first marked point is attached to a vertex $\nv_0$ of valence at least $3$.
If $\Ga$ is a $B_I$ graph and $\cZ_{\Ga}$ the corresponding component of $\ov\M_{0,m}(\X,\bd)^{\T^N}$, then
$$\psi_1^n\!=\!0\qquad\forall\,n\!>\!\val(\nv_0)\!-\!3.$$
Thus, $\Ga$ contributes a polynomial in $\h^{-1}$ to the coefficient of $Q^{\bd}$ in $\cZp_{\eta,\be}(\h,Q)(I)$ and $\cZpp_{\eta,\be}(\h,Q)(I)$.

In an $A_I$ graph there is
a unique vertex $\nv$ joined to $\nv_0$ by an edge.
Let $A_{(I,j)}(d_0)$ be the set of all $A_I$ graphs such that
$\fp(\nv)\!=\!v(I,j)$ and the edge having $\nv_0$ as a vertex is labeled $d_0$.
Thus, 
$$A_I=\bigcup\limits_{d_0=1}^{\i}\bigcup\limits_{j\notin I}A_{(I,j)}(d_0).$$
We fix $\Ga\!\in\!A_{(I,j)}(d_0)$ and denote by $\Ga_0$ and $\Ga_c$ the two graphs obtained by breaking $\Ga$
at $\nv$, adding a second marked point to the vertex $\nv$ in $\Ga_0$ and a first marked point to $\nv$ in $\Ga_c$, and requiring that marked points $2,\ldots,m$ are in $\Ga_c$; see Figure~\ref{Zrec_fig}.\footnote{Figure~\ref{Zrec_fig} is \cite[Figure~2]{bcov0} adapted to the toric setting.}
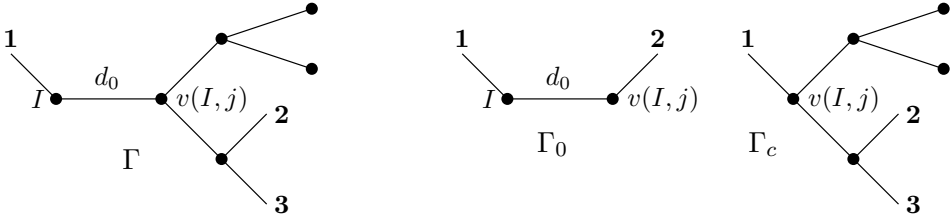
\begin{figure}
\begin{pspicture}(-1.2,-1.5)(10,2)
\psset{unit=.4cm}
\psline[linewidth=.04](2.5,0)(6,0)\rput(4.2,.7){\smsize{$d_0$}}
\pscircle*(2.5,0){.2}\rput(1.9,0){\smsize{$I$}}
\pscircle*(6,0){.2}\rput(7.7,0){\smsize{$v(I,j)$}}
\psline[linewidth=.04](2.5,0)(1,1.5)\rput(1,2){\smsize{$\bf 1$}}
\psline[linewidth=.04](6,0)(8,2)\pscircle*(8,2){.2}
\psline[linewidth=.04](8,2)(11,3)\pscircle*(11,3){.2}
\psline[linewidth=.04](8,2)(11,1)\pscircle*(11,1){.2}
\psline[linewidth=.04](6,0)(8,-2)\pscircle*(8,-2){.2}
\psline[linewidth=.04](8,-2)(9.5,-.5)\rput(10,-.5){\smsize{$\bf 2$}}
\psline[linewidth=.04](8,-2)(9.5,-3.5)\rput(10,-3.5){\smsize{$\bf 3$}}
\rput(5,-2){$\Ga$}
\psline[linewidth=.04](17.5,0)(21,0)\rput(19.2,.7){\smsize{$d_0$}}
\pscircle*(17.5,0){.2}\rput(16.9,0){\smsize{$I$}}
\pscircle*(21,0){.2}\rput(22.7,0){\smsize{$v(I,j)$}}
\psline[linewidth=.04](17.5,0)(16,1.5)\rput(16,2){\smsize{$\bf 1$}}
\psline[linewidth=.04](21,0)(22.5,1.5)\rput(22.5,2){\smsize{$\bf 2$}}
\rput(19,-1.5){$\Ga_0$}
\pscircle*(27,0){.2}\rput(28.7,0){\smsize{$v(I,j)$}}
\psline[linewidth=.04](27,0)(25.5,1.5)\rput(25.5,2){\smsize{$\bf 1$}}
\psline[linewidth=.04](27,0)(29,2)\pscircle*(29,2){.2}
\psline[linewidth=.04](29,2)(32,3)\pscircle*(32,3){.2}
\psline[linewidth=.04](29,2)(32,1)\pscircle*(32,1){.2}
\psline[linewidth=.04](27,0)(29,-2)\pscircle*(29,-2){.2}
\psline[linewidth=.04](29,-2)(30.5,-.5)\rput(31,-.5){\smsize{$\bf 2$}}
\psline[linewidth=.04](29,-2)(30.5,-3.5)\rput(31,-3.5){\smsize{$\bf 3$}}
\rput(26,-1.5){$\Ga_c$}
\end{pspicture}
\caption{A graph of type $A_{(I,j)}(d_0)$ and its two subgraphs}
\label{Zrec_fig}
\end{figure}
Thus, $\Ga_0$ consists only of the vertices $\nv_0$ and $\nv$ and the marked points~$1$ and $2$ attached to $\nv_0$ and $\nv$, respectively. With $\cZ_{\Ga}$ denoting the component in $\ov\M_{0,m}(\X,\bd)^{\T^N}$
corresponding to $\Ga$, $$\cZ_{\Ga}\cong\cZ_{\Ga_0}\times\cZ_{\Ga_c};$$
we denote by $\pi_0$ and $\pi_c$ the two projections. Thus,
\begin{equation}\label{gluing_e}
\cVp_E=\pi_0^*\cVp_E\oplus\pi_c^*\cVp_E\quad\textnormal{and}\quad
\cVpp_E=\pi_0^*\cVpp_E\oplus\pi_c^*\cVpp_E.
\end{equation}
These identities are obtained by considering the short exact sequence
of sheaves
$$0\lra f^*E^{\pm}\lra f_0^*E^{\pm}\oplus f_c^*E^{\pm}\lra E^{\pm}\big|_p\lra 0,$$
where $f\!:\!\Si\!\lra\!\X$ is a $\T^N$\!-fixed stable map whose corresponding graph
is $\Ga$, while $f_0$ and $f_c$ are its restrictions to the components of $\Si$ corresponding
to the edge leaving $\nv_0$ and the rest of $\Ga$.
Let $$\eta^{\be}\equiv\prod\limits_{j=2}^m\left(\psi_j^{\be_j}\ev_j^*\eta_j\right).$$
By \e_ref{gluing_e},
\begin{equation}\begin{split}\label{gluing_e2}
\frac{\E\left(\cVp_E\right)\!\!\eta^{\be}}{\h\!-\!\psi_1}\Big|_{\cZ_{\Ga}}\!\!\!\!=\!
\pi_0^*\left(\frac{\E\left(\cVp_E\right)}{\h\!-\!\psi_1}\right)\!
\pi_c^*\left(\E\left(\cVp_E\right)\!\!\eta^{\be}\right),\,
\frac{\E\left(\cVpp_E\right)\!\!\eta^{\be}}{\h\!-\!\psi_1}\Big|_{\cZ_{\Ga}}\!\!\!\!=\!
\pi_0^*\left(\frac{\E\left(\cVpp_E\right)}{\h\!-\!\psi_1}\right)\!
\pi_c^*\left(\E\left(\cVpp_E\right)\!\!\eta^{\be}\right).
\end{split}\end{equation}
By Proposition~\ref{vnormal_prp}, \e_ref{omF_e}, and \e_ref{ufixed_e}, 
\begin{equation}\label{normalglue_e}
\frac{\ev_1^*\phi_I\big|_{\cZ_{\Ga}}}{\E(N^{vir}_{\Ga})}=
\pi_0^*\left(\frac{\ev_1^*\phi_I}{\E(N^{vir}_{\Ga_0})}\right)
\pi_c^*\left(\frac{\ev_1^*\phi_{v(I,j)}}{\E(N^{vir}_{\Ga_c})}\right)
\frac{1}{-\frac{u_j(I)}{d_0}-\pi_c^*\psi_1}.
\end{equation}
By  \e_ref{cV0_e} and \e_ref{Lfixed_e}, on $\cZ_{\Ga_0}$
\begin{equation}\begin{split}\label{Zrec_e0}
\E\left(\cVp_E\right)\!&=
\!\prod_{i=1}^a\!\prod_{s=1}^{d_0L_i^+(\ov{Ij})}\!\left[\la_i^+(I)\!-\!\frac{s}{d_0}u_j(I)\right]
\prod_{i=1}^b\!\prod_{s=0}^{-d_0L_i^-(\ov{Ij})-1}\!\left[\la_i^-(I)\!+\!\frac{s}{d_0}u_j(I)\right],\\
\E\left(\cVpp_E\right)\!&=
\!\prod_{i=1}^a\!\prod_{s=0}^{d_0L_i^+(\ov{Ij})-1}\!\left[\la_i^+(I)\!-\!\frac{s}{d_0}u_j(I)\right]
\prod_{i=1}^b\!\prod_{s=1}^{-d_0L_i^-(\ov{Ij})}\!\left[\la_i^-(I)\!+\!\frac{s}{d_0}u_j(I)\right].
\end{split}\end{equation}
By Proposition~\ref{vnormal_prp},
\begin{equation}\label{normal0_e}
\E(N^{vir}_{\Ga_0})=\frac{(-1)^{d_0}\left(d_0!\right)^2}{d_0^{2d_0}}\left[u_j(I)\right]^{2d_0}
\prod_{\begin{subarray}{c}r\in[N]- (I\cup\{j\})\end{subarray}}\!\!
\frac{\prod\limits_{s=0}^{d_0D_r(\ov{Ij})}\!\left[u_r(I)\!-\!\frac{s}{d_0}u_j(I)\right]}
{\prod\limits_{s=d_0D_r(\ov{Ij})+1}^{-1}\!
\left[u_r(I)\!-\!\frac{s}{d_0}u_j(I)\right]}.
\end{equation}
By \e_ref{Zrec_e0}, \e_ref{normal0_e}, \e_ref{omF_e}, and \e_ref{fC_e},
\begin{equation}\label{Ga0_e}
\int_{\cZ_{\Ga_0}}\frac{\E\left(\cVp_E\right)\ev_1^*\phi_I}
{(\h\!-\!\psi_1)\E\left(N^{vir}_{\Ga_0}\right)}=
\frac{\fCp_{I,j}(d_0)}{\h\!+\!\frac{u_j(I)}{d_0}}\quad\textnormal{and}\quad
\int_{\cZ_{\Ga_0}}\frac{\E\left(\cVpp_E\right)\ev_1^*\phi_I}
{(\h\!-\!\psi_1)\E\left(N^{vir}_{\Ga_0}\right)}=
\frac{\fCpp_{I,j}(d_0)}{\h\!+\!\frac{u_j(I)}{d_0}}.
\end{equation}
By \e_ref{gluing_e2}, \e_ref{normalglue_e}, and \e_ref{Ga0_e},
\begin{equation}\begin{split}\label{finalrec_e}
\int_{\cZ_{\Ga}}
\frac{\E(\cVp_E)\ev_1^*\phi_I\eta^{\be}}{\h-\psi_1}\Big|_{\cZ_{\Ga}}
\frac{1}{\E(N^{vir}_{\Ga})}&=
\frac{\fCp_{I,j}(d_0)}{\h+\frac{u_j(I)}{d_0}}
\int_{\cZ_{\Ga_c}}\frac{\E(\cVp_E)\ev_1^*\phi_{v(I,j)}\eta^{\be}}{\h-\psi_1}
\frac{1}{\E(N^{vir}_{\Ga_c})}\Big|_{\h=-\frac{u_j(I)}{d_0}},\\
\int_{\cZ_{\Ga}}
\frac{\E(\cVpp_E)\ev_1^*\phi_I\eta^{\be}}{\h-\psi_1}\Big|_{\cZ_{\Ga}}
\frac{1}{\E(N^{vir}_{\Ga})}&=
\frac{\fCpp_{I,j}(d_0)}{\h+\frac{u_j(I)}{d_0}}
\int_{\cZ_{\Ga_c}}\frac{\E(\cVpp_E)\ev_1^*\phi_{v(I,j)}\eta^{\be}}{\h-\psi_1}
\frac{1}{\E(N^{vir}_{\Ga_c})}\Big|_{\h=-\frac{u_j(I)}{d_0}}.
\end{split}\end{equation}
By the first equation in \e_ref{finalrec_e} and the Virtual Localization Theorem~\e_ref{virloc_e},
the contribution of the $A_I$ graphs to the coefficient of
$Q^{\bd}$ in $\cZp_{\eta,\be}\big|_I$ is 
$$\sum_{d_0\ge 1}\sum\limits_{\begin{subarray}{c}j\in [N]-I\\d_0\cdot\deg\ov{Ij}\preceq\bd\end{subarray}}
\frac{\fCp_{I,j}(d_0)}{\h+\frac{u_j(I)}{d_0}}
\LR{\cZp_{\eta,\be}(x(v(I,j)),\h,Q)}_{Q;\bd-d_0\cdot\deg\ov{Ij}}\big|_{\h=-\frac{u_j(I)}{d_0}}$$
whenever $\bd\!\equiv\!\bd^*$ satisfies the two properties in Definition~\ref{rec_dfn} (which make evaluation at $\h\!=\!-\frac{u_j(I)}{d_0}$ meaningful).
An analogous statement holds when summing in the second equation in \e_ref{finalrec_e}.
\end{proof}

\subsection{MPC for the GW power series}
\label{ZMPC_sec}
Let $\cZp_1$ and $\cZpp_1$ be as in \e_ref{eZeta1ptdfn_e} and
$\cZp_{\eta,\be}$ and $\cZpp_{\eta,\be}$ be as in \e_ref{Zeb_e}.
\begin{lmm}\label{ZMPC_lmm}
For all $m\!\ge\!3,\eta_j\!\in\!H^*_{\T^N}(\X),
\be_j\!\in\!\Z^{\ge 0}$, the pairs
$(\cZpp_1(\h,Q),\h^{m-2}\cZp_{\eta,\be}(\h,Q))$ and $(\cZp_1(\h,Q),\h^{m-2}\cZpp_{\eta,\be}(\h,Q))$
satisfy the MPC.
\end{lmm}
Lemma~\ref{ZMPC_lmm} extends \cite[Lemma~1.2]{bcov0} from the case of a positive line bundle over $\P^{n-1}$  to that of a split vector bundle $E\!=\!E^+\!\oplus\!E^-$ as in \e_ref{splitE_e} over an arbitrary symplectic toric manifold $\X$.
While \cite[Lemma~1.2]{bcov0} follows from
\cite[Lemma~3.1]{bcov0}, Lemma~\ref{ZMPC_lmm}
follows from Lemma~\ref{push_lmm} below, which extends
\cite[Lemma~3.1]{bcov0} to the general toric case. 
The proof of Lemma~\ref{push_lmm} uses the Virtual Localization Theorem~\e_ref{virloc_e} 
instead of the classical one used in the $\X\!=\!\P^{n-1}$ case and  
Lemma~\ref{Om_lmm}, which is a general toric version of
the first displayed formula in \cite{bcov0} after \cite[(3.32)]{bcov0}.

As in \cite{Gi_equiv} and \cite{bcov0}, we consider the action of $\T^1$ on $V\!\equiv\!\C^2$
given by $\xi\cdot(z_0,z_1)\!\equiv\!(z_0,\xi^{-1}z_1)$
and the induced action on $\P V$.
Let $\h$ be the weight of the standard action
of $\T^1$ on $\C$.
For any $\bd\!\in\!\La$, let
$$\fX_{\bd}(\X)\!\equiv\!\left\{f\!\in\!\ov\M_{0,m}(\P V\!\times\!\X,(1,\bd)):
\ev_1(f)\!\in\![1,0]\!\times\!\X,\ev_2(f)\!\in\![0,1]\!\times\!\X\right\}.$$

By Proposition~\ref{flocus_prp},
the components of the fixed locus $\fX_{\bd}(\X)^{\T^1\times\T^N}$
of the $\T^1\!\times\!\T^N\!$-action on $\fX_{\bd}(\X)$ are indexed by decorated graphs
$\Ga$ of the following form.
Such a graph $\Ga$ has a unique edge of positive $\P V$-degree;
this special edge corresponds to a degree-one map
$f\!:\!\P^1\!\lra\!\P V\!\times\![I]$ for some $I\!\in\!\V$.
Edges to the left (respectively right) of this edge are mapped into $[1,0]\!\times\!\X$
(respectively $[0,1]\!\times\!\X$); see Figure~\ref{Xlocus_fig}, where we dropped the $\P V$-label of the 
vertices.\footnote{Figure~\ref{Xlocus_fig} is \cite[Figure~3]{bcov0} adapted to the toric setting.} Thus, the first marked point is attached to
some vertex to the left of the special edge,
while the second marked point is attached to some vertex to the right of the special edge.
\begin{figure}
\begin{pspicture}(-.5,-1.2)(10,1.8)
\psset{unit=.4cm}
\psline[linewidth=.1](17,0)(22,0)
\pscircle*(17,0){.2}\rput(17,.7){\smsize{$I$}}
\pscircle*(22,0){.2}\rput(22,.7){\smsize{$I$}}
\psline[linewidth=.04](17,0)(14,1)\pscircle*(14,1){.2}
\rput(15.5,1){\smsize{$2$}}\rput(14,1.7){\smsize{$J$}}
\psline[linewidth=.04](14,1)(11.5,2.5)\pscircle*(11.5,2.5){.2}
\rput(12.8,2.3){\smsize{$2$}}\rput(10.9,2.5){\smsize{$K$}}
\psline[linewidth=.04](14,1)(11.5,-.5)\pscircle*(11.5,-.5){.2}
\rput(12.3,.6){\smsize{$1$}}\rput(10.9,-.5){\smsize{$L$}}
\psline[linewidth=.04](17,0)(14.5,-1.5)\pscircle*(14.5,-1.5){.2}
\rput(16.2,-1.1){\smsize{$3$}}\rput(14.5,-.8){\smsize{$L$}}
\psline[linewidth=.04](14.5,-1.5)(12,-2)\rput(11.5,-2){\smsize{$\bf 1$}}
\psline[linewidth=.04](22,0)(25,1)\pscircle*(25,1){.2}
\rput(23.5,1){\smsize{$1$}}\rput(25,1.7){\smsize{$K$}}
\psline[linewidth=.04](25,1)(27,2)\rput(27.5,2){\smsize{$\bf 2$}}
\psline[linewidth=.04](25,1)(27.5,-.5)\pscircle*(27.5,-.5){.2}
\rput(26.7,.6){\smsize{$3$}}\rput(28.1,-.5){\smsize{$J$}}
\psline[linewidth=.04](22,0)(24.5,-1.5)\pscircle*(24.5,-1.5){.2}
\rput(22.9,-1.1){\smsize{$5$}}\rput(24.5,-.8){\smsize{$J$}}
\psline[linewidth=.04](24.5,-1.5)(27,-2)\rput(27.5,-2){\smsize{$\bf 3$}}
\end{pspicture}
\caption{A graph representing a fixed locus in $\fX_{\bd}(\X)$; $I,J,K,L\in\V$, $I\!\neq\!J,K,L$.}
\label{Xlocus_fig}
\end{figure}

Let $$\bd_L\!\equiv\!\bd_L(\Ga),\quad\bd_R\!\equiv\!\bd_R(\Ga)\!\in\!\La$$ denote the $\X$-degrees of the left-
and right-hand side (with respect to the special edge)
sub-graphs, respectively; thus, $\bd\!=\!\bd_L\!+\!\bd_R$.
Let $\cZ_{\Ga}$ be the component of $\fX_{\bd}(\X)^{\T^1\times\T^N}$ corresponding to $\Ga$.

\begin{lmm}\label{Om_lmm}
For every $i\!\in\![k]$ and $\bd\!\in\!\La$, there exists 
$$\Om_i\!\in\!H^2_{\T^1\times\T^N}(\fX_{\bd}(\X))\quad
\textnormal{such that}\quad\Om_i\big|_{\cZ_{\Ga}}\!=\!x_i(I)\!+\!\left(\bd_L(\Ga)\right)_i\h$$
for all graphs $\Ga$
corresponding to components of $\fX_{\bd}(\X)^{\T^1\times\T^N}\!\!\!\!$, with
$\bd_L(\Ga)$ and $I$ depending on $\Ga$ as above.
\end{lmm}
\begin{proof}
We follow the proof in \cite[Section~11]{Gi_equiv} and
\cite[Section~2]{Gi_mirr}. 

Given $s\!\in\!\Z^{\ge 0}$ and $n\!\ge\!1$, let 
$$\Poly_s^n\equiv\P\left(\left\{P\!\in\C[z_0,z_1]:P\textnormal{ homogeneous of degree }s\right\}^{\oplus n}\right).$$
We next define a morphism
$$\theta_0:\ov\M_{0,0}(\P V\!\times\!\P^{n-1},(1,s))\lra \Poly_s^n.$$
If $[\Si,f]$ is an element of $\ov\M_{0,0}(\P V\!\times\!\P^{n-1},(1,s))$,
$\Si\!=\!\Si_0\!\cup\!\Si_1\!\cup\!\ldots\!\cup\!\Si_r$, where
$\Si_0$ is a $\P^1$, $f\Big|_{\Si_0}$ has degree $(1,s_0)$,
$\Si_i$ is connected for all $i\!\in\![r]$, and $f\Big|_{\Si_i}$ has degree $(0,s_i)$
for all $i\!\in\![r]$.
Thus, $$f(\Si_i)\subseteq\left\{[A_i,B_i]\right\}\!\times\!\P^{n-1}\quad\textnormal{for some}\quad
[A_i,B_i]\!\in\!\P V\qquad\forall\,i\!\in\![r].$$
Let ${\theta}_0[\Si,f]\!\equiv\![P_1g,\ldots,P_ng]$, where
\begin{align*}
f\big|_{\Si_0}\equiv(f_1,f_2),\qquad f_2\circ f_1^{-1}\equiv[P_1,\ldots,P_n]\!\in\!\Poly_{s_0}^n,
\qquad g\equiv\prod\limits_{i=1}^r(A_iz_1\!-\!B_iz_0)^{s_i}.
\end{align*}
Let $\theta\!\equiv\!\theta_0\circ\fgt$, where
$$\fgt:\ov\M_{0,m}(\P V\!\times\!\P^{n-1},(1,s))\lra\ov\M_{0,0}(\P V\!\times\!\P^{n-1},(1,s))$$
is the forgetful morphism.
By \cite[Section~11, Main Lemma]{Gi_equiv}, 
$\theta\big|_{\fX_s(\P^{n-1})}$ is continuous.

The torus $\T^1\!\times\!\T^n$ acts on $\Poly_s^n$ by
$$(\xi,t_1,\ldots,t_n)\cdot\left(P_1[z_0,z_1],\ldots,P_n[z_0,z_1]\right)\equiv
\left(t_1P_1[z_0,\xi z_1],\ldots,t_nP_n[z_0,\xi z_1]\right).$$
This action naturally lifts to the hyperplane line bundle over $\Poly_s^n$.
The map $\theta_0$ is $\T^1\!\times\!\T^n$-equivariant and hence so is $\theta$.

Let 
$\cL\!\lra\!\X$ be any very ample line bundle.
For any $\bd\!\in\!\La$, let $\cL(\bd)\!\equiv\!\blr{c_1(\cL),\bd}$.
Consider the canonical lift of the $\T^N$\!-action on $\X$ to $\cL$
given by Proposition~\ref{Pic_prp} together with \e_ref{trivact_e}.
Thus, there exists $n$, an injective group homomorphism
$\iota_{\T}\!:\!\T^N\!\lra\!\T^n$, and an $\iota_{\T}$-equivariant embedding
$\iota\!:\!\X\!\lra\!\P^{n-1}$ such that $\iota^*\cO_{\P^{n-1}}(1)\!=\!\cL$.
We consider the $\T^N$\!-action on $\P^{n-1}$ induced by $\iota_{\T}$.
The embedding $\iota$ induces a $\T^1\!\times\!\T^N\!$-equivariant embedding
$$\fX_{\bd}(\X)\xrightarrow{F}\fX_{\cL(\bd)}(\P^{n-1}).$$
The composition
$$\fX_{\bd}(\X)\xrightarrow{F}\fX_{\cL(\bd)}(\P^{n-1})\xrightarrow{\theta}\Poly_{\cL(\bd)}^n$$
maps $\cZ_{\Ga}$ onto $[z_0^{\cL(\bd_R)}z_1^{\cL(\bd_L)}a_1,\ldots,z_0^{\cL(\bd_R)}z_1^{\cL(\bd_L)}a_n]$,
where $[a_1,\ldots,a_n]\!\equiv\!\iota([I])$.

Let $\Om\!\in\!H^2_{\T^1\times\T^N}\!(\Poly_{\cL(\bd)}^n)$ be the equivariant Euler class of the hyperplane line bundle and
$$\Om(\cL)\equiv F^*\theta^*\Om\in H^2_{\T^1\times\T^N}(\fX_{\bd}(\X)).$$
It follows that
\begin{equation}\label{Om_e}
\Om(\cL)\big|_{\cZ_{\Ga}}=
\Om\big|_{[z_0^{\cL(\bd_R)}z_1^{\cL(\bd_L)}a_1,\ldots,z_0^{\cL(\bd_R)}z_1^{\cL(\bd_L)}a_n]}
=\E(\cL)(I)\!+\!\blr{c_1(\cL),\bd_L}\h,
\end{equation}
where $\E(\cL)$ is the $\T^N$\!-equivariant Euler class of $\cL$.

By Proposition~\ref{dimK_prp}, there exist very ample line bundles
$\cL_i$ for all $i\!\in\![k]$ such that $\{c_1(\cL_i):i\!\in\![k]\}$ is a basis for $H^2(\X)$;
so, using the $\T^N$\!-action on each $\cL_i$ defined by \e_ref{trivact_e}, we find that
\begin{equation*}\begin{split}
\Span_{\Q}\left\{\E(\cL_i):i\!\in\![k]\right\}
=\Span_{\Q}\left\{x_i:i\!\in\![k]\right\}.
\end{split}\end{equation*}
Via Proposition~\ref{equivcoh_prp}\ref{equivcohr_prp}, this shows that
$\{\E(\cL_i),\al_j:i\!\in\![k],j\!\in\![N]\}$ is a basis for $H^2_{\T^N}(\X)$. 
As in \cite{Gi_mirr}, we define a $\Q$-linear map from $H^2_{\T^N}(\X)$ to $H^2_{\T^1\times\T^N}(\fX_{\bd}(\X))$
by sending
$\E(\cL_i)$ to $\Om(\cL_i)$ for all $i\!\in\![k]$ and $\al_j$ to $\al_j$ for all $j\!\in\![N]$.
Let $\Om_i\!\in\!H^2_{\T^1\!\times\!\T^N}(\fX_{\bd}(\X))$ be the image of $x_i$ 
under this map.
The claim now follows from \e_ref{Om_e}.
\end{proof}
\begin{lmm}\label{push_lmm}
Let $\eta^{\be}\!\equiv\!\prod\limits_{j=2}^m\left(\psi_j^{\be_j}\ev_j^*\eta_j\right)$ in $H_{\T^N}^*(\ov\M_{0,m}(\X,\bd))$ and
let $$\pi:\ov\M_{0,m}(\P V\!\times\!\X,(1,\bd))\lra\ov\M_{0,m}(\X,\bd)$$ denote
the natural projection.
With $\Phi$ as in Definition~\ref{MPC_dfn} and $\Om_i$ as in Lemma~\ref{Om_lmm},
\begin{equation}\begin{split}\label{push_e}
(-\h)^{m-2}\Phi_{\cZpp_1,\cZp_{\eta,\be}}(\h,z,Q)&=
\sum_{\bd\in\La}\!Q^{\bd}\!\int_{[\fX_{\bd}(\X)]^{vir}}\!\!\!\!\ne^{\sum\limits_{i=1}^k\Om_i z_i}\!\!
\pi^*\left[\E\left(\cVp_E\right)\eta^{\be}\right]\prod_{j=3}^m\ev_j^*\E(\cO_{\P V}(1)),\\
(-\h)^{m-2}\Phi_{\cZp_1,\cZpp_{\eta,\be}}(\h,z,Q)&=
\sum_{\bd\in\La}\!Q^{\bd}\!\int_{[\fX_{\bd}(\X)]^{vir}}\!\!\!\!\ne^{\sum\limits_{i=1}^k\Om_i z_i}\!\!
\pi^*\left[\E\left(\cVpp_E\right)\eta^{\be}\right]\prod_{j=3}^m\ev_j^*\E(\cO_{\P V}(1)).
\end{split}\end{equation}
\end{lmm}
\begin{proof}
We apply the Virtual Localization Theorem~\e_ref{virloc_e} to the right-hand side of each of the two equations in \e_ref{push_e},
using Section~\ref{Sp_sec} and extending the proof of \cite[Lemma~3.1]{bcov0} from the case of a positive line bundle over $\P^{n-1}$  to that of a split vector bundle $E\!=\!E^+\!\oplus\!E^-$ as in \e_ref{splitE_e} over an arbitrary symplectic toric manifold $\X$. The possible contributing fixed loci graphs are described above.
Given such a fixed locus graph $\Ga$, we denote by $N^{vir}_{\Ga}$ the virtual normal bundle to the corresponding
component of the fixed locus inside the moduli space.
We denote by $\cA_I$ the set of all $\T^1\!\times\!\T^N\!$-fixed loci graphs whose unique edge
of positive $\P V$-degree corresponds to a map $\P^1\!\lra\!\P V\!\times\![I]$, where $I\!\in\!\V$.
A graph $\Ga\!\in\!\cA_I$ breaks into $3$ graphs - $\Ga_L$, $\Ga_R$, and $\Ga_0$ - as follows;
see also Figure~\ref{Xsplit_fig}.\footnote{Figure~\ref{Xsplit_fig} is \cite[Figure~4]{bcov0} adapted to the toric setting.}
The graph $\Ga_L$ is obtained by considering all vertices and edges of $\Ga$ to the left of the special edge
(of positive $\P V$-degree) and adding a marked point labeled $2$ at the vertex belonging to the special edge.
Given that all vertices in this ``left-hand side graph'' are labeled $([1,0],I)$ for some $I\!\in\!\V$, it defines a component of
$\ov\M_{0,2}(\X,\bd_L)^{\T^N}\!$. The graph $\Ga_R$ is obtained by considering all vertices of $\Ga$
to the right of the special edge and adding a marked point labeled $1$ at the vertex belonging to the special edge.
Given that all vertices in this ``right-hand side graph'' are labeled $([0,1],I)$ for some $I\!\in\!\V$,
it defines a component of
$\ov\M_{0,m}(\X,\bd_R)^{\T^N}\!$. Finally, $\Ga_0$ is the special edge with $2$ marked points added. They are labeled $1$ in the left-hand side and $2$ in the right-hand side.
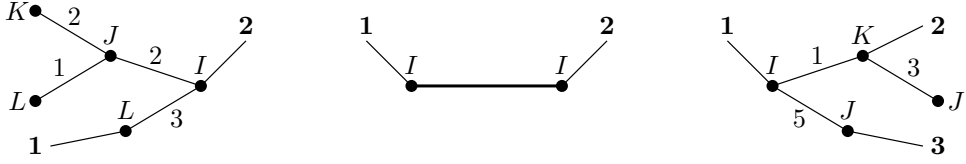
\begin{figure}
\begin{pspicture}(-.5,-1.2)(10,1.8)
\psset{unit=.4cm}
\pscircle*(10,0){.2}\rput(10,.7){\smsize{$I$}}
\psline[linewidth=.04](10,0)(7,1)\pscircle*(7,1){.2}
\rput(8.5,1){\smsize{$2$}}\rput(7,1.7){\smsize{$J$}}
\psline[linewidth=.04](7,1)(4.5,2.5)\pscircle*(4.5,2.5){.2}
\rput(5.8,2.3){\smsize{$2$}}\rput(3.9,2.5){\smsize{$K$}}
\psline[linewidth=.04](7,1)(4.5,-.5)\pscircle*(4.5,-.5){.2}
\rput(5.3,.6){\smsize{$1$}}\rput(3.9,-.5){\smsize{$L$}}
\psline[linewidth=.04](10,0)(7.5,-1.5)\pscircle*(7.5,-1.5){.2}
\rput(9.2,-1.1){\smsize{$3$}}\rput(7.5,-.8){\smsize{$L$}}
\psline[linewidth=.04](7.5,-1.5)(5,-2)\rput(4.5,-2){\smsize{$\bf 1$}}
\psline[linewidth=.04](10,0)(11.5,1.5)\rput(11.5,2){\smsize{$\bf 2$}}
\psline[linewidth=.04](17,0)(15.5,1.5)\rput(15.5,2){\smsize{$\bf 1$}}
\psline[linewidth=.04](22,0)(23.5,1.5)\rput(23.5,2){\smsize{$\bf 2$}}
\psline[linewidth=.1](17,0)(22,0)
\pscircle*(17,0){.2}\rput(17,.7){\smsize{$I$}}
\pscircle*(22,0){.2}\rput(22,.7){\smsize{$I$}}
\psline[linewidth=.04](29,0)(27.5,1.5)\rput(27.5,2){\smsize{$\bf 1$}}
\pscircle*(29,0){.2}\rput(29,.7){\smsize{$I$}}
\psline[linewidth=.04](29,0)(32,1)\pscircle*(32,1){.2}
\rput(30.5,1){\smsize{$1$}}\rput(32,1.7){\smsize{$K$}}
\psline[linewidth=.04](32,1)(34,2)\rput(34.5,2){\smsize{$\bf 2$}}
\psline[linewidth=.04](32,1)(34.5,-.5)\pscircle*(34.5,-.5){.2}
\rput(33.7,.6){\smsize{$3$}}\rput(35.1,-.5){\smsize{$J$}}
\psline[linewidth=.04](29,0)(31.5,-1.5)\pscircle*(31.5,-1.5){.2}
\rput(29.9,-1.1){\smsize{$5$}}\rput(31.5,-.8){\smsize{$J$}}
\psline[linewidth=.04](31.5,-1.5)(34,-2)\rput(34.5,-2){\smsize{$\bf 3$}}
\end{pspicture}
\caption{The three sub-graphs of the graph in Figure~\ref{Xlocus_fig}}
\label{Xsplit_fig}
\end{figure}
Thus,
$$\cZ_{\Ga}\cong \cZ_{\Ga_L}\times\cZ_{\Ga_0}\times\cZ_{\Ga_R};$$
we denote by $\pi_L$, $\pi_0$, and $\pi_R$
the corresponding projections.

It follows that
\begin{equation}\begin{split}\label{LR_e1}
\pi^*\cVp_E&=\pi_L^*\cVp_E
\oplus\pi_R^*\cVp_E,\quad\pi^*\cVpp_E=\pi_L^*\cVpp_E
\oplus\pi_R^*\cVpp_E,\\
\frac{N^{vir}_{\Ga}}{T_{[I]}\X}&=
\pi_L^*\left(\frac{N^{vir}_{\Ga_L}}{T_{[I]}\X}\right)
\oplus\pi_R^*\left(\frac{N^{vir}_{\Ga_R}}{T_{{[I]}}\X}\right)\oplus
\pi_L^*L_2\otimes\pi_0^*L_1\oplus\pi_0^*L_2\otimes\pi_R^*L_1,
\end{split}\end{equation}
where $L_2\!\lra\!\cZ_{\Ga_L}$,$L_1,L_2\!\lra\!\cZ_{\Ga_0}$, and $L_1\!\lra\!\cZ_{\Ga_R}$
are the tautological tangent line bundles.
The first two equations in \e_ref{LR_e1}
follow similarly to \e_ref{gluing_e}.

By \e_ref{LR_e1} and \e_ref{tan_e},
\begin{equation}\begin{split}\label{LR_e2}
\pi^*\left[\E\left(\cVp_E\right)\eta^{\be}\right]
\prod\limits_{j=3}^m\ev_j^*\left[\E\left(\cO_{\P V}(1)\right)\right]\big|_{\cZ_{\Ga}}&=
\pi_L^*\left[\E\left(\cVp_E\right)\right]\pi_R^*\left[\E\left(\cVp_E\right)\eta^{\be}\left(-\h\right)^{m-2}\right],\\
\frac{\E(T_{[I]}\X)}{\E\left(N^{vir}_{\Ga}\right)}=
\pi_L^*\left[\frac{\ev_2^*\phi_I}{\E\left(N^{vir}_{\Ga_L}\right)}\right]&
\pi_R^*\left[\frac{\ev_1^*\phi_I}{\E\left(N^{vir}_{\Ga_R}\right)}\right]
\frac{1}{\left(\h\!-\!\pi_L^*\psi_2\right)\left((-\h)\!-\!\pi_R^*\psi_1\right)},
\end{split}\end{equation}
and the first equation in \e_ref{LR_e2} with $\cVp_E$ replaced by $\cVpp_E$ also holds. 
By \e_ref{LR_e2} and Lemma~\ref{Om_lmm},
\begin{equation}\begin{split}\label{LR_e3}
\int_{\cZ_{\Ga}}\frac{\ne^{\sum\limits_{i=1}^k\Om_iz_i}
\pi^*\left[\E\left(\cVp_E\right)\eta^{\be}\right]\prod\limits_{j=3}^m\ev^*_j
\E\left(\cO_{\P V}(1)\right)\big|_{\cZ_{\Ga}}}{\E\left(N^{vir}_{\Ga}\right)}
=(-\h)^{m-2}\frac{\ne^{\sum\limits_{i=1}^kx_i(I)z_i}}{\E(T_{[I]}\X)}
\qquad\qquad\qquad
\\\times
\left\{\ne^{\sum\limits_{i=1}^k\left(\bd_L\right)_iz_i\h}
\int_{\cZ_{\Ga_L}}\frac{\E\left(\cVp_E\right)\ev_2^*\phi_I}{\h\!-\!\psi_2}
\big|_{\cZ_{\Ga_L}}\frac{1}{\E\left(N^{vir}_{\Ga_L}\right)}\right\}\left\{\int_{\cZ_{\Ga_R}}\frac{\E\left(\cVp_E\right)\eta^{\be}\ev_1^*\phi_I}{(-\h)\!-\!\psi_1}
\frac{1}{\E\left(N^{vir}_{\Ga_R}\right)}\right\};
\end{split}\end{equation}
\e_ref{LR_e3} with $\cVp_E$ replaced by $\cVpp_E$ also holds.
In the $\bd_L\!=\!0$ case, the first curly bracket on the right-hand side of \e_ref{LR_e3}
is defined to be $1$. By the Virtual Localization Theorem~\e_ref{virloc_e} and \e_ref{push-fwd_e},
\begin{equation}\begin{split}\label{LR_e4}
\sum_{\Ga_{L}}Q^{\bd_L}\left\{\ne^{\sum\limits_{i=1}^k\left(\bd_L\right)_iz_i\h}
\int_{\cZ_{\Ga_L}}\frac{\E\left(\cVp_E\right)\ev_2^*\phi_I}{\h\!-\!\psi_2}
\big|_{\cZ_{\Ga_L}}\frac{1}{\E\left(N^{vir}_{\Ga_L}\right)}\right\}
=&\cZpp_1(\h,Q\ne^{\h z})\big|_I,\\
\sum_{\Ga_R}Q^{\bd_R}\left\{\int_{\cZ_{\Ga_R}}\frac{\E\left(\cVp_E\right)\eta^{\be}\ev_1^*\phi_I}{(-\h)\!-\!\psi_1}\frac{1}{\E\left(N^{vir}_{\Ga_R}\right)}\right\}=&\cZp_{\eta,\be}(-\h,Q)\big|_I,
\end{split}\end{equation}
where the first sum is taken after all graphs $\Ga_L$ corresponding to components of $\ov\M_{0,2}(\X,\bd_L)^{\T^N}$
and with second marked point mapping to $[I]$,
while the second is taken after all graphs corresponding to components of $\ov\M_{0,m}(\X,\bd_R)^{\T^N}$
such that the first marked point is mapped to $[I]$.
The equation obtained from \e_ref{LR_e4} by replacing $\cVp_E$ by $\cVpp_E$, $\cZpp_1$ by $\cZp_1$,
and $\cZp_{\eta,\be}$ by $\cZpp_{\eta,\be}$ also holds.
The claims follow from \e_ref{LR_e3} and \e_ref{LR_e4} (and their $\cVpp_E$ analogues), by summing the left-hand side of \e_ref{LR_e3}
over all graphs $\Ga\!\in\!\cA_I$ and all $I\!\in\!\V$.
\end{proof}

\subsection{Recursivity and MPC for the explicit power series}
\label{Y_sec}

As in \cite{Gi_mirr}, for each $I\!\in\!\V$ we define 
\begin{equation}\label{DeI*_e}
\De_I^*\equiv\left\{\bd\!\in\!\La:D_j(\bd)\ge 0\quad\forall\,j\!\in\!I\right\}.
\end{equation}
By \e_ref{uYgeom_e}, \e_ref{cY_e}, and \e_ref{fprestr_e},
\begin{equation}\label{Yrestr_e}
\lrbr{\cYp(x(I),\h,q)}_{q;\bd}\neq 0\Lra \bd\!\in\!\De_I^*\quad\textnormal{and}\quad
\lrbr{\cYpp(x(I),\h,q)}_{q;\bd}\neq 0\Lra \bd\!\in\!\De_I^*.
\end{equation}
\begin{lmm}\label{YCrec_lmm}
The power series $\cYp(x,\h,q)$ of \e_ref{cY_e} is $\fCp$-recursive with
$\fCp$ given by \e_ref{fC_e}. The power series $\cYpp(x,\h,q)$ of \e_ref{cY_e} is $\fCpp$-recursive with $\fCpp$
given by \e_ref{fC_e}.
\end{lmm}
\begin{proof}
The recursivity of $\cYp$ in the $E\!=\!E^+$ case is \cite[Proposition~6.3]{Gi_mirr}.
The proof of the recursivity of $\cYp$ in the general case is similar and so is
the proof of the recursivity of $\cYpp$.
We prove below the recursivity of $\cYpp$
extending the proof of (a) in \cite[Section~2.3]{bcov0}
and the proof of \cite[Proposition~6.3]{Gi_mirr}.
Let $I\!\in\!\V$, $j\!\in\![N]\!-\!I$, $J\!\equiv\!v(I,j)$, $\{\wh{j}\}\!\equiv\!I\!-\!J$.
By \e_ref{Yrestr_e}, \e_ref{cY_e}, Remark~\ref{non0u_rmk}, and \e_ref{ufixed_e}, 
\begin{equation}\begin{split}\label{Yrec_e1}
\cYpp\left(x(J),-\frac{u_j(I)}{d},q\right)\!=\!\!\!\!\!\!
\sum_{\begin{subarray}{c}\bd'\in\De^*_J\\
D_{j'}(\bd')\ge -d\end{subarray}}\!\!\!\!
q^{\bd'}\frac{\prod\limits_{\begin{subarray}{c}r\in[N]\\ D_r(\bd')<0\end{subarray}}\prod\limits_{s=D_r(\bd')+1}^0\left[u_r(J)\!-\!\frac{s}{d}u_j(I)\right]}
{\prod\limits_{\begin{subarray}{c}r\in[N]\\D_r(\bd')\ge0\end{subarray}}
\prod\limits_{s=1}^{D_r(\bd')}\left[u_r(J)\!-\!\frac{s}{d}u_j(I)\right]}\\\times
\prod\limits_{i=1}^a\prod\limits_{s=0}^{L_i^+(\bd')-1}\left[\la_i^+(J)\!-\!\frac{s}{d}u_j(I)\right]
\prod\limits_{i=1}^b\prod\limits_{s=1}^{-L_i^-(\bd')}\left[\la_i^-(J)\!+\!\frac{s}{d}u_j(I)\right].
\end{split}\end{equation}
By \e_ref{Yrec_e1}, \e_ref{ufixed2_e}, and \e_ref{Lfixed_e},
\begin{equation}\begin{split}\label{Yrec_e2}
\cYpp\left(x(J),-\frac{u_j(I)}{d},q\right)\!=\!\!\!\!\!\!
\sum_{\begin{subarray}{c}\bd'\in\De^*_J\\
D_{\wh{j}}(\bd')\ge -d\end{subarray}}\!\!\!\!\!\!
q^{\bd'}\frac{\prod\limits_{\begin{subarray}{c}r\in[N]\\D_r(\bd')<0\end{subarray}}
\prod\limits_{s=D_r(\bd')+1+dD_r(\ov{Ij})}^{dD_r(\ov{Ij})}\left[u_r(I)\!-\!\frac{s}{d}u_j(I)\right]}
{\prod\limits_{\begin{subarray}{c}r\in[N]\\D_r(\bd')\ge 0\end{subarray}}
\prod\limits_{s=1+dD_r(\ov{Ij})}^{D_r(\bd')+dD_r(\ov{Ij})}\left[u_r(I)\!-\!\frac{s}{d}u_j(I)\right]}\qquad\qquad\\\times
\prod_{i=1}^a\prod_{s=dL_i^+(\ov{Ij})}^{L_i^+(\bd')-1+dL_i^+(\ov{Ij})}
\left[\la_i^+(I)\!-\!\frac{s}{d}u_j(I)\right]
\prod\limits_{i=1}^b\prod\limits_{s=1-dL_i^-(\ov{Ij})}^{-L_i^-(\bd')-dL_i^-(\ov{Ij})}
\left[\la_i^-(I)\!+\!\frac{s}{d}u_j(I)\right].
\end{split}\end{equation}
By \e_ref{Dedge_e} and \e_ref{tifC_e},
\begin{equation}\label{fCjj'_e}
\wt{\fC}_{I,j}(d)=
\frac{(-1)^dd^{2d-1}}{\left(d!\right)^2}
\frac{1}{\left[u_j(I)\right]^{2d-1}}
\prod\limits_{r\in[N]-\{j,\wh{j}\}}\frac{\prod\limits_{s=1+dD_r(\ov{Ij})}^0
\left[u_r(I)\!-\!\frac{s}{d}u_j(I)\right]}
{\prod\limits_{s=1}^{dD_r(\ov{Ij})}
\left[u_r(I)\!-\!\frac{s}{d}u_j(I)\right]}.
\end{equation}
If $d\!\ge\!1$, $\bd^*\!\in\!\La$, $\bd'\!\equiv\!\bd^*\!-\!d\!\cdot\!\deg\ov{Ij}\!\in\!\La$, then, 
\begin{equation}\label{Y0rec_e2}
\Big[\bd^*\!\in\!\De_I^*\textnormal{ and }D_{j}(\bd^*)\!\ge\!d\Big]\quad
\Llra\quad\Big[\bd'\!\in\!\De^*_J\textnormal{ and }D_{\wh{j}}(\bd')\!\ge\!-d\Big]
\end{equation}
by \e_ref{Dedge_e}.
By  \e_ref{Yrestr_e}, \e_ref{cY_e}, \e_ref{Yrec_e2}, \e_ref{fCjj'_e}, and \e_ref{fC_e},
\begin{equation*}
\Res_{z=-\frac{u_j(I)}{d}}\left\{\frac{1}{\h\!-\!z}\LR{\cYpp(x(I),z,q)}_{q;\bd^*}\right\}
=\frac{\fCpp_{I,j}(d)}{\h\!+\!\frac{u_j(I)}{d}}
\LR{\cYpp\left(x(J),-\frac{u_j(I)}{d},q\right)}_{q;\bd^*-d\cdot\deg\ov{Ij}}
\end{equation*}
for all $\bd^*\!\in\!\La$.
Finally, viewing
$\frac{1}{\h-z}\LR{\cYpp(x(I),z,q)}_{q;\bd^*}$ as a rational function in $\h,z$, and $\al_j$
and using the Residue Theorem on $\P^1$, we obtain
\begin{equation*}\begin{split}
\sum_{d\ge 1}\!\!\!\!\sum_{\begin{subarray}{c}j\in [N]-I\\d\cdot\deg\ov{Ij}\preceq \bd^*\end{subarray}}\!\!
\frac{\fCpp_{I,j}(d)}{\h\!+\!\frac{u_j(I)}{d}}
\LR{\cYpp\left(x(J),-\frac{u_j(I)}{d},q\right)}_{q;\bd^*-d\cdot\deg\ov{Ij}}\!=\!\LR{\cYpp(x(I),\h,q)}_{q;\bd^*}
\\
-\Res_{z=0,\i}\left\{\frac{1}{\h\!-\!z}\LR{\cYpp\left(x(I),z,q\right)}_{q;\bd^*}\right\},
\end{split}\end{equation*}
where $\Res_{z=0,\i}\cF\!\equiv\!\Res_{z=0}\cF\!+\!\Res_{z=\i}\cF$. Since
$$\Res_{z=0,\i}\left\{\frac{1}{\h\!-\!z}\LR{\cYpp\left(x(I),z,q\right)}_{q;\bd^*}\right\}\!\in\!
\Q_{\al}[\h,\h^{-1}],$$
this concludes the proof. 
\end{proof}

\begin{lmm}\label{YMPC_lmm}
With $\cYp$ and $\cYpp$ defined by \e_ref{cY_e},
$(\cYp,\cYpp)$ satisfies the MPC.
\end{lmm}
We follow the idea of the proof of \cite[Proposition~6.2]{Gi_mirr} and begin with some preparations.
Let $\bd\!\in\!\!\bigcup\limits_{I\in\V}\De_I^*$, 
$$J\equiv J(\bd)\equiv\left\{j\in [N]:D_j(\bd)\ge 0\right\},\quad S\!\equiv\!|J|\!+\!\sum\limits_{j\in J} D_j(\bd).$$
Let $A$ be the $|J|\!\times\!S$ matrix
giving $\prod\limits_{j\in J}\P^{D_j(\bd)}$ as in \e_ref{prodproj_e}.
Denote
the coordinates of a point $y\!\in\!\C^S$ by 
$$\left(y_{j;0},y_{j;1},\ldots,y_{j;D_j(\bd)}\right)_{j\in J}.$$
The pair $(M_JA,\tau)$ is toric in the sense of Definition~\ref{toric_dfn}.
It satisfies \ref{S} in Definition~\ref{toric_dfn}, since
\begin{equation}\label{Y0MPC_e1}
\mathscr{V}^{\tau}_{M_JA}\!=\!\left\{\left((i_1;p_1),\!\ldots\!,(i_k;p_k)\right):\!
\{i_1,\!\ldots\!,i_k\}\!\in\!\V,\,\{i_1,\!\ldots\!,i_k\}\!\subseteq\!J, 0\!\le\!p_r\!\le\!D_{i_r}(\bd)~\forall\,r\!\in\![k]\right\}
\end{equation}
by the second statement in Lemma~\ref{moment_lmm}\ref{Mrestr_lmm}. 
We identify $\C^S$ with $\bigoplus\limits_{j\in J}H^0(\P^1,\cO_{\P^1}(D_j(\bd)))$ via
$$\left(y_{j;0},y_{j;1},\ldots,y_{j;D_j(\bd)}\right)_{j\in J}
\lra\left(\sum_{r=0}^{D_{j}(\bd)}y_{j;r}z_0^{D_j(\bd)-r}z_1^r\right)_{j\in J}$$
and set $X_{\bd}\equiv X^{\tau}_{M_{J}A}$.
The torus $\T^1\!\times\!\T^{|J|}$ acts on $\bigoplus\limits_{j\in J} H^0(\P^1,\cO_{\P^1}(D_j(\bd))$ by
\begin{equation}\label{Y0MPC_e0}
\left(\xi,\left(t_j\right)_{j\in J}\right)\!\cdot\!\left(P_j(z_0,z_1)\right)_{j\in J}\equiv\left(t_jP_j(z_0,\xi z_1)\right)_{j\in J},
\end{equation}
while the torus $\T^{|J|}$ acts on $\bigoplus\limits_{j\in J} H^0(\P^1,\cO_{\P^1}(D_j(\bd)))$ by restricting this action via $$\T^{|J|}\!\ni\!t\hookrightarrow(1,t)\!\in\!\T^1\!\times\!\T^{|J|};$$
these actions descend to actions on $X_{\bd}$.
\begin{lmm}\begin{enumerate}[label=(\emph{\alph*}),leftmargin=*]\label{bigfixed_lmm}
\item\label{bigfpt_lmm} The fixed points of the $\T^1\!\times\!\T^{|J|}$-action on $X_{\bd}$ are
\begin{alignat}{1}
\label{Ip_e} [I,\bp]&\!\equiv\!\left[\left(P_j(z_0,z_1)\right)_{j\in J}\right],
\intertext{where $I\!\in\!\V$, $I\!\subseteq\!J$, $\bp\!=\!\left(p_i\right)_{i\in I}\!\in\!\Z^k$,
$0\!\le\!p_i\!\le\!D_i(\bd)$ for all $i\!\in\!I$, and}
P_j(z_0,z_1)&\!\equiv\!\begin{cases}
z_0^{D_j(\bd)-p_j}z_1^{p_j},&\hbox{\textnormal{ if }}j\!\in\!I;\\
0,&\hbox{ \textnormal{otherwise}}.
\end{cases}\notag
\end{alignat}
\item\label{bigfpt_lmm2} Let $I\!\in\!\V$ and $\bp\!=\!(p_i)_{i\in I}\!\in\!\Z^k$.
Then 
\begin{equation*}\label{newfpt_e}
0\!\le\!p_i\!\le\!D_i(\bd)\quad\forall\,i\!\in\!I\qquad\Llra\qquad
\bp M_I^{-1},\bd\!-\!\bp M_I^{-1}\!\in\!\De^*_I.
\end{equation*}
\end{enumerate}
\end{lmm}
\begin{proof}
Let $[(P_j(z_0,z_1)_{j\in J})]$ be any fixed point of the $\T^1\!\times\!\T^{|J|}$-action on $X_{\bd}$ and 
$(\xi_0,\xi_1)\!\in\!\C^2$ be such that
$P_j(\xi_0,\xi_1)\!\neq\!0$ whenever $P_j\!\neq\!0$.
By Lemma~\ref{moment_lmm}\ref{charts_lmm2} and \e_ref{Y0MPC_e1}, $(P_j(\xi_0,\xi_1))_{j\in J}\!\in\!\wt{X}^{\tau}_{M_J}$.
Since $[(P_j(\xi_0,\xi_1))_{j\in J}]$ is a $\T^{|J|}$-fixed point in $X^{\tau}_{M_J}$,
there exists $I\!\in\!\V$ with $I\!\subseteq\!J$
such that $P_j\!\neq\!0$ if and only if $j\!\in\!I$; see Corollary~\ref{fixed_crl}\ref{fpt_crl}.
This concludes the proof of \ref{bigfpt_lmm}.
Part~\ref{bigfpt_lmm2} follows from \e_ref{DeI*_e} and the identity $(D_i(\br))_{i\in I}\!\equiv\!\br M_I$ for $\br\!=\!\bp M_I^{-1}$; see the second equation in \e_ref{U_e}.
\end{proof}

We consider the $\T^1\!\times\!\T^N\!$-action on $X_{\bd}$ obtained by composing the projection
$\T^1\!\times\!\T^N\!\lra\!\T^1\!\times\!\T^{|J|}$ induced by $J\!\hookrightarrow\![N]$ with the action \e_ref{Y0MPC_e0} of $\T^1\!\times\!\T^{|J|}$ on $X_{\bd}$.
We denote by $\wbE(T_{[I,\bp]}X_{\bd})$ the $\T^1\!\times\!\T^N\!$-equivariant Euler class
of $T_{[I,\bp]}X_{\bd}$ and by
$$\cdot(I,\bp):H^*_{\T^1\times\T^N}(X_{\bd})\!\lra\!H^*_{\T^1\times\T^N}$$
the restriction map induced by the inclusion $[I,\bp]\!\hookrightarrow\!X_{\bd}$, where $[I,\bp]$ is the 
$\T^1\!\times\!\T^N\!$-fixed point defined by \e_ref{Ip_e}.
Let $\h$ denote the weight of the standard action of $\T^1$ on $\C$.
\begin{lmm}\label{classes_lmm}
There exist
classes $(\bx_i)_{i\in[k]},(\bu_r)_{r\in [N]},(\bla^+_i)_{i\in[a]},(\bla^-_i)_{i\in[b]}\!\in\!H^*_{\T^1\times\T^N}(X_{\bd})$
such that\begin{alignat}{1}
\label{equivXd_e1} \bu_{r}&=\sum_{i=1}^km_{ir}\bx_i\!-\!\al_r\qquad\forall\,r\!\in\![N],
\intertext{and such that for all $(I,\bd')$ with $I\!\in\!\V$,
$I\!\subseteq\!J$, $\bd',\bd\!-\!\bd'\!\in\!\De_I^*$, and all $[I,\bp]$ as in \e_ref{Ip_e},}
\label{equivXd_e0} \left(\bx_1(I,\bd'M_I),\ldots,\bx_k(I,\bd'M_I)\right)&=
\left(x_1(I),\ldots,x_k(I)\right)\!+\!\h\,\bd,'\\
\label{equivXd_e3} \wbE(T_{[I,\bp]}X_{\bd})&\!=\!\!
\prod_{\begin{subarray}{c}j\in J-I\end{subarray}}\prod\limits_{\begin{subarray}{c}0\le s\le D_j(\bd)\end{subarray}}\!\!
\left[\bu_j(I,\bp)\!-\!s\,\h\right]\!\times\!\!
\prod_{\begin{subarray}{c}j\in I\end{subarray}}\prod\limits_{\begin{subarray}{c}0\le s\le D_j(\bd)\\
s\neq p_j\end{subarray}}\!\!\!\left[\bu_j(I,\bp)\!-\!s\,\h\right],\\
\label{equivXd_e4} \bla^{\pm}_i\left(I,\bd'M_I\right)&=\la_i^{\pm}(I)\!+\!\h\,L_i^{\pm}(\bd')\qquad\forall\,i\!\in\![a]\quad
(\forall\,i\!\in\![b]).
\end{alignat}
\end{lmm}
\begin{proof}
We define the
classes $\wt{x}_1,\ldots,\wt{x}_k$ and $u_{j;s}$ in $H^*_{\T^S}(X_{\bd})$ with $j\!\in\!J$ and
$0\!\le\!s\!\le\!D_j(\bd)$ 
by \e_ref{ebdle_e} with $(M,\tau)$ replaced by $(M_JA,\tau)$.
By \e_ref{u_e},
\begin{equation}\label{bigu_e}
u_{j;s}\!=\!\sum_{i=1}^km_{ij}\wt{x}_i-\al_{j;s},
\end{equation}
where $\al_{j;s}\!\equiv\!\pi_{j;s}^*c_1(\O_{\P^{\i}}(1))$
and $\pi_{j;s}\!:\!(\P^{\i})^S\!\lra\!\P^{\i}$ is the projection onto the $(j;s)$ component.
By Corollary~\ref{fixed_crl}\ref{fpt_crl}, \e_ref{Y0MPC_e1}, and \e_ref{tan_e},
the $\T^S\!$-fixed points in $X_{\bd}$ are the points $[I,\bp]$ and
\begin{equation}\label{bigtan_e}
\E^{\T^S}(T_{[I,\bp]}X_{\bd})\!=\!\prod\limits_{\begin{subarray}{c} j\in J-I\end{subarray}}
\prod\limits_{\begin{subarray}{c}0\le s\le D_j(\bd)\end{subarray}}
\left[u_{j;s}\big|_{[I,\bp]}\right]\times
\prod\limits_{\begin{subarray}{c} j\in I\end{subarray}}\prod\limits_{\begin{subarray}{c}
0\le s\le D_j(\bd)\\s\neq p_j\end{subarray}}
\left[u_{j;s}\big|_{[I,\bp]}\right],
\end{equation}
where $\E^{\T^S}(T_{[I,\bp]}X_{\bd})$ denotes the $\T^S\!$-equivariant Euler class of $T_{[I,\bp]}X_{\bd}$ and
$$\big|_{[I,\bp]}:H^*_{\T^S}(X_{\bd})\lra H^*_{\T^S}$$ the restriction homomorphism
induced by $[I,\bp]\!\hookrightarrow\!X_{\bd}$.
The map 
\begin{equation*}
F\!:\!(\C^{\i}\!-\!\{0\})^{N+1}\!\!\lra(\C^{\i}\!-\!\{0\})^S\!,\quad
F(e_0,e_1,\ldots,e_N)\!\equiv\!\left(e_j,e_j\cdot e_0,e_j\!\cdot\!e_0^2,\ldots,e_j\!\cdot\!e_0^{D_j(\bd)}\right)_{j\in J},
\end{equation*}
where
\begin{equation*}\begin{split}
(z_1,z_2,\ldots)^d&\equiv(z_1^d,z_2^d,\ldots)\qquad\forall\,d\!\ge\!1,\, (z_1,z_2,\ldots)\!\in\!\C^{\i}\!-\!\{0\},\quad\textnormal{and}\\
(z_1,z_2,\ldots)\!\cdot\!(y_1,y_2,\ldots)&\equiv(z_iy_j)_{(i,j)\in\Z^{>0}\times\Z^{>0}}\qquad\forall\,
(z_1,z_2,\ldots),(y_1,y_2,\ldots)\!\in\!\C^{\i}\!-\!\{0\}
\end{split}\end{equation*}
is equivariant with respect to the homomorphism
\begin{equation*}\begin{split}
f\!:\!\T^1\!\times\!\T^N\!\lra\!\T^S,\quad
f(\xi,t_1,\!\ldots\!,t_N)\equiv
\left(t_j,t_j\xi,t_j\xi^2,\ldots,t_j\xi^{D_j(\bd)}\right)_{j\in J}.
\end{split}\end{equation*}
It induces a map
$\ov{F}\!:\!(\C^{\i}\!-\!\{0\})^{N+1}\!\times_{\T^1\times\T^N}\!X_{\bd}\!\lra\!(\C^{\i}\!-\!\{0\})^S\!\times_{\T^S}\!X_{\bd}$,
\begin{equation*}\begin{split}
\ov{F}\left[e_0,e_1,\ldots,e_N,[(P_j)_{j\in J}]\right]\equiv&
\left[F(e_0,e_1,\ldots,e_N),[(P_j)_{j\in J}]\right]\\
&\forall\,(e_0,e_1,\ldots,e_N)\!\in\!(\C^{\i}\!-\!0)^{N+1},\,
[(P_j)_{j\in J}]\!\in\!X_{\bd},
\end{split}\end{equation*} 
and thus a homomorphism $\ov F^*\!:\!H^*_{\T^S}(X_{\bd})\!\lra\!H^*_{\T^1\times\T^N}(X_{\bd})$.
It follows that
\begin{equation}\label{bigal_e}
\ov F^*\al_{j;s}=\al_{j}+s\h\qquad\forall\,(j;s)\quad\textnormal{with}\quad j\!\in\!J,\,0\!\le\!s\!\le\!D_j(\bd).
\end{equation}
We define $\bx_i$ and $\bu_r$ as the $\T^1\!\times\!\T^N\!$-equivariant Euler classes of the line bundles
$$\wt{X}^{\tau}_{M_JA}\times\C/\sim_i\lra X_{\bd}\quad\textnormal{and}\quad\wt{X}^{\tau}_{M_JA}\times\C/\sim_r\lra X_{\bd},$$
where
\begin{equation}\begin{split}\label{defboldx_e}
\left((P_j)_{j\in J},c\right)&\sim_i\left((t^{M_j}P_j)_{j\in J},t_ic\right)\\
\left((P_j)_{j\in J},c\right)&\sim_r\left((t^{M_j}P_j)_{j\in J},t^{M_r}c\right)
\end{split}\qquad
\forall\,t\!\in\!\T^k,((P_j)_{j\in J},c)\!\in\!\wt{X}^{\tau}_{M_JA}\!\times\!\C\end{equation}
with respect to the lifts of the $\T^1\!\times\!\T^N\!$-action on $X_{\bd}$ given by
\begin{equation}\begin{split}\label{boldlift_e}
(\xi,t_1,\ldots,t_N)\!\cdot\!\left[(P_j(z_0,z_1))_{j\in J},c\right]&\equiv\left[(t_jP_j(z_0,\xi z_1))_{j\in J},c\right]\quad
\textnormal{and}\\
(\xi,t_1,\ldots,t_N)\!\cdot\!\left[(P_j(z_0,z_1))_{j\in J},c\right]&\equiv\left[(t_jP_j(z_0,\xi z_1))_{j\in J},t_rc\right]
\end{split}\end{equation}
respectively.
It follows that 
\begin{equation}\label{boldx_e}
\bx_i=\ov F^*\wt{x}_i
\end{equation}
and $\bu_j$ satisfy \e_ref{equivXd_e1}. The latter follows
similarly to the proof of \e_ref{u_e} using equations analogous to \e_ref{j_e} and \e_ref{jL0_e} with $\T^N$ replaced by $\T^1\!\times\!\T^N$.
Equation \e_ref{equivXd_e0} follows from \e_ref{boldx_e}, Proposition~\ref{equivcoh_prp}\ref{erestr_prp}, and \e_ref{bigal_e}.
Equation \e_ref{equivXd_e2gen} follows from \e_ref{equivXd_e1}, \e_ref{equivXd_e0}, and \e_ref{u_e}.
Equation 
\e_ref{equivXd_e3} follows from \e_ref{bigtan_e} together with \e_ref{bigu_e},
\e_ref{bigal_e}, \e_ref{boldx_e}, and \e_ref{equivXd_e1}.
Finally, define 
\begin{equation}\label{boldla_e}
\bla_i^+\equiv\sum_{r=1}^k\ell^+_{ri}\bx_r\quad\textnormal{and}\quad\bla_i^-\equiv\sum_{r=1}^k\ell^-_{ri}\bx_r,
\end{equation}
with $\ell^+_{ri},\ell^-_{ri}$ as in \e_ref{Etensor_e}.
Equations \e_ref{equivXd_e4} then follow from \e_ref{boldla_e}, \e_ref{equivXd_e0}, and \e_ref{la_e}.
\end{proof}
With $\bu_j$ and $(I,\bd')$ as in Lemma~\ref{classes_lmm},
\begin{equation}\label{equivXd_e2gen} 
\bu_j\left(I,\bd'M_I\right)=u_j(I)\!+\!\h\,D_j(\bd')\qquad\forall\,j\!\in\![N],
\end{equation}
by \e_ref{equivXd_e1}, \e_ref{equivXd_e0}, and \e_ref{u_e}.
\begin{lmm}\label{bigvbdle_lmm}
There exists a vector bundle $V_{\bd}\!\lra\!X_{\bd}$ and a lift of the $\T^1\!\times\!\T^N\!$-action
to $V_{\bd}$ such that the $\T^1\!\times\!\T^N\!$-equivariant Euler class $\wbE(V_{\bd})$ satisfies
$$\wbE(V_{\bd})(I,\bp)=\prod\limits_{j=1}^N
\prod\limits_{s=1}^{-D_j(\bd)-1}\left[\bu_j(I,\bp)+s\,\h\right]$$
for all $\T^1\!\times\!\T^N\!$-fixed points $[I,\bp]$ defined by \e_ref{Ip_e} and with $\bu_j\!\in\!H^*_{\T^1\times\T^N}(X_{\bd})$ as in Lemma~\ref{classes_lmm}.
\end{lmm}
\begin{proof}
Let 
\begin{equation*}\begin{split}
\wt{V}_{\bd}&\equiv\left\{\left(P_j\right)_{j\in[N]-J}
\!\in\!\!\!\!\bigoplus\limits_{\begin{subarray}{c}j\in [N]-J\end{subarray}}\!\!\!H^0\!\left(\P^1,\cO_{\P^1}(-\!D_j(\bd)\!-\!1)\right)\!:P_j(1,0)\!=\!0\quad
\forall\,j\!\in\![N]\!-\!J\right\},\\
 V_{\bd}&\equiv\frac{\wt{X}^{\tau}_{M_JA}\!\times\!\wt{V}_{\bd}}{\sim}\!\lra\!X_{\bd},\quad
\left(\left(P_j\right)_{j\in J},\left(P_j\right)_{j\in [N]-J}\right)
\!\sim\!\left(\left(t^{M_j}P_j\right)_{j\in J},\left(t^{M_j}P_j\right)_{j\in [N]-J}\right)\quad
\forall\,t\!\in\T^k.
\end{split}\end{equation*} 
Since $\wt{X}^{\tau}_{M_JA}\!\lra\!X_{\bd}$ is a principal bundle,
$V_{\bd}\!\lra\!X_{\bd}$ is a holomorphic vector bundle.
The $\T^1\!\times\!\T^N$\!-action on $X_{\bd}$ lifts to $V_{\bd}$ via
\begin{equation*}\begin{split}
\left(\xi,t_1,\ldots,t_N\right)\!\cdot\!\left[\left(P_j(z_0,z_1)\right)_{j\in J},\left(P_j(z_0,z_1)\right)_{j\in [N]-J}\right]
\equiv\left[\left(t_jP_j(z_0,\xi z_1)\right)_{j\in J},\left(t_jP_j(z_0,\xi z_1)\right)_{j\in [N]-J}\right].
\end{split}\end{equation*}
The lemma now follows from the definition of $\bu_j$ in \e_ref{defboldx_e} and \e_ref{boldlift_e}.
\end{proof}
By the Localization Theorem~\e_ref{ABo_e}, Lemma~\ref{bigfixed_lmm}, Lemma~\ref{bigvbdle_lmm}, and \e_ref{equivXd_e3},
\begin{equation}\label{Xdint_e}
\int_{X_{\bd}}\!\!\!\!\!f\,\wbE(V_{\bd})\!=\!\!\!\!\!\!\!
\sum_{\begin{subarray}{c}I\in\V\\
\bd',\bd-\bd'\in\De_I^*\end{subarray}}\!
\frac{f(I,\bd'M_I)\prod\limits_{j=1}^N
\prod\limits_{s=D_j(\bd)+1}^{-1}\left[\bu_j(I,\bd'M_I)\!-\!s\h\right]}
{\prod\limits_{\begin{subarray}{c}j\in J- I\end{subarray}}\prod\limits_{\begin{subarray}{c}0\le s\le D_j(\bd)\end{subarray}}\!\!\!\!\!
\left[\bu_j(I,\bd'M_I)\!-\!s\h\right]\!
\prod\limits_{\begin{subarray}{c}j\in I\end{subarray}}\prod\limits_{\begin{subarray}{c}0\le s\le D_j(\bd)\\
s\neq D_j(\bd')\end{subarray}}\!\!\!\!\!\left[\bu_j(I,\bd'M_I)\!-\!s \h\right]},
\end{equation}
for all $f\!\in\!H^*_{\T^1\times\T^N}(X_{\bd})$.
\begin{proof}[Proof of Lemma~\ref{YMPC_lmm}]
By Definition~\ref{MPC_dfn}, \e_ref{Yrestr_e}, \e_ref{cY_e}, \e_ref{Xdint_e}, \e_ref{equivXd_e0}, \e_ref{equivXd_e2gen},
and \e_ref{equivXd_e4},
\begin{align*}
\Phi_{\cYp,\cYpp}\left(\h,z,Q\right)
\!&=\!\!
\sum_{I\in\V}\!\sum_{\bd\in\De_I^*}\!Q^{\bd}
\Bigg\{\sum_{\begin{subarray}{c}\bd',\bd''\in\De^*_I\\\bd'+\bd''=\bd\end{subarray}}
\frac{\ne^{\left(x(I)+\h \bd'\right)\cdot z}}{\prod\limits_{j\in[N]-I}u_j(I)}
\frac{\prod\limits_{\begin{subarray}{c}j\in [N]\\D_j(\bd)<0\end{subarray}}u_j(I)\prod\limits_{\begin{subarray}{c}j\in [N]\\D_j(\bd)<0\end{subarray}}\prod\limits_{s=D_j(\bd')+1}^{-D_j(\bd'')-1}\left[u_j(I)\!+\!s\h\right]}
{\prod\limits_{\begin{subarray}{c}D_j(\bd)\ge 0\end{subarray}}\prod\limits_{\begin{subarray}{c}-D_j(\bd'')\le s\le D_j(\bd')\\s\neq 0\end{subarray}}\!
\left[u_j(I)\!+\!s\h\right]}\\
&\qquad\qquad\qquad\qquad\quad\times
\prod\limits_{i=1}^a\!\prod\limits_{s=-L_i^+(\bd'')+1}^{L_i^+(\bd')}\left[\la^+_i(I)\!+\!s\h\right]
\prod\limits_{i=1}^b\!\prod\limits_{s=L_i^-(\bd')+1}^{-L_i^-(\bd'')}\left[\la_i^-(I)+s\h\right]
\Bigg\}\\
&=\!\!\!
\sum_{\bd\in\bigcup\limits_{I\in\V}\De_I^*}\!\!\!Q^{\bd}
\int_{X_{\bd}}\!\!\wbE(V_{\bd})
\ne^{\bx\cdot z}\prod\limits_{i=1}^a\!\prod\limits_{s=-L_i^+(\bd)+1}^{0}\!\!\left[\bla^+_i\!+\!s\h\right]
\prod\limits_{i=1}^b\prod\limits_{s=1}^{-L_i^-(\bd)}\left[\bla_i^-\!+\!s\h\right].
\end{align*}
The last expression is in $\Q[\al,\h][[z,\La]]$.
\end{proof}
\appendix
\section{Derivation of \e_ref{Z'Y'_e} from \cite{LLY3}}
\label{LLY_a}
\newcommand\ST{\rule[-0.5em]{0pt}{1.5em}}
$$\begin{array}{|c|c|}\hline
\cite{LLY3}& \textnormal{our notation}\\\hlinewd{2.3pt}
\ST m&k\\\hline
\ST e^{t_j}&q_j\\\hline
\ST \cR&\C(\al_1,\ldots,\al_N)[\h]\\\hline
\ST \C[\cT^*]&\Q[\al_1,\ldots,\al_N]\\\hline
\ST \al&\h\\\hline
\ST T&\T^N\\\hline
\ST e_T&\E\\\hline
\ST c_1(L_d)&-\psi_1\in H^2\left(\ov\M_{0,1}(\X,\bd)\right)\\\hline
\ST \rho&\textnormal{forgetful morphism } \ov\M_{0,1}(X,\bd)\lra\ov\M_{0,0}(X,\bd)\\\hline
\ST e^X_d&\ev_1:\ov\M_{0,1}(X,\bd)\lra X\\\hline
\ST LT_{0,1}(d,X)&\left[\ov\M_{0,1}(X,\bd)\right]^{vir}\\\hline
\ST V_d&\cV_E\lra\ov\M_{0,0}(X,\bd)\\\hline
\ST U_d=\rho^*V_d&\cV_E\lra\ov\M_{0,1}(X,\bd)\\\hline
\ST (D_a)_{a\in [N]}&(u_j)_{j\in [N]}\\\hline
\end{array}$$

In \cite[Section~3.2]{LLY3}, we take $b_T\!\equiv\!e_T$ (that is, $\E$), $X\!=\!\X$, and
$V\!\equiv\!E$.
By \cite[Section~3.2]{LLY3},
$$A^{V,b_T}(t)\!=\!A(t)\!=\!e^{-H\cdot t/\al}\left[\frac{\E(E^+)}{\E(E^-)}\!+\!\sum_{d\in\La-0}A_d e^{d\cdot t}\right],\quad
A_d\!=\!\ne^X_*\left(\frac{\rho^*b_T(V_d)\cap LT_{0,1}(d,X)}{e_G(F_0/M_d(X))}\right),$$
where $\{H_a\}\!\subset\!H^2_{\T^N}(\X)$
is a basis whose restriction to $H^2(\X)$ is a basis of first Chern classes
of ample line bundles; see \cite[Section~3,viii]{LLY3}.
By \cite[Lemma~3.5]{LLY3}, 
$$e_G(F_0/M_d(X))=\al(\al\!-\!c_1(L_d)).$$
Thus, in our notation,
\begin{equation}\label{LLYA_e}
A(t)=\ne^{-H\cdot t/\h}
\left\{\frac{\E(E^+)}{\E(E^-)}\!+\!\sum_{\bd\in\La-0}\ne^{\bd\cdot t}\ev_{1*}\left[\frac{\E(\cV_E)}{\h(\h\!+\!\psi_1)}\right]\right\},
\end{equation}
where $\ev_1\!:\!\ov\M_{0,1}(\X,\bd)\!\lra\!\X$ is the evaluation map at the marked point.
By \e_ref{eZeta1ptdfn_e}, \e_ref{LLYA_e}, and 
the string relation~\cite[Section~26.3]{MirSym}, 
\begin{equation}\label{LLYA_e2}
A(t)= \ne^{-H\cdot t/\h}\frac{\E(E^+)}{\E(E^-)}\cZp_1(-\h,e^t).
\end{equation}

By \e_ref{cY_e}, Remark~\ref{cYcoeff_rmk}, and \e_ref{la_e},
\e_ref{Z'Y'_e} is independent of the choice of a $\Q[\al]$-basis for $H^2_{\T^N}(\X)$ and so it is not necessary to assume that the restrictions of $x_i$
to $H^2(\X)$ are Chern classes of ample line bundles. Thus, we may take $H\!=\!(x_1,\ldots,x_k)$ in \cite{LLY3}.
By \cite[(5.2)]{LLY3} and \cite[Theorem~4.9]{LLY3}, 
\begin{equation}\label{LLYB_e}
B(t)=\ne^{-H\cdot t/\h}\frac{\E(E^+)}{\E(E^-)}\cYp\left(x,-\h,\ne^t\right)
\end{equation}
in \cite[Theorem~4.7]{LLY3}.
In the notation of the proof of \cite[Theorem~4.7]{LLY3} correlated with Remark~\ref{cYcoeff_rmk},
\begin{equation}\begin{split}\label{LLYcoeffs_e}
C\!&=\!\Ip_{\b0}(q),\quad C'\!=\!-\Ip_{\b0}(q)\left(G(q)\!+\!\sum_{j=1}^N\al_jg_j(q)\right),\quad
C''\!=\!-\Ip_{\b0}(q)\cdot(f_1(q),\ldots,f_k(q)),\\
\ne^{f/\al}\!&=\!\ne^{-\log C-\frac{C'}{C\al}}\!=
\!\frac{1}{\Ip_{\b0}(q)}\ne^{\frac{1}{\h}\left[G(q)+\sum\limits_{j=1}^N\al_jg_j(q)\right]},\quad
g\!=\!-\frac{C''}{C}\!=\!(f_1(q),\ldots,f_k(q)).
\end{split}\end{equation}

Finally, by \cite[Section~5.2]{LLY3} and \cite[Corollary~4.11]{LLY3}, the hypothesis of \cite[Theorem~4.7]{LLY3} are satisfied
with $A(t)$ and $B(t)$ as in \e_ref{LLYA_e2} and \e_ref{LLYB_e} if $\nu_E(\bd)\!\ge\!0$ for all $\bd\!\in\!\La$,
since $\E(E^+)$ and $\E(E^-)$ are non-zero whenever restricted to any $\T^N$-fixed point;
see Proposition~\ref{equivcoh_prp}\ref{erestr_prp}.
Thus, \e_ref{Z'Y'_e} follows from \cite[Theorem~4.7]{LLY3}, \e_ref{LLYA_e2}, \e_ref{LLYB_e}, and \e_ref{LLYcoeffs_e}.

\vspace{3mm}
\noindent
{\it Department of Mathematics, SUNY Stony Brook, NY 11794-3651\\
alexandra@math.sunysb.edu}


\begin{thebibliography}{99}

\bibitem[ABo]{ABo} M.~Atiyah and R.~Bott,
{\it The moment map and equivariant cohomology}, Topology 23 (1984), no.~1, 1--28.

\bibitem[At]{At} M.~F.~Atiyah,
{\it Convexity and commuting hamiltonians}, Bull. London Math. Soc. 14 (1982), no. 1, 1--15.

\bibitem[Au]{Au} M.~Audin,
{\it Torus Actions on Symplectic Manifolds}, Progress in Mathematics,~93, Birkh\"{a}user, Basel, 2004.

\bibitem[Ba]{Ba} V.~Batyrev,
{\it Quantum cohomology ring of toric manifolds}, 
Journ\'{e}es de G\'{e}om\'{e}trie Alg\'{e}brique d'Orsay (Orsay, 1992),
Ast\'{e}risque no.~218 (1993), 9--34.


\bibitem[BK]{BK} A.~Bertram and H.~Kley,
{\it New recursions for genus-zero Gromov-Witten invariants}, 
Topology~44 (2005), no.~1, 1--24.

\bibitem[Br]{Br} A.~Br{\o}ndsted,
{\it An Introduction to Convex Polytopes}, GTM~90,
Springer-Verlag, New York-Berlin, 1983.

\bibitem[Ca]{Ca} A. Cannas da Silva,
{\it Symplectic Toric Manifolds}, 
Part~B of 
{\it Symplectic Geometry of Integrable Hamiltonian Systems},
Advanced Courses in Mathematics - CRM Barcelona, Birkh\"{a}user, Basel, 2003.

\bibitem[Ch]{Ch} L~Cherveny,
{\it Genus-zero mirror principle for two marked points}, math/1001.0242v1.

\bibitem[CK]{CK} D.~A.~Cox and S.~Katz,
{\it Mirror Symmetry and Algebraic Geometry}, Mathematical Surveys and Monographs,~68,
AMS, 1999.

\bibitem[De]{De} T.~Delzant,
{\it Hamiltoniens p\'{e}riodiques et images convexes de l'application moment},
Bulletin de la Soci\'{e}t\'{e} Math\'{e}matique de France 116 (1988), no.~3, 315--339.

\bibitem[El]{El} A.~Elezi,
{\it Virtual class of zero loci and mirror theorems},
Adv. Theor. Math. Phys.~7 (2004), 1103-1115.

\bibitem[Gi1]{Gi_equiv} A.~Givental,
{\it Equivariant Gromov-Witten invariants},
IMRN, 1996, no.~13, 613--663.

\bibitem[Gi2]{Gi_mirr} A.~Givental,
{\it A mirror theorem for toric complete intersections}, 
Topological Field Theory, Primitive Forms and Related Topics (Kyoto, 1996), 141--175,
Progr. Math., 160, Birkh\"{a}user Boston, 1998.

\bibitem[Gi3]{Gi_fpt} A.~Givental,
{\it A symplectic fixed point theorem for toric manifolds},
The Floer Memorial Volume, 445--481, Progr. Math., 133, Birkh\"{a}user, Basel, 1995.

\bibitem[GraPa]{GraPa} T.~Graber and R.~Pandharipande,
{\it Localization of virtual classes},
Invent. Math.~135 (1999), no.~2, 487--518.

\bibitem[GriH]{GH} P.~Griffiths and J.~Harris,
{\it Principles of Algebraic Geometry},
Wiley, 1994.

\bibitem[GuS]{GS1} V.~Guillemin and S.~Sternberg,
{\it Convexity properties of the moment mapping},
Invent. Math. 67 (1982), no.~3, 491--513.


\bibitem[Ki]{Ki} F.~C.~Kirwan,
{\it Cohomology of Quotients in Symplectic and Algebraic Geometry},
Mathematical Notes~31, Princeton University Press, 1984.

\bibitem[KlPa]{KlPa} A.~Klemm and R.~Pandharipande,
{\it Enumerative geometry of Calabi-Yau $4$-folds},
Comm. Math. Phys.~281 (2008), no.~3, 621--653.

\bibitem[La]{La} R.~Lazarsfeld,
{\it Positivity in Algebraic Geometry I}, 48,
Springer, 2004.

\bibitem[LLY1]{LLY1} B.~Lian, K.~Liu, and S.T.~Yau,
{\it Mirror principle I}, Asian J. Math. 1(1997), no.~4, 729--763.

\bibitem[LLY3]{LLY3} B.~Lian, K.~Liu, and S.T.~Yau, 
{\it Mirror principle III}, Asian J.~Math.~3(1999), no.~4, 771--800.

\bibitem[McDSa]{McDSa} D.~McDuff and D.~Salamon,
{\it J-holomorphic Curves and Symplectic Topology}, 
AMS, 2004.

\bibitem[MW]{MW} J.~Marsden and A.~Weinstein,
{\it Reduction of symplectic manifolds with symmetry},
Reports on Mathematical Physics~5 (1974), no.~1, 121--130.

\bibitem[MirSym]{MirSym} K.~Hori, S.~Katz, A.~Klemm, R.~Pandharipande,
R.~Thomas, C.~Vafa, R.~Vakil, and E.~Zaslow, {\it Mirror Symmetry},
Clay Math.\ Inst., AMS, 2003.

\bibitem[Mu]{Mu} J.~Munkres, {\it Topology}, Prentice Hall, 2000.

\bibitem[PaZ]{PaZ} R.~Pandharipande and A.~Zinger, 
{\it Enumerative geometry of Calabi-Yau 5-folds,
New Developments in Algebraic Geometry, Integrable Systems and Mirror Symmetry},
Adv. Stud. Pure Math., 59, Math. Soc. Japan, Tokyo, (2010), 239--288. 

\bibitem[Po]{bcov1_ci} A.~Popa,
{\it The genus one Gromov-Witten invariants of Calabi-Yau complete intersections},
Trans. AMS 365 (2013), no. 3, 1149--1181.

\bibitem[PoZ]{bcov0_ci} A.~Popa and A.~Zinger,
{\it Mirror symmetry for closed, open, and unoriented Gromov-Witten invariants}, math/1010.1946.

\bibitem[Spa]{Spa} E.~Spanier,
{\it Algebraic Topology}, Springer, 1991.

\bibitem[Sp]{Sp} H.~Spielberg,
{\it The Gromov-Witten invariants of symplectic toric manifolds}, math/0006156v1.

\bibitem[tD]{tD} T.~tom~Dieck,
{\it Transformation Groups}, Walter de Gruyter, 1987.

\bibitem[Zi]{Zi} G.~M.~Ziegler,
{\it Lectures on Polytopes}, GTM~152,
Springer-Verlag, 1995.


\bibitem[Z1]{CoZ} A.~Zinger,
{\it Double and triple Givental's J-functions for stable quotients invariants},
math/1305.2142.


\bibitem[Z2]{bcov0} A.~Zinger,
{\it Genus-zero two-point hyperplane integrals in the Gromov-Witten theory},
Communications in Analysis and Geometry~17 (2010), no.~5, 1--45. 

\bibitem[Z3]{bcov1} A.~Zinger,
{\it The reduced genus 1 Gromov-Witten invariants of Calabi-Yau hypersurfaces},
J. AMS~22 (2009), no.~3, 691--737. 

\bibitem[Z4]{multipt} A.~Zinger,
{\it The genus~0 Gromov-Witten invariants of projective complete intersections},
math/1106.1633v2.

\end{thebibliography}
\end{document}